%% file: assignment_paper.tex
\documentclass{article}

\usepackage{setspace}

\input{include/header}
\input{include/symbols}

\allowdisplaybreaks

\title{Spectral analysis of the logit mapping and implications for stochastic user equilibrium algorithms}

\author[a]{Debojjal Bagchi\thanks{Corresponding author: Email: debojjalb@utexas.edu; Mobile: (737) 304-9684\\
Email addresses: debojjalb@utexas.edu (Debojjal Bagchi), sboyles@austin.utexas.edu (Stephen D. Boyles)}}
\author[a]{Stephen D. Boyles}

\affil[a]{\small Fariborz Maseeh Department of Civil, Architectural and Environmental Engineering, The University of Texas at Austin, 301 E. Dean Keeton St., Austin, TX, 78712, USA}
\date{ }
\begin{document}
 \maketitle
\vspace{-30pt}
\begin{abstract}
We analyze the Jacobian of the logit mapping for stochastic user equilibrium (SUE) and use it to develop two improved algorithms for path-based SUE. We show that the Jacobian decomposes into two matrices: one that annihilates differences of feasible path flow vectors, and another whose eigenvalues are all non-positive reals, provided link costs are monotone non-decreasing and separable.
Using these properties, we first show that the method of successive averages (MSA) with a small constant step-size $s$ converges linearly at a rate $1-s$, with the largest admissible step-size depending on the eigenvalues of the Jacobian of the logit mapping. Building on this result, we develop an adaptive constant step-size rule that retains the global convergence of MSA while achieving asymptotic linear convergence.
Our second algorithm is a Newton-based method using a reformulation of SUE as a root-finding problem. Unlike gradient-projection approaches that operate on the Hessian of the SUE objective function (a dense matrix), our method exploits the structure of the Jacobian of the logit mapping, making computations tractable and removing the need for manifold optimization. Numerical experiments show superlinear convergence on most tested networks, with our methods outperforming existing approaches on large networks or when demand is high. To our knowledge, this article is the first to report runtimes for logit-based SUE on networks as large as Chicago Regional and Philadelphia, providing a benchmark for future algorithmic development.
\vspace{3mm}
\end{abstract}

\section{Introduction}\label{sec:intro}

Many iterative methods for solving traffic assignment problems operate by computing a search direction from a feasible solution toward a ``target'' solution, and then take a step along this direction of a given size. In stochastic user equilibrium (SUE) with the logit path choice model, the target solution can be obtained from the current path flow solution through a continuous mapping, reviewed below, which we refer to as the ``logit mapping.'' In this paper, we present two algorithmic advances for path-based approaches to find SUE, both built on a spectral analysis of the Jacobian of the logit mapping. First, we design a step-size rule under which the method of successive averages (MSA) achieves asymptotic linear convergence while retaining global convergence. Second, we design a search direction that achieves quadratic convergence in a neighborhood of SUE. This search direction can be paired with any globally convergent algorithm and used adaptively once the path flow solution enters the neighborhood of quadratic convergence. Our numerical results show that these insights lead to very efficient algorithms to solve logit SUE problems.

MSA with a harmonic step-size sequence is a widely used algorithm for traffic assignment problems. Although MSA is globally convergent and popular due to its simplicity, it is known to suffer from slow, sublinear convergence. However, \cite{bargeraboyce06} showed that MSA can converge much faster when used with a constant step-size, provided the target solution varies smoothly in the current solution. Specifically, they showed that MSA with a small constant step-size converges asymptotically at a linear rate, with the rate depending on the largest eigenvalue of the Jacobian of the target mapping at the fixed point. The logit mapping is smooth, so this result should in principle hold for logit-based SUE. Yet recent studies rarely use a constant step-size in MSA for logit-based SUE. The main reason is a lack of clear guidance on when to switch to a constant step-size, and on how small that step-size must be to guarantee convergence~\citep{liu2009method}. We show that both these questions can be answered through a spectral analysis of the Jacobian of the logit mapping.

To this end, we compute the Jacobian of the logit mapping and show that it can be decomposed into a sum of two matrices. We will show that all of the eigenvalues of one of these matrices are real and non-positive, with the largest eigenvalue exactly zero and the smallest bounded from below, and that adding the other matrix has no impact on the convergence rate of MSA. Exploiting these properties, we prove that MSA converges linearly near equilibrium at a rate of $1-s$ for a small step-size $s$, independent of the eigenvalues of the Jacobian. This result complements \cite{bargeraboyce06}, who observed empirically that the rate of convergence near equilibrium was approximately $1-s$ in a different problem setting, suggesting that the specific values of the eigenvalues might not play a role in their specific problem. Our results show that this is exactly the case in the logit-based SUE setting. Furthermore, we derive an upper bound on the constant step-size that guarantees linear convergence close to equilibrium. Using these results, we then construct a step-size rule that steers any feasible starting solution into a region where linear convergence holds, while satisfying the conditions required for global convergence. We verify empirically that MSA with our proposed step-size rule converges asymptotically at rate $1-s$ on several real-world networks obtained from \cite{transportationNetworks}.

Although the proposed step-size rule achieves asymptotic linear convergence, the rate $1-s$ can become close to one for small $s$. Since a small step-size is required to ensure convergence, the resulting linear rate can itself be slow in practice. Faster path-based approaches for SUE, such as gradient projection, exploit second-order information through the Hessian of the SUE objective. However, as \cite{bekhor2005investigating} note, the Hessian of Fisk's SUE objective \citep{fisk1980some} is dense and impractical to invert, and its diagonal provides a poor descent direction. As a consequence, recent literature on second-order methods for SUE relies only on partial second-order information, such as diagonal approximations of the Hessian \citep{bekhor2005investigating} or secant approximations via Barzilai-Borwein (BB) step-sizes \citep{du2021faster}.

We will address these difficulties by reformulating SUE as a root-finding problem on the fixed-point system of equations and applying Newton's method to this system, exploiting the properties of the Jacobian of the logit mapping. Rather than working with the Hessian of the SUE objective, we update the current solution at each iteration by adding a ``step'' obtained from solving a linear system, which we refer to as the full Newton system. Our spectral analysis, however, reveals that this Newton system is singular and has infinitely many solutions. To resolve this issue, we will construct a reduced Newton system that provides the unique demand-feasible solution among all solutions of the full Newton system. We then show that close to equilibrium, taking the ``step'' from the reduced Newton system satisfies both demand conservation and non-negativity of path flows. As a result, manifold optimization is not needed near equilibrium with this approach, in contrast to most second-order path-based methods for SUE \citep{bekhor2005investigating}. To solve the reduced Newton system efficiently, we present an inexact Newton approach using a Krylov subspace method, and show that it also preserves feasibility and achieves quadratic convergence close to equilibrium.

We test the computational performance of our Newton-based approach and benchmark it against a recent quasi-Newton method using BB step-sizes \citep{du2021faster} on several real-world networks from \cite{transportationNetworks}. Since the Newton-based procedure is only locally convergent, we pair it with BB step-sizes and present a procedure to switch adaptively to Newton steps near equilibrium. We find that our Newton-based approach performs either similarly to or better than \cite{du2021faster}. Since \cite{du2021faster} presents one of the fastest known approaches for logit-based SUE, we conclude that our method is at the very least competitive with, and in some cases significantly faster than extant methods. Furthermore, our Newton-based method achieves tight convergence on networks as large as Chicago Regional and Philadelphia, even with inflated demand, within a modest computational budget. To the best of our knowledge, no prior work has reported logit-based SUE assignment results on instances of this size, and hence our results provide a benchmark for future algorithms to compare against.

The remainder of the paper is organized as follows. Section \ref{sec:literature} reviews relevant literature for SUE. Section \ref{sec:background} presents the logit route choice model and key mathematical properties of SUE required for developing the results in this paper. Section \ref{sec:3_Jacobian} computes the Jacobian of the logit mapping and summarizes its spectral properties. Section~\ref{sec:bound} proves the asymptotic linear convergence of MSA close to equilibrium and presents a step-size rule that is globally convergent while achieving asymptotic linear convergence. Section~\ref{sec:newton} presents an asymptotically quadratically convergent formulation for logit-based SUE. Finally, Section \ref{sec:results} presents the computational results, and Section \ref{sec:conc} concludes with recommendations and directions for future work.

\section{Literature review}\label{sec:literature}

Traffic assignment problems predict the routes travelers are likely to choose based on congestion caused by traffic flow in the network. These problems involve distributing the demand for every origin-destination (OD) pair in a network among a set of feasible paths. Traffic assignment problems can be broadly categorized into static and dynamic models, depending on whether they account for temporal variations in network conditions. A detailed discussion of these models can be found in \cite{blubook-vol1-v10}. Static and dynamic traffic assignment are both useful tools: dynamic assignment models offer greater traffic realism, whereas static models are faster to solve, have mathematical regularity properties~\citep{iryo2013properties}, and are more stable with respect to errors in the input data~\citep{boyles2019queue}. These latter points are reasons why static traffic assignment is still commonly used for long-range planning problems (where forecasting future demand is difficult) or in applications where many assignment runs are needed (model calibration, OD matrix estimation, or solving bi-level optimization problems such as network design or toll-setting). Clearly the greater realism of dynamic models is a more important consideration in many other applications; the two serve complementary roles as different tools for modeling, and research in both classes of models remains valuable. This article focuses entirely on static assignment.

Static traffic assignment models can be further divided into deterministic and stochastic models. Deterministic assignments assume that users have perfect knowledge of travel costs. Stochastic assignments relax this assumption \citep{daganzo1977stochastic} to model more realistic scenarios where travelers perceive travel costs with some degree of error. In stochastic traffic assignment, the utility of choosing a path is expressed as the sum of a deterministic component (typically a function of the path cost) and a random error term, and travelers choose the path with the highest perceived utility. A discrete choice model specifies the distribution from which these errors are drawn. The multinomial logit route choice model, obtained by assuming the errors are iid Gumbel random variates, is commonly used because it has an elegant mathematical representation amenable to computation, including representing equilibrium as the minimum of a strictly convex function~\citep{fisk1980some}.

Although the multinomial logit choice model is convenient mathematically, the iid assumption on the error terms can be strong, especially when paths share several links, as overlapping paths plausibly have correlated unobserved utilities. The probit model \citep{daganzo1977stochastic} assumes a multivariate normal distribution on the error terms, motivated by the central limit theorem, allowing for correlation between alternatives. However, it lacks a closed-form expression for path choice probabilities and is typically evaluated through simulation \citep{ben1999discrete}, which has limited its application in practice.

Several extensions of the logit model have also been proposed to account for correlation among overlapping paths while retaining analytical tractability. C-logit \citep{cascetta1996modified}, implicit availability/perception logit \citep{cascetta1999implicit}, and path-size logit \citep{ben1999discrete} retain the multinomial logit structure but capture path overlap through a ``commonality factor'' added to the deterministic part of the utility \citep{prashker2004route}. More general formulations come from the generalized extreme value family, such as the paired combinatorial logit \citep{chu1989paired} and the cross-nested logit \citep{vovsha1997cross}, which capture similarity between paths in the error component of the utility rather than the deterministic part. An even more general formulation is the mixed logit model \citep{mcfadden2000mixed}, in which the error term has two components: one drawn from a multivariate distribution (probit-like) and one iid Gumbel (logit-like). As with probit, path choice probabilities under mixed logit are typically computed via simulation. A comprehensive review of these route choice models can be found in \cite{prashker2004route}. Another class of logit-type choice models includes the recursive models introduced by \cite{fosgerau2013link}. These models are link-based and require no restrictions on the choice set; in fact, \cite{fosgerau2013link} show they are mathematically equivalent to a static multinomial logit model with infinitely many alternatives. These models formulate path choice as a Markov decision process in which links are states and a path is a sequence of link choices from an origin to a destination. Building on \cite{fosgerau2013link}, \cite{mai2015nested} proposed a nested recursive logit model that allows correlated utilities. A comprehensive review of several other recursive logit models can be found in \cite{zimmermann2020tutorial}.  Below, we focus only on the original multinomial logit model; broadening our analysis to these extensions of the logit model would be valuable but is beyond our present scope.



In addition to the distribution of unobserved utility, stochastic traffic assignment models vary based on the set of paths travelers consider.  Approaches to specifying this path set broadly fall into two categories: those that enumerate a finite set of paths, and those that operate over an implicit path set without explicitly generating it. Approaches in the first category include $k$-shortest paths \citep{brander1995comparative}, labelling-based paths in which shortest paths under different objectives (e.g., shortest, quickest, fuel-efficient) are combined \citep{ben1984modelling}, link penalty and elimination methods, and simulation-based path generation. Approaches in the second category include Dial's algorithm \citep{dial1971probabilistic}, which performs logit-based loading over an implicit acyclic path set consisting of paths along which a traveler monotonically moves away from the origin, called ``reasonable paths.'' \cite{bell1995alternatives} extends Dial's method to also allow cyclic paths. They propose two variants: the first considers a finite number of paths, and the second considers infinitely many paths in the presence of loops without enumerating them.  These approaches do not require explicit path enumeration, relying instead on a Markov property in which the choice of route from an intermediate node onward is independent of the route taken thus far. \cite{maher1997probit} performs an analogous loading under a probit choice model, again relying on the Markov assumption, without enumerating all paths. \cite{prashker2004route} provides a review of path sets used in conjunction with various route choice models.

MSA and its variants remain among the most widely used algorithms for stochastic traffic assignment, owing to their simplicity and broad applicability across network settings \citep{liu2009method}. MSA works by starting from an initial solution and iteratively updating the current solution as a convex combination of itself and a target solution. The weights in this combination are typically determined by a sequence of step-sizes. When the step-size follows a sequence decreasing to zero, such as $1/k$ at iteration $k$, MSA is known to suffer from slow tail convergence, leading to a sublinear convergence rate. 

To accelerate MSA, researchers often use ``adaptive averaging,'' where 
the step-size at each iteration is determined using information 
generated from previous iterations \citep{liu2009method}. Examples 
include the convex combination algorithm of \citet{chen1991algorithms}, 
the optimal step length algorithm of \citet{maher1998algorithms}, and 
the self-regulating scheme of \citet{liu2009method}. 
\citet{liu2009method} show empirically that their self-regulating 
variant converges significantly faster than standard MSA, although it 
remains sublinear. Their computational results also suggest that a 
hybrid strategy, starting with standard MSA and later switching to the 
optimal step length algorithm of \citet{maher1998algorithms}, can 
achieve linear convergence in practice. However, there is still no 
general way to determine exactly when to make that switch. \citet{bargeraboyce06} showed that if the target mapping is a smooth 
function of the current solution, then MSA with a small constant step 
size converges asymptotically at a linear rate governed by the dominant 
eigenvalue of the iteration matrix.  \citet{mounce2015convergence} further showed, that under certain assumptions, MSA with a constant step drives the iterates only into a 
neighborhood of equilibrium whose radius depends on the constant step-size. However, constant 
step-size rules in the context of logit-based SUE have not yet been explored in literature.

Deterministic models of traffic assignment have been studied extensively on networks of various scales, including large-scale networks such as Philadelphia and Chicago Regional \citep{transportationNetworks}, which contain tens of thousands of links. In contrast, research on stochastic traffic assignment has remained largely confined to smaller networks. Recent studies on SUE \citep{damberg1996algorithm, maher1998algorithms, chen1991algorithms, bekhor2005investigating, bekhor2007application, zhou2012c, xu2012path, chen2013self, yu2014solving, zhou2015modified, cantarella2015stochastic, xu2019improving, du2021faster, zhang2024distributed, wang2025enhancing} have primarily focused on modestly sized networks such as Sioux Falls, Winnipeg, and Chicago Sketch \citep{transportationNetworks}, which contain at most a few thousand nodes and links, far smaller than the networks used in deterministic traffic assignment studies.

Path-based approaches such as disaggregate simplicial decomposition \citep{damberg1996algorithm} and gradient projection \citep{bekhor2005investigating} form the basis of recent computationally fast approaches for logit-based SUE. Building on the SUE objective formulation of \citet{fisk1980some}, \citet{bekhor2005investigating} developed a gradient projection algorithm for logit-based SUE. Their algorithm works by projecting the gradient on the manifold of demand constraints, scaled by a diagonal approximation of the Hessian of Fisk's objective. More recently, \citet{du2021faster} used BB step-sizes, which incorporate second-order information inspired by quasi-Newton methods to accelerate convergence.  Their approach outperforms alternatives such as self-regulating averaging, and they report a runtime of approximately $20$ seconds on the Chicago Sketch network to reach a relative gap of $10^{-10}$ with a path set of approximately nine paths per OD pair. \citet{wang2025enhancing} report similar runtimes of around $50$ seconds on Chicago Sketch to reach the same relative gap using BB step-sizes within a gradient projection framework. In contrast, in the deterministic case, state-of-the-art algorithms (e.g., Algorithm B or TAPAS) achieve a similar gap on Chicago Sketch in only a few seconds. 

Such comparisons must be made carefully, as details of implementation, hardware, and convergence criteria play major roles in run times; but we nevertheless find the difference in reported run times striking, and believe there is value in developing faster algorithms for logit-based SUE.  Although the run time of a single assignment is rarely limiting, many applications involve repeated solution of assignment; as just a few examples, we give model calibration, including OD matrix estimation; Monte Carlo simulation to reflect uncertainty in parameters or future conditions; and bi-level optimization where assignment is a subproblem that may be solved thousands of times.  Any reduction of run time will immediately translate into direct improvements in such applications (e.g., additional calibration iterations or Monte Carlo samples within the same computational budget).

\section{Background}\label{sec:background}

Unlike deterministic user equilibrium, which assumes that all users have perfect knowledge of travel costs, SUE accounts for users' imperfect perception of travel costs. In SUE, users minimize their perceived travel costs between their origin and destination rather than the actual travel costs. If we assume the perception errors are independently and identically distributed according to a Gumbel distribution, the target path flows can be expressed as a closed-form equation that is continuous in the current path flows. The target path flows in this context refer to the distribution of flows across paths in an OD pair that users would choose given the current path costs.

We adopt the following notation conventions. Vector-valued functions are denoted using calligraphic letters (e.g., $\mathcal{Y}$), while scalar-valued functions are represented by lowercase Greek letters (e.g., $\gamma$). Sets are denoted using outlined Roman symbols (e.g., $\mathbb{Y}$). Vectors are represented by bold lowercase Roman letters (e.g., $\mb{y}$), and matrices are denoted by bold uppercase Roman letters (e.g., $\mb{Y}$). The $i^{\text{th}}$ component of a vector is written as $y_i$, and the $(i,j)^{\text{th}}$ component of a matrix is denoted by $Y_{ij}$. Different instances of a vector are distinguished using superscripts, such as $\mb{y}^{1}$ and $\mb{y}^{2}$. Further, throughout this paper, unless otherwise specified, $\|.\|$ denotes the $\ell_2$ vector norm and its induced matrix norm.   Table~\ref{tab:notation} lists the notation used in the paper, as detailed next.

The set of links in the network is denoted by $\Set{L}$, with each link indexed by $l$. The vector of link flows is denoted by $\mb{a}$, where $a_l$ represents the flow on link $l$. The set of OD pairs is denoted by $\Set{OD}$, and we write $m = \lvert \Set{OD} \rvert$ for the number of OD pairs. The demand between an origin $O$ and a destination $D$ is denoted by $d_{OD}$. For each OD pair, $\Set{P}_{OD}$ denotes the set of paths connecting $O$ to $D$, and $[\mb{h}]_{OD}$ denotes the vector of path flows for that OD pair. We assume each $\Set{P}_{OD}$ is finite, contains only acyclic paths, and does not change during the assignment procedure. For an OD pair indexed by $r$, we write $n_r = \lvert \Set{P}_{OD_r} \rvert$ for the number of paths in that OD pair. We denote by $\mb{1} \in \mathbb{R}^n$ and $\mb{1}_{OD} \in \mathbb{R}^{|\mathbb{P}_{OD}|}$ the all-ones vectors over all paths and over an OD pair, respectively, and by $\mb{I}$ the identity matrix. 

Many of the vectors and matrices that appear throughout this paper are naturally organized by OD pair. We assume an arbitrary but fixed ordering $OD_1, OD_2, \ldots, OD_m$ of the OD pairs, and within each OD pair an arbitrary but fixed ordering of the paths in that OD pair. We call a vector OD-blocked if its components are indexed by paths and grouped according to this fixed ordering, with the components indexed by $\mathbb{P}_{OD_1}$ appearing first, followed by those indexed by $\mathbb{P}_{OD_2}$, and so on. Similarly, we call a matrix OD-blocked if its rows and columns are indexed by paths and partitioned into blocks according to the same ordering. We organize matrices and vectors in an OD-blocked fashion whenever possible. For any OD-blocked vector $\mb{x}$, we write $[\mb{x}]_{OD}$ for its restriction to the components indexed by $\mathbb{P}_{OD}$. For any OD-blocked matrix $\mb{X}$, we write $[\mb{X}]_{OD}$ for its diagonal block corresponding to $\mathbb{P}_{OD}$, and $[\mb{X}]_{OD_1, OD_2}$ for the off-diagonal block whose rows are indexed by $\mathbb{P}_{OD_1}$ and whose columns are indexed by $\mathbb{P}_{OD_2}$.

We denote a path by $i$, its associated flow by $h_i$, and we let $OD(i)$ denote the OD pair to which path $i$ belongs. To distinguish paths within and across OD pairs, paths in the same OD pair are denoted using a prime, so $i$ and $i^\prime$ belong to the same OD pair, while $i$ and $j$ belong to different OD pairs. We denote the union of acyclic paths across all OD pairs by $\Set{P}$, with $n = \lvert \Set{P} \rvert = \sum_{r=1}^{m} n_r$ denoting the total number of paths. The vector $\mb{h} \in \mathbb{R}^{n}$ represents the ordered vector of all $[\mb{h}]_{OD}$ sub-vectors across OD pairs. A path flow vector $[\mb{h}]_{OD}$ is called feasible if each of its components is non-negative and they sum to $d_{OD}$, and the vector $\mb{h}$ is feasible if all its sub-vectors $[\mb{h}]_{OD}$ are feasible. The link-path incidence matrix $\mb{D} \in \{0,1\}^{|\mathbb{L}| \times n}$ is defined as
\begin{equation}
D_{li} =
\begin{cases}
1, & \text{if link } l \text{ is part of path } i, \\
0, & \text{otherwise}.
\end{cases}
\end{equation}

Given a feasible path flow vector $\mb{h}$, the corresponding path cost vector $\mb{c} = \mathcal{C}(\mb{h})$ is computed via network loading. The path flows are first aggregated into link flows by summing across all paths traversing each link. The cost on link $l$ is then evaluated using a link performance function $\tau_l(a_l)$. We assume the cost of each link depends only on the flow on that link. The cost of a path is the sum of the costs of all links it traverses. We further assume that each $\tau_l$ is twice continuously differentiable and monotone non-decreasing on $[0, \infty)$, with $\tau_l'$ Lipschitz continuous on the bounded feasible set of path flows. Suppose $\mathcal{T}$ denotes the vector-valued link cost mapping whose $l$-th component is $\tau_l(a_l)$. The path cost mapping then takes the composed form $\mathcal{C}(\mb{h}) = \mb{D}^T \mathcal{T}(\mb{D}\mb{h})$. The Jacobian of path costs with respect to path flows is given by the chain rule as $\mathcal{C}'(\mb{h}) = \mb{D}^T \mathcal{T}^{\prime}(\mb{a})\, \mb{D}$, where $\mathcal{T}^{\prime}(\mb{a}) := \diag(\tau_l'(a_l))$ is the diagonal matrix of marginal link costs. Given our assumption that the link cost functions are separable and monotone non-decreasing, $\mathcal{T}'(\mb{a})$ is positive semi-definite (PSD), and hence $\mathcal{C}'(\mb{h})$ is symmetric PSD at every feasible $\mb{h}$.

For a feasible path flow vector, the target path flows are obtained using the logit path choice model. Under this model, the target flow on path $i \in \Set{P}_{OD}$ is given by
\begin{equation}
    h_i = d_{OD} p_i = d_{OD} \left(\frac{\exp(-\theta c_i)}{\sum_{i^\prime \in \Set{P}_{OD(i)}} \exp(-\theta c_{i^\prime})} \right)\,,
    \label{eq:2_suepathflow}
\end{equation}
where $c_i$ is the cost associated with path $i$, $p_i$ is the probability of choosing path $i$ given the path cost vector $\mb{c}$, and $\theta > 0$ is a dispersion parameter. The parameter $\theta$ captures the sensitivity of travelers to differences in perceived path costs relative to their perception errors.

We denote the vector-valued logit probability mapping by $\mb{p} = \mathcal{P}(\mb{c}) \in \mathbb{R}^{n}$, whose components are the individual path probabilities $p_i$. This vector is composed of sub-vectors $[\mathcal{P}(\mb{c})]_{OD}$, each of which sums to one. We define $\mathcal{H}(\mb{h}, \mb{p})$ as the mapping that yields target path flows by multiplying the total OD demand (the sum of the components of $\mb{h}$ for each OD pair) by the corresponding probability vector. Composing these mappings, we write $\mathcal{L}(\mb{h}) := \mathcal{H}(\mb{h},\, \mathcal{P}(\mathcal{C}(\mb{h})))$ and refer to $\mathcal{L}$ as the logit mapping. As a composition of differentiable functions, $\mathcal{L}$ is itself differentiable.

The SUE path flow $\hat{\mb{h}}$ is defined as the solution to the fixed-point problem $\hat{\mb{h}} = \mathcal{L}(\hat{\mb{h}})$. It is well known that under mild conditions, this equilibrium exists, is unique, and can be characterized as the minimizer of a strictly convex function \citep{fisk1980some}. When an iterative procedure such as MSA is applied to compute the SUE, the procedure takes the form shown in Algorithm~\ref{alg:2_msa}. At each iteration $k$, a step-size $s_k \in (0,1]$ is chosen, and the current path flows are updated as a convex combination of the current and target flows. The target flows are then recomputed from the logit model applied to the updated path costs. The algorithm terminates when the current and target flows are sufficiently close for every OD pair, as measured by a gap function.

\begin{algorithm}
\caption{Method of successive averages}
\label{alg:2_msa}
\begin{algorithmic}[1]
\State \textbf{Initialization:} Set $\mb{h}^0$ to free-flow path flows.
\For{$k = 0,1,2,\dots$}
    \State \textbf{(a)} Choose a step-size $s_k \in (0,1]$.
    \State \textbf{(b)} Update path flows:
    \[
        \mb{h}^{k+1} \gets (1-s_k)\,\mb{h}^k + s_k\,\mathcal{L}(\mb{h}^k).
    \]
    \State \textbf{(c)} If the gap function evaluates to a sufficiently small value, \textbf{terminate}.
\EndFor
\end{algorithmic}
\end{algorithm}

\begin{table}
\centering
\small
\caption{Notation used throughout the paper.}
\label{tab:notation}
\begin{tabular}{@{}p{0.12\textwidth} p{0.16\textwidth} p{0.66\textwidth}@{}}
\toprule
\textbf{Symbol} & \textbf{Dimension} & \textbf{Description} \\
\midrule
\multicolumn{3}{@{}l}{\textit{Sets and indices}} \\
\midrule
$l \in \mathbb{L}$ & $|\mathbb{L}|$ & Links in the network \\
$OD \in \mathbb{OD}$ & $|\mathbb{OD}|$ & Origin-destination pair; indexed by $r, s$ when needed \\
$i, i' \in \mathbb{P}_{OD}$ & $|\mathbb{P}_{OD}|$ & Paths within the same OD pair ($j$ represents paths in a different OD pair)\\
$i \in \mathbb{P}$ & $|\mathbb{P}|=n$ & Path in the union of all acyclic paths \\
$\mathbb{U}$ & --- & Arbitrary neighborhood of $\hat{\mb{h}}$ \\
$k$ & --- & Iteration number of traffic assignment solver \\
\midrule
\multicolumn{3}{@{}l}{\textit{Scalars}} \\
\midrule
$d_{OD}$ & $\mathbb{R}_{+}$ & Demand for an OD pair \\
$\theta$ & $\mathbb{R}_{+}$ & Logit dispersion parameter \\
$s_k$ & $(0, 1]$ & MSA step-size in iteration $k$ \\
$g_k$ & $\mathbb{R}_{+}$ & Gap function value in iteration $k$ \\
$w_i$ & $\mathbb{R}$ & First-order derivative of the SUE objective with respect to $h_i$\\ 
$w_{\min}^{OD}$ & $\mathbb{R}$ & Minimum of $w_i$ over $i \in \Set{P}_{OD}$ \\
\midrule
\multicolumn{3}{@{}l}{\textit{Vectors}} \\
\midrule
$\mb{a}$ & $\mathbb{R}_{+}^{|\mathbb{L}|}$ & Link flow vector, components $a_l$ \\
$\mb{h}$, $[\mb{h}]_{OD}$ & $\mathbb{R}_{+}^{n}$, $\mathbb{R}_{+}^{|\mathbb{P}_{OD}|}$ & Path flow vector and OD sub-vector having components $h_i$ \\
$\mb{h}^k$, $\hat{\mb{h}}$ & $\mathbb{R}_{+}^{n}$ & Path flow iterate at iteration $k$ and SUE path flow\\
$\mb{c}$, $[\mb{c}]_{OD}$ & $\mathbb{R}_{+}^{n}$, $\mathbb{R}_{+}^{|\mathbb{P}_{OD}|}$ & Path cost vector and OD sub-vector having components $c_i$ \\
$\mb{p}$, $[\mb{p}]_{OD}$ & $(0,1)^{n}$, $(0,1)^{|\mathbb{P}_{OD}|}$ & Path probability vector and OD sub-vector having components $p_i$ \\
$\mb{0}$, $\mb{1}$, $\mb{1}_{OD}$ & $\mathbb{R}^n$, $\mathbb{R}^n$, $\mathbb{R}^{|\mathbb{P}_{OD}|}$ & Zero vector, and all-ones vectors over all paths and over an OD pair \\
$\mb{\bar{g}}$ & $\mathbb{R}_+^q$ & FIFO queue of gap values \\
$\boldsymbol{\delta}^k$ & $\mathbb{R}^n$ & Newton-like step at iteration $k$ \\
\midrule
\multicolumn{3}{@{}l}{\textit{Scalar-valued functions}} \\
\midrule
$\tau_l(a_l)$ & $\mathbb{R}_{+} \to \mathbb{R}_{+}$ & Link performance function on link $l$ \\
$\tau_l'(a_l)$ & $\mathbb{R}_{+} \to \mathbb{R}_{+}$ & Derivative of $\tau$ with respect to link flow \\
$OD(i)$ & $\mathbb{P} \to \mathbb{OD}$ & OD pair to which path $i$ belongs \\
\midrule
\multicolumn{3}{@{}l}{\textit{Vector-valued functions}} \\
\midrule
$\mathcal{C}(\mb{h})$ & $\mathbb{R}_{+}^{n} \to \mathbb{R}_{+}^{n}$ & Path cost mapping (network loading) \\
$\mathcal{P}(\mb{c})$ & $\mathbb{R}_{+}^{n} \to (0,1)^{n}$ & Logit path choice probability mapping \\
$\mathcal{H}(\mb{h}, \mb{p})$ & $\mathbb{R}_{+}^{n} \times (0,1)^{n} \to \mathbb{R}_{+}^{n}$ & Target path flow mapping \\
$\mathcal{L}(\mb{h})$ & $\mathbb{R}_{+}^{n} \to \mathbb{R}_{+}^{n}$ & Logit mapping defined as $\mathcal{H}(\mb{h}, \mathcal{P}(\mathcal{C}(\mb{h})))$\\
$\mathcal{F}(\mb{h})$ & $\mathbb{R}_{+}^{n} \to \mathbb{R}^{n}$ & Residual mapping defined as $\mathcal{L}(\mb{h}) - \mb{h}$ \\
$\mathcal{T}(\mb{a})$ & $\mathbb{R}_{+}^{|\mathbb{L}|} \to \mathbb{R}_{+}^{|\mathbb{L}|}$ & Mapping from link flows to link costs, $\mathcal{T}_{l}(\mb{a}) = \tau_l(a_l)$ \\
\midrule
\multicolumn{3}{@{}l}{\textit{Matrix-valued functions}} \\
\midrule
$\mathcal{C}'(\mb{h})$ & $\mathbb{R}_{+}^{n} \to \mathbb{R}^{n \times n}$ & Jacobian of path costs with respect to $\mb{h}$ \\
$\mathcal{P}'(\mb{c})$ & $\mathbb{R}_{+}^{n} \to \mathbb{R}^{n \times n}$ & Jacobian of path probabilities with respect to $\mb{c}$ \\
$\mathcal{H}'_{\mb{h}}(\mb{p})$ & $(0,1)^{n} \to \mathbb{R}^{n \times n}$ & Jacobian of $\mathcal{H}$ with respect to $\mb{h}$ ($\mb{p}$ fixed) \\
$\mathcal{H}'_{\mb{p}}(\mb{h})$ & $\mathbb{R}^{n}_+ \to \mathbb{R}^{n \times n}$ & Jacobian of $\mathcal{H}$ with respect to $\mb{p}$ ($\mb{h}$ fixed) \\
$\mathcal{L}'(\mb{h})$ & $\mathbb{R}_{+}^{n} \to \mathbb{R}^{n \times n}$ & Jacobian of the logit mapping with respect to $\mb{h}$\\
$\mathcal{F}'(\mb{h})$ & $\mathbb{R}_{+}^{n} \to \mathbb{R}^{n \times n}$ & Jacobian of $\mathcal{F}$ with respect to $\mb{h}$\\
$\mathcal{K}(\mb{h})$ & $\mathbb{R}_{+}^{n} \to \mathbb{R}^{n \times n}$ & Reduced Jacobian defined as $\mathcal{H}'_{\mb{p}}(\mb{h})\,\mathcal{P}'(\mb{c})\,\mathcal{C}'(\mb{h})$ \\
$\mathcal{T}^{\prime}(\mb{a})$ & $\mathbb{R}_{+}^{|\mathbb{L}|} \to \mathbb{R}^{|\mathbb{L}| \times |\mathbb{L}|}$ & Diagonal matrix of marginal link costs, $\mathcal{T}^{\prime}_{ll} = \tau'(a_l)$ \\\midrule
\multicolumn{3}{@{}l}{\textit{Matrices}} \\
\midrule
$\mb{D}$ & $\{0,1\}^{|\mathbb{L}| \times n}$ & Link-path incidence matrix \\
$\mb{I}$ & $\mathbb{R}^{n \times n}$ & Identity matrix \\
$\mb{J}$ & $\mathbb{R}^{n \times n}$ & Shorthand for $\mathcal{C}'(\mb{h})$ \\
$\mb{S}$ & $\mathbb{R}^{n \times n}$ & Shorthand for $-\mathcal{H}'_{\mb{p}}(\mb{h})\,\mathcal{P}'(\mb{c})$ \\
\bottomrule
\end{tabular}
\end{table}

\section{Spectral analysis of the logit mapping}\label{sec:3_Jacobian}

We start by analyzing the Jacobian of the logit mapping, focusing on its spectral properties. We first derive the Jacobian and summarize its key properties in Section~\ref{subsec:3_Jacobian}. We then decompose the Jacobian into two matrices and show that one term vanishes when multiplied by the difference of any two feasible path flow vectors. In Section~\ref{subsec:3_evals}, we characterize the eigenvalues of the remaining term and show that they are non-positive real numbers, with zero as an eigenvalue. We then use these spectral properties of the Jacobian in Section~\ref{sec:bound} to establish linear convergence of MSA with a constant step-size, and in Section~\ref{sec:newton} to derive a quadratically convergent traffic assignment algorithm for logit-SUE. We provide a worked example illustrating the construction of the Jacobian of the logit mapping and its properties on the Braess network in Appendix~\ref{app:braess_jacob}, for readers who prefer to follow the analysis along with a demonstration.

\subsection{Jacobian of the logit mapping}\label{subsec:3_Jacobian}

To construct the Jacobian $\mathcal{L}'(\mb{h}) \in \mathbb{R}^{n \times n}$, we need to compute the partial derivatives $\partial \mathcal{L}_i / \partial h_j$ for all pairs of paths $i, j \in \Set{P}$. To derive the Jacobian, we adopt the following notation. For any two OD pairs $OD_1$ and $OD_2$, we denote the number of paths in $OD_1$ as $n_1$ and the number of paths in $OD_2$ as $n_2$. Constructing the Jacobian requires some care, since perturbing a single path flow $h_j$ while holding all others fixed changes the total demand of the OD pair to which path $j$ belongs, but leaves the demands of all other OD pairs unchanged. To account for this, we express the mapping $\mathcal{L}$ in a more general form that allows each OD demand to vary, while preserving its value for demand-feasible solutions:
\begin{equation}
\mathcal{L}(\mb{h})_i = \left(\sum_{i^\prime \in \Set{P}_{OD(i)}} h_{i^\prime} \right)\, \mathcal{P}(\mathcal{C}(\mb{h}))_{i},
\label{eq:dem_h}
\end{equation}
where $OD(i)$ denotes the OD pair to which path $i$ belongs, and $\mathcal{P}(\mathcal{C}(\mb{h}))_{i}$ is the path choice probability of path $i$. Equation~\eqref{eq:dem_h} expresses each component $\mathcal{L}_i$ as the product of the demand of path $i$'s OD pair (obtained by summing path flows over all paths $i^\prime \in \Set{P}_{OD(i)}$) and the choice probability of path $i$. 

Because each partial derivative $\partial \mathcal{L}_i / \partial h_j$ behaves differently depending on whether paths $i$ and $j$ share an OD pair, the Jacobian admits a natural block structure as shown in Equation~\eqref{eq:jacobian_block}. We call the blocks $OD_1 = OD_2$ the ``within-OD blocks,'' and $OD_1 \neq OD_2$ the ``cross-OD blocks.''

\begin{equation}
\centering
\begin{tikzpicture}[scale=1.0, baseline=(current bounding box.center)]]
  \draw[thick] (0.0, 3.0) -- (-0.05, 3.0) -- (-0.05,-3.0) -- (0.0,-3.0);
  \draw[thick] (9.0, 3.0) -- ( 9.05, 3.0) -- ( 9.05,-3.0) -- (9.0,-3.0);
  \draw[dashed] (0.0, 1.0) -- (9.0, 1.0);
  \draw[dashed] (0.0,-1.0) -- (9.0,-1.0);
  \draw[dashed] (3.0, 3.0) -- (3.0,-3.0);
  \draw[dashed] (6.0, 3.0) -- (6.0,-3.0);
  \node[align=center] at (1.5, 2.6)
        {$\big[\mathcal{L}'\big]_{OD_1}$};
  \node[align=center] at (1.5, 2.1)
        {\footnotesize within-OD block};
  \node[align=center] at (1.5, 1.6)
        {\footnotesize $i, i^\prime \in \Set{P}_{OD_1}$};
  \node[align=center] at (4.5, 2.0)
        {$\cdots$};
  \node[align=center] at (7.5, 2.6)
        {$\big[\mathcal{L}'\big]_{OD_1,\, OD_m}$};
  \node[align=center] at (7.5, 2.1)
        {\footnotesize cross-OD block};
  \node[align=center] at (7.5, 1.6)
        {\footnotesize $i \in \Set{P}_{OD_1},\, j \in \Set{P}_{OD_m}$};
  \node[align=center] at (1.5, 0.0)
        {$\vdots$};
  \node[align=center] at (4.5, 0.0)
        {$\ddots$};
  \node[align=center] at (7.5, 0.0)
        {$\vdots$};
  \node[align=center] at (1.5,-2.0)
        {$\big[\mathcal{L}'\big]_{OD_m,\, OD_1}$};
  \node[align=center] at (4.5,-2.0)
        {$\cdots$};
  \node[align=center] at (7.5,-2.0)
        {$\big[\mathcal{L}'\big]_{OD_m}$};
  \node at (-1.0, 0.0) {$\mathcal{L}'(\mb{h}) \;=\;$};
  \node at (10, 0.0) {$\in \mathbb{R}^{n \times n}$};
\end{tikzpicture}
\label{eq:jacobian_block}
\end{equation}

Applying the chain rule, the Jacobian of $\mathcal{L}(\mb{h}) = \mathcal{H}(\mb{h},\, \mathcal{P}(\mathcal{C}(\mb{h})))$ with respect to $\mb{h}$ is
\begin{equation}
    \mathcal{L}'(\mb{h}) = \mathcal{H}'_{\mb{h}}(\mb{p}) + \mathcal{H}'_{\mb{p}}(\mb{h}) \, \mathcal{P}'(\mb{c}) \, \mathcal{C}'(\mb{h}), \label{eq:JacobianFull}
\end{equation}
where $\mathcal{H}'_{\mb{h}}(\mb{p})$ is the Jacobian of $\mathcal{H}$ with respect to path flows $\mb{h}$ (treating $\mb{p}$ as constant), $\mathcal{H}'_{\mb{p}}(\mb{h})$ is the Jacobian of $\mathcal{H}$ with respect to probabilities $\mb{p}$ (treating $\mb{h}$ as constant), $\mathcal{P}'(\mb{c})$ is the Jacobian of path probabilities with respect to path costs $\mb{c}$, and $\mathcal{C}'(\mb{h})$ is the Jacobian of path costs with respect to path flows $\mb{h}$. 

We analyze $\mathcal{L}'(\mb{h})$ block-wise to build the full Jacobian. In general, the $(OD_1, OD_2)$ block of $\mathcal{H}'_{\mb{h}}(\mb{p}) + \mathcal{H}'_{\mb{p}}(\mb{h}) \, \mathcal{P}'(\mb{c}) \, \mathcal{C}'(\mb{h})$ depends on every cross-OD block of each factor, since
\begin{equation}
    [\mathcal{H}'_{\mb{p}}(\mb{h}) \, \mathcal{P}'(\mb{c}) \, \mathcal{C}'(\mb{h})]_{OD_1, OD_2} = \sum_{OD_r,\, OD_s \in\, \mathbb{OD}} [\mathcal{H}'_{\mb{p}}(\mb{h})]_{OD_1, OD_r} \, [\mathcal{P}'(\mb{c})]_{OD_r, OD_s} \, [\mathcal{C}'(\mb{h})]_{OD_s, OD_2},
    \label{eq:blockproduct}
\end{equation}
so the within-OD and cross-OD blocks of $\mathcal{L}'(\mb{h})$ would not, in general, decouple from one another. However, both $\mathcal{H}'_{\mb{p}}(\mb{h})$ and $\mathcal{P}'(\mb{c})$ are in fact block-diagonal in the OD partition. The matrix $\mathcal{H}'_{\mb{p}}(\mb{h})$ is block-diagonal because, by Equation~\eqref{eq:dem_h}, $\mathcal{L}_i$ depends only on the probability $p_i$ of path $i$ itself, so $\partial \mathcal{H}_i / \partial p_j = 0$ whenever $j \notin \Set{P}_{OD(i)}$. The matrix $\mathcal{P}'(\mb{c})$ is block-diagonal because, under the logit model, the choice probability of a path within an OD pair depends only on the costs of paths within that same OD pair. Substituting this block-diagonal structure into Equation~\eqref{eq:blockproduct} collapses the double sum, and the $(OD_1, OD_2)$ block of the product reduces to
\begin{equation}
    [\mathcal{H}'_{\mb{p}}(\mb{h}) \, \mathcal{P}'(\mb{c}) \, \mathcal{C}'(\mb{h})]_{OD_1, OD_2} = [\mathcal{H}'_{\mb{p}}(\mb{h})]_{OD_1} \, [\mathcal{P}'(\mb{c})]_{OD_1} \, [\mathcal{C}'(\mb{h})]_{OD_1, OD_2}.
\end{equation}
This allows us to characterize $\mathcal{L}'(\mb{h})$ block by block. We now characterize each of these Jacobian matrices as block matrices organized by OD pairs.

The Jacobian of path costs with respect to path flows, $\mathcal{C}'(\mb{h}) \in \mathbb{R}^{n \times n}$, has entries
\begin{equation}
    \pdr{\mathcal{C}_i}{h_j} = \sum_{\substack{l \in \Set{L} \text{ common } \\ \text{ on paths } i \text{ and } j}} \tau_l'(a_l), \label{eq:3_codtderdef}
\end{equation}
for any pair of paths $i, j \in \Set{P}$ (whether or not they belong to the same OD pair), where $\tau_l(a_l)$ is the link cost of link $l$ with flow $a_l$.  We denote the within-OD block as $[\mathcal{C}'(\mb{h})]_{OD_1} \in \mathbb{R}^{n_1 \times n_1}$ and the cross-OD block as $[\mathcal{C}'(\mb{h})]_{OD_1, OD_2} \in \mathbb{R}^{n_1 \times n_2}$ with both having entries given by the same element-wise formula in Equation~\eqref{eq:3_codtderdef}. 

We now consider the within-OD blocks of $\mathcal{L}$. For paths $i, {i^\prime} \in \Set{P}_{OD_1}$, we examine how $\mathcal{L}$ changes with a change in flow within the same OD pair. In $\mathcal{H}'_{\mb{h}}(\mb{p})$ for this block (denoted by $[\mathcal{H}'_{\mb{h}}(\mb{p})]_{OD_1} \in \mathbb{R}^{n_1 \times n_1}$), we treat $\mb{p}$ as constant and find that $\partial \mathcal{H}_i / \partial h_{i^\prime} = p_i$, so each row of the within-OD block consists of the same value $p_i$. Similarly, the matrix $[\mathcal{H}'_{\mb{p}}(\mb{h})]_{OD_1} \in \mathbb{R}^{n_1 \times n_1}$ restricted to this block is diagonal with elements $\sum_{{i^\prime} \in \Set{P}_{OD_1}} h_{i^\prime} = d_{OD}$. The elements of $[\mathcal{P}'(\mb{c})]_{OD_1} \in \mathbb{R}^{n_1 \times n_1}$ within this block are the sensitivities of the logit model with respect to path costs. As is well-known \citep[e.g.,][]{koppelman2006self}, they are given by
\begin{equation}
    ([\mathcal{P}'(\mb{c})]_{OD})_{i{i^\prime}} = \begin{cases}
-\theta \cdot p_i \cdot (1 - p_i) & \text{if } i = {i^\prime} \in \Set{P}_{OD},\\
\theta \cdot p_i \cdot p_{i^\prime} & \text{if } i \neq {i^\prime} \in \Set{P}_{OD}.
\end{cases} \label{eq:pprime}
\end{equation}

Next, we consider the cross-OD blocks of $\mathcal{L}$. For $i \in \Set{P}_{OD_1}$ and $j \in \Set{P}_{OD_2}$ with $OD_1 \neq OD_2$, perturbing $h_j$ does not change the total demand in $OD_1$. Consequently, $\mathcal{H}'_{\mb{h}}$ does not contribute to the cross-OD blocks of $\mathcal{L}'(\mb{h})$, and
\begin{equation}
\{[\mathcal{L}'(\mb{h})]_{OD_1, OD_2}\}_{ij} = \left\{[\mathcal{H}'_{\mb{p}}(\mb{h})]_{OD_1} \, [\mathcal{P}'(\mb{c})]_{OD_1} \, [\mathcal{C}'(\mb{h})]_{OD_1, OD_2}\right\}_{ij}.
\label{eq:crossOD}
\end{equation}

The above block-by-block analysis reveals a useful structural property of $\mathcal{H}'_{\mb{h}}(\mb{p})$: although it is generally fully dense and nonzero, it annihilates differences of feasible path flow vectors. We state this formally in the next lemma.

\begin{lem}\label{lem:Hh_zero}
The product of $\mathcal{H}'_{\mb{h}}(\mb{p})$ with the difference of any two feasible path flow vectors is zero. That is, for any two feasible path flow vectors $\mb{h}^{1}, \mb{h}^{2} \in \mathbb{R}^{n}$,
\begin{equation}
\mathcal{H}'_{\mb{h}}(\mb{p})\,(\mb{h}^{1} - \mb{h}^{2}) = \mb{0}.
\end{equation}
\end{lem}

\begin{proof}
We know $\mathcal{H}'_{\mb{h}}(\mb{p})$ has a block-diagonal structure: the blocks $[\mathcal{H}'_{\mb{h}}(\mb{p})]_{OD_r} \in \mathbb{R}^{n_r \times n_r}$ have each row consisting of the same value $p_i$, and the blocks ($[\mathcal{H}'_{\mb{h}}(\mb{p})]_{OD_r, OD_s}$) with $r\neq s$ are zero because perturbing $h_j$ in one OD pair does not change the demand of another. For any two feasible path flow vectors $\mb{h}^{1}, \mb{h}^{2} \in \mathbb{R}^{n}$, within each OD block the $i$-th component of the product satisfies
\begin{align}
    (\mathcal{H}'_{\mb{h}}(\mb{p}) \, \mb{h}^{1})_i
= \sum_{i^\prime \in \Set{P}_{OD(i)}} p_i \, h^{1}_{i^\prime}
= p_i \, d_{OD},
\end{align}
which is the same for $\mb{h}^{2}$. Hence, $\mathcal{H}'_{\mb{h}}(\mb{p})(\mb{h}^{1} - \mb{h}^{2}) = \mb{0}$. In other words, the first term in the Jacobian decomposition~\eqref{eq:JacobianFull} vanishes whenever it acts on the difference of two feasible path flow vectors.
\end{proof}

Lemma~\ref{lem:Hh_zero} has an important consequence for the structure of the Jacobian $\mathcal{L}'(\mb{h})$. The Jacobian $\mathcal{L}'(\mb{h})$ splits into two terms: the first involving $\mathcal{H}'_{\mb{h}}(\mb{p})$, which captures the dependence of target path flows on current path flows through the demand, and the second involving the chain $\mathcal{H}'_{\mb{p}}(\mb{h}) \cdot \mathcal{P}'(\mb{c}) \cdot \mathcal{C}'(\mb{h})$, which captures the dependence through the path choice probabilities. By Lemma~\ref{lem:Hh_zero}, the first term vanishes when applied to differences of feasible path flow vectors, so only the second term contributes to $\mathcal{L}'(\mb{h})(\mb{h}^1 - \mb{h}^2)$. This property motivates the following definition.

\begin{dfn}
The reduced Jacobian operator $\mathcal{K}(\mb{h}) \in \mathbb{R}^{n \times n}$ is defined as
\begin{equation}\label{eq:3_defK}
\mathcal{K}(\mb{h}) := \mathcal{H}'_{\mb{p}}(\mb{h}) \cdot \mathcal{P}'(\mb{c}) \cdot \mathcal{C}'(\mb{h}),
\end{equation}
which isolates the second term in~\eqref{eq:JacobianFull}.
\end{dfn}

When the Jacobian of the logit mapping $\mathcal{L}$ is applied to the difference of two feasible path flow vectors $\mb{h}^{1}$ and $\mb{h}^{2}$, Lemma~\ref{lem:Hh_zero} implies that
\begin{align}
    \mathcal{L}'(\mb{h})(\mb{h}^{1} - \mb{h}^{2}) = \mathcal{K}(\mb{h})(\mb{h}^{1} - \mb{h}^{2}). \label{eq:equalLK}
\end{align}

\subsection{Spectral analysis of the reduced Jacobian\label{subsec:3_evals}}

We now perform a spectral analysis of $\mathcal{K}(\mb{h})$. We can express the matrix $\mathcal{K}(\mb{h})$ using the block notation from Section~\ref{subsec:3_Jacobian}. We define two new matrices $\mb{J} := \mathcal{C}'(\mb{h})$ and $\mb{S} := -\mathcal{H}'_{\mb{p}}(\mb{h}) \, \mathcal{P}'(\mb{c})$ to help with the spectral analysis of $\mathcal{K}(\mb{h})$, and write $\mathcal{K}(\mb{h})$ as $-\mb{S} \mb{J}$. Because both $\mathcal{H}'_{\mb{p}}(\mb{h})$ and $\mathcal{P}'(\mb{c})$ are block-diagonal with respect to OD pairs, $\mb{S}$ is also block-diagonal. For any given OD pair, the within-OD block of $\mb{S}$ is
\begin{equation}
    [\mb{S}]_{OD} = -[\mathcal{H}'_{\mb{p}}(\mb{h})]_{OD} \, [\mathcal{P}'(\mb{c})]_{OD} = d_{OD} \theta \left( \diag([\mb{p}]_{OD}) - [\mb{p}]_{OD} ([\mb{p}]_{OD})^T \right), \label{eq:S}
\end{equation}
where $[\mb{p}]_{OD}$ is the column vector of path choice probabilities for that specific OD pair. The matrix $\mb{S}$ is symmetric, being the difference of two symmetric matrices. Furthermore, $\mb{J} = \mathcal{C}'(\mb{h})$ is symmetric by Equation~\eqref{eq:3_codtderdef}. This decomposition therefore expresses $\mathcal{K}(\mb{h})$ as the product of two symmetric matrices.

Our goal in this subsection is to characterize the spectrum of $\mathcal{K}(\mb{h})$. Our argument is distributed over three lemmas leading to the main result. In Lemma~\ref{lem:onesS} we show the vector of all-ones lies in the left null space of $\mb{S}$, which immediately yields zero as an eigenvalue of $\mathcal{K}(\mb{h})$. Then, in Lemma~\ref{lem:Spsd}, we show the matrix $\mb{S}$ is PSD, and therefore admits a symmetric square root $\mb{S}^{1/2}$. Finally, in Lemma~\ref{lem:Snorm} we show the spectral norm of $\mb{S}$ is bounded above by $\theta \max d_{OD}$. With these three lemmas in hand, we prove the main spectral result of this paper, which establishes that all eigenvalues of $\mathcal{K}(\mb{h})$ are real and non-positive, that zero is the maximum eigenvalue, and that the minimum eigenvalue is bounded below in terms of the maximum OD demand and the spectral norm of $\mathcal{C}'(\mb{h})$.

\begin{lem}\label{lem:onesS}
The all-ones vector $\mb{1} \in \mathbb{R}^n$ satisfies $\mb{1}^T \mb{S} = \mb{0}^T$. Consequently, $0$ is an eigenvalue of $\mathcal{K}(\mb{h})$.
\end{lem}
\begin{proof}
Since Equation~\eqref{eq:S} holds for each OD block, and the path choice probabilities sum to one, the vector of all ones $\mb{1}_{OD} \in \mathbb{R}^{|\mathbb{P}_{OD}|}$ satisfies
\begin{align}
    \mb{1}_{OD}^T [\mb{S}]_{OD} &= d_{OD} \theta \left( \mb{1}_{OD}^T \diag([\mb{p}]_{OD}) - \mb{1}_{OD}^T [\mb{p}]_{OD} ([\mb{p}]_{OD})^T \right)\\& = d_{OD} \theta \left( ([\mb{p}]_{OD})^T - (1)([\mb{p}]_{OD})^T \right) = \mb{0}^T. \label{eq:odDS}
\end{align}
Equation~\eqref{eq:odDS} holds for every block along the diagonal, so the all-ones vector $\mb{1} \in \mathbb{R}^n$ satisfies $\mb{1}^T \mb{S} = \mb{0}^T$. Consequently, $\mb{1}^T \mathcal{K}(\mb{h}) = -(\mb{1}^T \mb{S}) \mb{J} = \mb{0}^T$, and hence $0$ is an eigenvalue of $\mathcal{K}(\mb{h})$.
\end{proof}

To analyze the remaining eigenvalues, we first establish that $\mb{S}$ is PSD. This ensures that a real symmetric square root $\mb{S}^{1/2}$ exists, which allows us to compare the eigenvalues of $\mathcal{K}(\mb{h}) = -\mb{S}\mb{J}$ with those of $-\mb{S}^{1/2}\mb{J}\mb{S}^{1/2}$. This is helpful because, although $-\mb{S}\mb{J}$ is not necessarily symmetric, $-\mb{S}^{1/2}\mb{J}\mb{S}^{1/2}$ is, allowing us to invoke the spectral theorem for real symmetric matrices.

\begin{lem}\label{lem:Spsd}
The matrix $\mb{S}$ is symmetric PSD, and therefore possesses a symmetric square root $\mb{S}^{1/2}$.
\end{lem}
\begin{proof}
Since $\mb{S}$ is block-diagonal, it is PSD if and only if each block $[\mb{S}]_{OD}$ is PSD. For any real vector $\mb{x} \in \mathbb{R}^{|\mathbb{P}_{OD}|}$ we evaluate the quadratic form:
\begin{align}
    \mb{x}^T [\mb{S}]_{OD} \mb{x} &= d_{OD} \theta \, \mb{x}^T \left( \diag([\mb{p}]_{OD}) - [\mb{p}]_{OD} ([\mb{p}]_{OD})^T \right) \mb{x}  \\
    &= d_{OD} \theta \left( \sum_{i \in \Set{P}_{OD}} p_i x_i^2 - \left(\sum_{i \in \Set{P}_{OD}} p_i x_i\right)^2 \right).
\end{align}
Because $[\mb{p}]_{OD}$ represents a valid probability distribution (i.e., $p_i \ge 0$ and sums to unity), this difference is guaranteed to be non-negative as it represents the variance of a random variable taking values $x_i$ with probabilities $p_i$. Because $d_{OD} \ge 0$ and $\theta > 0$, $\mb{x}^T [\mb{S}]_{OD} \mb{x} \ge 0$, so every block of $\mb{S}$ is PSD. Hence, $\mb{S}$ is symmetric PSD and therefore possesses a symmetric square root $\mb{S}^{1/2}$.
\end{proof}

The next lemma bounds the spectral norm of $\mb{S}$, which will be used to derive the lower bound on the minimum eigenvalue of $\mathcal{K}(\mb{h})$.

\begin{lem}\label{lem:Snorm}
The spectral norm of $\mb{S}$ satisfies $\left\|\mb{S}\right\| \le \theta  \left( \max\limits_{OD} d_{OD} \right)$.
\end{lem}
\begin{proof}
Since $\mb{S}$ is a symmetric PSD matrix, its spectral norm is given by
\begin{equation}
    \left\|\mb{S}\right\| = \max_{\|\mb{x}\| = 1} \mb{x}^T \mb{S} \mb{x}. \label{eq:specnormdef}
\end{equation}
We know that $\mb{S}$ is block-diagonal with respect to OD pairs. So, its spectral norm is the maximum of the norms of its blocks. From Equation~\eqref{eq:S}, the within-OD block is
\begin{equation}
    [\mb{S}]_{OD} = d_{OD} \theta \left( \diag([\mb{p}]_{OD}) - [\mb{p}]_{OD} ([\mb{p}]_{OD})^T \right).
\end{equation}
Given that $[\mb{p}]_{OD}$ is a probability vector, for any unit vector $\mb{x}$,
\begin{align}
    \mb{x}^T \left( \diag([\mb{p}]_{OD}) - [\mb{p}]_{OD} ([\mb{p}]_{OD})^T \right) \mb{x} &= 
    \sum_{i \in \mathbb{P}_{OD}} p_i x_i^2 - \left(\sum_{i \in \mathbb{P}_{OD}} p_i x_i\right)^2 \\ &\leq \sum_{i \in \mathbb{P}_{OD}} p_i x_i^2 \leq \max_{i \in \mathbb{P}_{OD}} x_i^2 \leq \|\mb{x}\|^2 = 1.
\end{align}
Applying~\eqref{eq:specnormdef} to the PSD matrix $\diag([\mb{p}]_{OD}) - [\mb{p}]_{OD}([\mb{p}]_{OD})^T$ gives $\left\|\diag([\mb{p}]_{OD}) - [\mb{p}]_{OD}([\mb{p}]_{OD})^T\right\| \le 1$, hence $\left\|[\mb{S}]_{OD}\right\| \le d_{OD} \theta$. Taking the maximum over OD pairs, $\left\|\mb{S}\right\| \le \theta (\max d_{OD})$.
\end{proof}
We are now ready to prove the main spectral result on the reduced Jacobian of the logit mapping.

\begin{thm}\label{thm:spectrum} The eigenvalues of $\mathcal{K}(\mb{h})$ have the following properties:
\begin{enumerate}[itemsep=0pt, topsep=0pt]
    \item All eigenvalues of $\mathcal{K}(\mb{h})$ are real and non-positive.
    \item The maximum eigenvalue is zero.
    \item The minimum eigenvalue, $\lambda_{\min}$, is bounded below by $-\theta \left( \max\limits_{OD} d_{OD} \right) \|\mathcal{C}'(\mb{h})\|$.
\end{enumerate}
\end{thm}
\begin{proof}
By Lemma~\ref{lem:onesS}, zero is an eigenvalue of $\mathcal{K}(\mb{h})$. We now show that every other eigenvalue is real and non-positive.

Let $\lambda$ be any non-zero eigenvalue of $\mathcal{K}(\mb{h}) = -\mb{S} \mb{J}$. The non-zero eigenvalues of $-\mb{S} \mb{J}$ are identical to those of the matrix $-\mb{S}^{1/2} \mb{J} \mb{S}^{1/2}$. (This follows from the standard linear algebra property that the matrix products $\mb{A}\mb{B}$ and $\mb{B}\mb{A}$ share the same non-zero eigenvalues. Letting $\mb{A} = \mb{S}^{1/2}$ and $\mb{B} = \mb{S}^{1/2} \mb{J}$ yields $\mb{A}\mb{B} = \mb{S}\mb{J}$ and $\mb{B}\mb{A} = \mb{S}^{1/2} \mb{J} \mb{S}^{1/2}$.) The existence of real symmetric $\mb{S}^{1/2}$ is guaranteed by Lemma~\ref{lem:Spsd}. 

From Equation~\eqref{eq:3_codtderdef}, the Jacobian $\mb{J}$ is symmetric. Consequently, the matrix $\mb{S}^{1/2} \mb{J} \mb{S}^{1/2}$ is a real symmetric matrix, since $(\mb{S}^{1/2} \mb{J} \mb{S}^{1/2})^T = \mb{S}^{1/2} \mb{J}^T \mb{S}^{1/2} = \mb{S}^{1/2} \mb{J} \mb{S}^{1/2}$. By the spectral theorem, all its eigenvalues are real. Let $\mb{w} \in \mathbb{R}^n$ be an eigenvector with eigenvalue $\lambda \in \mathbb{R}$ such that
\begin{equation}
    -\mb{S}^{1/2} \mb{J} \mb{S}^{1/2} \mb{w} = \lambda \mb{w}.
\end{equation}
Pre-multiplying by $\mb{w}^T$ yields $-\mb{w}^T \mb{S}^{1/2} \mb{J} \mb{S}^{1/2} \mb{w} = \lambda \|\mb{w}\|^2$. Dividing both sides by the positive scalar $\|\mb{w}\|^2$, we isolate $\lambda$:
\begin{equation}
    \lambda = -\frac{1}{\|\mb{w}\|^2} \mb{w}^T \mb{S}^{1/2} \mb{J} \mb{S}^{1/2} \mb{w}. \label{eq:bigbiglamda}
\end{equation}
To simplify this quadratic form, we define $\mb{u} = \mb{S}^{1/2} \mb{w} \in \mathbb{R}^n$. Substituting into Equation~\eqref{eq:bigbiglamda} gives
\begin{equation}
    \lambda = -\frac{1}{\|\mb{w}\|^2} \mb{u}^T \mb{J} \mb{u}.
\end{equation}

Under our assumptions, each link cost function $\tau_l(a_l)$ is monotone non-decreasing and the cost of each link depends only on the flow on that link. Consequently, $\mb{J} = \mathcal{C}'(\mb{h}) = \mb{D}^T \mathcal{T}'(\mb{a})\, \mb{D}$ is symmetric PSD, and hence $\mb{u}^T \mb{J} \mb{u} \ge 0$ for any vector $\mb{u}$. It immediately follows that $\lambda \le 0$. Combined with Lemma~\ref{lem:onesS}, this establishes properties 1 and 2 of the theorem.

It remains to establish the lower bound on $\lambda_{\min}$ of $\mathcal{K}(\mb{h}) = -\mb{S}\mb{J}$. Since $-\mb{S}\mb{J}$ and $-\mb{S}^{1/2}\mb{J}\mb{S}^{1/2}$ share the same nonzero eigenvalues, $-\lambda_{\min}$ is the maximum eigenvalue of $\mb{S}^{1/2}\mb{J}\mb{S}^{1/2}$. Because $\mb{S}^{1/2} \mb{J} \mb{S}^{1/2}$ is a real symmetric PSD matrix, its spectral norm equals its largest eigenvalue, hence $-\lambda_{\min} = \|\mb{S}^{1/2} \mb{J} \mb{S}^{1/2}\|$. By submultiplicativity of the spectral norm and the identity $\|\mb{S}^{1/2}\|^2 = \|\mb{S}\|$ for PSD matrices,
\begin{equation}
    -\lambda_{\min} = \left\|\mb{S}^{1/2} \mb{J} \mb{S}^{1/2}\right\| \le \|\mb{S}^{1/2}\| \, \|\mb{J}\| \, \|\mb{S}^{1/2}\| = \|\mb{S}\| \, \|\mb{J}\|. \label{eq:lbeval}
\end{equation}

Applying Lemma~\ref{lem:Snorm} to bound $\|\mb{S}\|$, we obtain
\begin{equation}\label{eq:lambda_min_bound}
    \lambda_{\min} \ge - \, \theta \left( \max_{OD} d_{OD} \right) \|\mathcal{C}'(\mb{h})\|,
\end{equation}
which establishes property 3.
\end{proof}

\section{Convergence analysis of MSA in logit-based SUE}\label{sec:bound}

With the spectral properties of the Jacobian of the logit mapping in hand, we establish linear convergence of MSA with a constant step-size. We first establish the local convergence rate of MSA for logit-based SUE in Theorem~\ref{thm:msa_convergence}, following an approach similar to \cite{bargeraboyce06}. We show that, near equilibrium, MSA with a small constant step-size $s$ converges at a linear rate of $1 - s$, independent of the Jacobian. We also derive an upper bound $s_g$ such that any step-size $s < s_g$ guarantees linear convergence in a neighborhood of the SUE solution. Based on these results, we then propose a step-size rule in Algorithm~\ref{alg:4_stepsize} that ensures asymptotic linear convergence of MSA starting from any feasible solution. The method starts with a large step-size and reduces it adaptively, ensuring it ultimately becomes lower than $s_g$ close to SUE. We then establish convergence of this adaptive scheme through Lemma~\ref{lem:finite_resets} and Theorem~\ref{thm:global_convergence}.

\begin{thm}[Asymptotic linear convergence of MSA for logit-based SUE]\label{thm:msa_convergence}
Let $\hat{\mb{h}}$ denote the logit-based SUE path flow solution, and let $\{\mb{h}^k\}$ be the sequence of iterates generated by the MSA update
\begin{equation*}
    \mb{h}^{k+1} = (1-s)\mb{h}^k + s\mathcal{L}(\mb{h}^k),
\end{equation*}
with constant step-size $s \in (0,1)$. Assume the link cost functions are monotone non-decreasing. Then MSA converges linearly to $\hat{\mb{h}}$ for any step-size satisfying
\begin{equation}\label{eq:step_size_b}
    0 < s < \frac{2}{2 - \lambda_{\min}} =: s_g,
\end{equation}
where $\lambda_{\min}$ is the smallest (most negative) eigenvalue of $\mathcal{K}(\hat{\mb{h}})$. Moreover, applying the bounds on $\lambda_{\min}$ derived in Theorem~\ref{thm:spectrum}, convergence is guaranteed whenever
\begin{equation}\label{eq:step_size_c}
    0 < s < \frac{2}{2 + \theta \left( \max\limits_{OD} d_{OD} \right) \|\mathcal{C}'(\hat{\mb{h}})\|} \leq s_g.
\end{equation}
Under these conditions, the asymptotic rate of convergence of MSA is
\begin{equation}\label{eq:conv_rate_def}
    r^*(s) = 1 - s.
\end{equation}
\end{thm}

\begin{proof}
We define the error vector $\mb{e}^k = \mb{h}^k - \hat{\mb{h}}$ at iteration $k$ of MSA. Linearizing the logit mapping near the equilibrium $\hat{\mb{h}}$ gives $\mathcal{L}(\mb{h}^k) - \mathcal{L}(\hat{\mb{h}}) \approx \mathcal{L}'(\hat{\mb{h}})\mb{e}^k$. The MSA update then yields
\begin{equation}\label{eq:err_evol}
    \mb{e}^{k+1} \approx \bigl((1-s)\mb{I} + s\mathcal{L}'(\hat{\mb{h}})\bigr)\mb{e}^k.
\end{equation}
Since $\mb{e}^k$ is the difference of two feasible path flow vectors, Lemma~\ref{lem:Hh_zero} implies $\mathcal{H}'_{\mb{h}}(\hat{\mb{p}})\mb{e}^k = \mb{0}$, so the contribution of $\mathcal{H}'_{\mb{h}}$ vanishes and the linearization~\eqref{eq:err_evol} simplifies to
\begin{equation}\label{eq:error_evolution}
    \mb{e}^{k+1} \approx \bigl((1-s)\mb{I} + s\mathcal{K}(\hat{\mb{h}})\bigr)\mb{e}^k.
\end{equation}
This is a linear iteration whose convergence is governed by the spectral radius of the iteration matrix. The asymptotic rate of convergence is therefore
\begin{equation*}
    r^*(s) = \max_i |1 - s + s\lambda_i|,
\end{equation*}
where $\lambda_1, \ldots, \lambda_n$ are the eigenvalues of $\mathcal{K}(\hat{\mb{h}})$. By the results of Section~\ref{subsec:3_evals}, all eigenvalues of $\mathcal{K}(\hat{\mb{h}})$ are real and non-positive, and zero is an eigenvalue, so $\max_i \lambda_i = 0$.

For convergence at rate $1-s$, the spectral radius must satisfy $\left|1 - s + s\lambda_i\right| < 1-s$ for all eigenvalues. This condition is most restrictive at the most negative eigenvalue, requiring $s-1 < 1 - s + s\lambda_{\min}$, which yields the step-size bound \eqref{eq:step_size_b}. Substituting the bound on $\lambda_{\min}$ from Theorem~\ref{thm:spectrum} produces the more conservative but explicit bound \eqref{eq:step_size_c}.

Finally, since $\max_i \lambda_i = 0$, for any $s$ satisfying \eqref{eq:step_size_b} the spectral radius is attained at the zero eigenvalue, giving $r^*(s) = \left|1 - s + s \cdot 0\right| = 1 - s$.
\end{proof}

We note that unlike the general result of \cite{bargeraboyce06}, where the asymptotic convergence rate of MSA $r^*(s) \approx 1 - s(1 - \max_i\{\operatorname{Re}(\lambda_i)\})$ depends on the eigenvalues of the Jacobian even for small $s$, Theorem~\ref{thm:msa_convergence} shows that the convergence rate for logit-based SUE is exactly $1 - s$ regardless of network topology, cost functions, or demand levels, provided only that link cost functions are monotone non-decreasing and that $s$ is sufficiently small. The expressions in \eqref{eq:step_size_b} and \eqref{eq:step_size_c} quantify the notion of ``sufficiently small step-sizes'' from \cite{bargeraboyce06}; note that our bound ensuring a ``sufficiently small'' step-size does depend on the demand level and the cost Jacobian, but the convergence rate itself does not.

The bound in \eqref{eq:step_size_c} is pessimistic, because it uses the worst-case bound on the minimum eigenvalue. In practice, as demonstrated in our numerical results, much larger values of $s$ can still yield linear convergence at rate $1 - s$. 

We now present a procedure to determine the step-size sequence for MSA that ensures global convergence starting from any feasible solution and local linear convergence near the SUE. We refer to the procedure to generate this sequence as \texttt{ACS}, and Algorithm~\ref{alg:4_stepsize} describes the method used to generate the sequence. When the step-size sequence from Algorithm~\ref{alg:4_stepsize} is used within MSA, we refer to the algorithm as \texttt{MSA-ACS}. 

In Algorithm~\ref{alg:4_stepsize}, a first-in-first-out queue $\mb{\bar{g}}$ tracks the last $q$ values of the gap function. During the first $I_s$ iterations (lines 3--4), the step-size follows the standard MSA rule $\hat{s} = 1/k$. After iteration $I_s$, a constant step-size of $1/I_s$ is used as long as the gap decreases by more than $\epsilon$ relative to the previous $q$ iterations. We refer to the failure of this decrease, $\frac{\bar{g}_0 - \bar{g}_{q-1}}{\bar{g}_0} < \epsilon$, as the \textit{reset condition}: when it triggers, the step-size is reset to the reciprocal of the current iteration number $1/k$ and held constant thereafter (lines 6--7); the reset condition is then checked again over the next $q$ iterations and the step-size is reset again only if it triggers. In our implementation of Algorithm~\ref{alg:4_stepsize}, we set $\epsilon = 0.01$ and $q=3$. If $I_s$ is chosen so that the first constant step-size $\hat{s} = 1/I_s$ is sufficiently small and $\mb{h}^k$ is sufficiently close to equilibrium, the reset condition should never trigger. To monitor this, $\hat{s}$ can be compared with $1/I_s$ at the end of the assignment procedure to determine whether the step-size had to be reduced.

\begin{algorithm}
\caption{Adaptive constant step-size (\texttt{ACS}) \label{alg:4_stepsize}}
\begin{algorithmic}[1]
    \State \textbf{Input:} $k$ (iteration), $\mb{\bar{g}}$ (FIFO list of last $q$ gap values with $\bar{g}_{q-1}$ the newest entry), current step-size $\hat{s}$
    \State \textbf{Hyperparameters:} $I_s$ (initial MSA phase length), $\epsilon=0.01$ (relative gap-reduction threshold)
    \If{$k \leq I_s$}
        \State $\hat{s} \gets 1/k$
    \Else
        \If{$\frac{\bar{g}_0 - \bar{g}_{q-1}}{\bar{g}_0} < \epsilon$} \Comment{Reset condition}
            \State $\hat{s} \gets 1/k$ \Comment{Implicitly held constant otherwise}
        \EndIf
    \EndIf
    \State \textbf{Return:} $\hat{s}$ (step-size for iteration $k$)
\end{algorithmic}
\end{algorithm}

The idea in Algorithm~\ref{alg:4_stepsize} is to follow MSA with harmonic step-sizes until $\mb{h}^k$ 
is in a neighborhood close enough to equilibrium with $\hat{s} < s_g$, then maintain a 
constant step-size for linear convergence. Intuitively, since $s_g > 0$ and the harmonic step-sizes decrease 
monotonically, $\hat{s}$ must eventually fall below $s_g$. From that point onward, the 
algorithm can only cycle through finitely many constant-step phases before the iterates 
enter a neighborhood of equilibrium where linear convergence applies. We now formalize this argument below. First, we show in Lemma~\ref{lem:finite_resets} that if the reset condition triggers only finitely many times, the iterates converge to $\hat{\mb{h}}$. Theorem~\ref{thm:global_convergence} then uses this fact to show global convergence. 

\begin{lem}[Convergence under finite resets]\label{lem:finite_resets}
Assume $\epsilon \in (0, 1)$ and that the gap function $g(\mb{h})$ satisfies the properties:
\begin{enumerate}
    \item[(i)] $g(\mb{h}) \ge 0$ for all feasible $\mb{h}$, with $g(\mb{h}) = 0$ if and only if $\mb{h} = \hat{\mb{h}}$;
    \item[(ii)] $g$ is a continuous function.
\end{enumerate}
If the reset condition in Algorithm~\ref{alg:4_stepsize} triggers only finitely many times, then $\mb{h}^k \to \hat{\mb{h}}$ as $k \to \infty$.
\end{lem}

\begin{proof}
If the reset condition triggers only finitely many times, there exists an iteration $k'$, after which the reset condition never triggers. At any iteration $k > k'$, the FIFO queue $\mb{\bar{g}}$ entries satisfy $\bar{g}_0 = g_{k - q + 1}$ and $\bar{g}_{q-1} = g_k$. By the design of Algorithm~\ref{alg:4_stepsize}, if the reset condition does not hold true, we have $\frac{\bar{g}_0 - \bar{g}_{q-1}}{\bar{g}_0} \ge \epsilon$ at every check for $k > k'$, which is equivalent to $g_k \le (1 - \epsilon)\, g_{k - q + 1}$. So every gap value is at most $(1-\epsilon)$ times the value $q-1$ iterations earlier. 

Applying the inequality $g_k \le (1 - \epsilon)\, g_{k - q + 1}$ repeatedly from any large $k > k'$, we eventually land at some iteration at or before $k'$, with each step contributing a factor of $(1-\epsilon)$. Since the gap function values at iterations up to $k'$ are bounded, $g_k \to 0$ as $k \to \infty$. Given that path flows lie in a compact feasible region, the sequence $\{\mb{h}^k\}$ has at least one limit point. As $g_k \to 0$, assumption (ii) ensures every limit point satisfies $g = 0$. By assumption (i), the only such point is $\hat{\mb{h}}$. Hence $\mb{h}^k \to \hat{\mb{h}}$.
\end{proof}

One choice of gap function satisfying assumptions (i) and (ii) in Lemma~\ref{lem:finite_resets} is the fixed-point residual $\|\mathcal{L}(\mb{h}) - \mb{h}\|$, which measures how far the current flow is from being a fixed point of the logit mapping. Another gap function which satisfies these conditions is the average excess cost (AEC), defined as
\begin{equation*}
    \text{AEC} = \frac{\sum_{OD \in \Set{OD}}\sum_{i \in \Set{P}_{OD}} h_i \left(w_i - w_{\min}^{OD}\right)}{\sum_{OD \in \Set{OD}} d_{OD}},
\end{equation*}
where $w_i$ is the first-order derivative of the SUE objective function with respect to path flow $h_i$, and $w_{\min}^{OD} = \min_{i \in \Set{P}_{OD}} w_i$. Both AEC and residual of the fixed point mapping are non-negative, vanish only at $\hat{\mb{h}}$, and are continuous in $\mb{h}$ whenever the link cost functions are continuous. We also note that another commonly used gap function in SUE \citep{du2021faster, zhang2024distributed, wang2025enhancing}, the relative gap (RGAP), defined as
\begin{equation}
    \text{RGAP} = \frac{\sum_{OD \in \Set{OD}}\sum_{i \in \Set{P}_{OD}} h_i \left(w_i - w_{\min}^{OD}\right)}{\sum_{OD \in \Set{OD}}\sum_{i \in \Set{P}_{OD}} h_i \left|w_i\right|}, \label{eq:rgaap}
\end{equation}
does \textit{not} satisfy the continuity assumption, since the denominator can vanish as path flows vary. In our numerical results, we will use $\|\mathcal{L}(\mb{h}) - \mb{h}\|$ as the gap function for determining the step-size in Algorithm~\ref{alg:4_stepsize}.  In addition to satisfying the conditions of Lemma~\ref{lem:finite_resets}, this choice also avoids any additional floating-point operations per iteration to compute extra gap metrics, since both $\mb{h}$ and $\mathcal{L}(\mb{h})$ are already computed during MSA.  (The gap function used in Algorithm~\ref{alg:4_stepsize} to determine the step-size sequence need not be the same as the one used to evaluate solution quality at each iteration of MSA.  Our numerical results will use RGAP as the gap measure for evaluating solution quality to facilitate benchmarking, since most algorithms in the current literature use RGAP.)  We now use Lemma~\ref{lem:finite_resets} to establish global convergence of MSA with the step-size sequence from Algorithm~\ref{alg:4_stepsize}.

\begin{thm}[Global convergence of MSA with adaptive step-size]\label{thm:global_convergence}
Let $\hat{\mb{h}}$ denote the logit-based SUE path flow solution, and let $\{\mb{h}^k\}$ be the sequence of iterates generated by MSA with step-sizes from Algorithm~\ref{alg:4_stepsize} starting from any feasible initial solution. If the gap function satisfies the conditions in Lemma~\ref{lem:finite_resets}, then for any $\epsilon \in (0,1)$, $q > 1$, and $I_s > 1$, the sequence $\{\mb{h}^k\}$ converges to $\hat{\mb{h}}$.
\end{thm}

\begin{proof}
We want to show that for any arbitrarily small neighborhood $\mathbb{U}$ around $\hat{\mb{h}}$ where a constant step-size guarantees linear convergence, the sequence of iterates $\{\mb{h}^k\}$ generated by Algorithm~\ref{alg:4_stepsize} eventually enters $\mathbb{U}$. To see this, we observe that for the initial $I_s$ iterations in Algorithm~\ref{alg:4_stepsize} the step-size follows the harmonic sequence $s_k = 1/k$. For $k > I_s$, the algorithm sets the step-size based on reduction in the gap function. If the gap reduces by more than $\epsilon$ ($\frac{\bar{g}_0 - \bar{g}_{q-1}}{\bar{g}_0} \ge \epsilon$), the step-size remains constant. Otherwise, if the gap does not reduce enough ($\frac{\bar{g}_0 - \bar{g}_{q-1}}{\bar{g}_0} < \epsilon$), the algorithm shrinks the step-size by resetting it to $1/k$. Consequently, two cases arise in the algorithm's long run: either the reset condition triggers finitely many times, or it triggers infinitely many times. We show that in either case the iterates $\mb{h}^k$ must enter $\mathbb{U}$.

\textit{Case 1: The reset condition occurs finitely many times.} By Lemma~\ref{lem:finite_resets}, $\mb{h}^k \to \hat{\mb{h}}$, so the iterates eventually enter $\mathbb{U}$.

\textit{Case 2: The reset condition triggers infinitely many times.} Assume for the sake of contradiction that the iterates $\mb{h}^k$ never enter the neighborhood $\mathbb{U}$. Then, $\{s_k\}$ follows a sequence:
\begin{align}
\{s_k\} =
\Bigg\{
&\underbrace{1, \frac{1}{2}, \dots, \frac{1}{I_s}}_{\substack{\text{length } I_s}},\;
\underbrace{\frac{1}{I_s+1}, \dots, \frac{1}{I_s+1}}_{\substack{\text{length } N_1}},\;
\underbrace{\frac{1}{I_s+N_1+1}, \dots, \frac{1}{I_s+N_1+1}}_{\substack{\text{length } N_2}},\;
\ldots \notag\\
&\ldots,\;
\underbrace{\frac{1}{I_s+\sum_{m=1}^{j-1}N_m+1},
\dots,
\frac{1}{I_s+\sum_{m=1}^{j-1}N_m+1}}_{\substack{\text{length } N_j}},\;
\ldots
\Bigg\}
\end{align}
Here, $N_j$ represents the length of each constant step-size after the $j$-th reset condition is met. Since $\mb{h}^k$ never enters $\mathbb{U}$ and the feasible region is compact, by continuity the gap remains bounded below by some $g_{\min} > 0$ for all $k$. During each constant step-size phase, the algorithm only maintains the constant step-size as long as the gap reduces by at least a factor of $\epsilon$ every $q-1$ iterations. Because the gap is bounded above by some $g_{\max}$ and bounded below by $g_{\min} > 0$, and the gap reduces by at least a factor of $\epsilon$ every $q-1$ iterations, each $N_j$ is bounded by some finite $N_{\max}$ that depends only on $g_{\max}$, $g_{\min}$, $\epsilon$, and $q$.

The starting iteration of the $j$-th constant step-size is $i_j = I_s + \sum_{m=1}^{j-1} N_m + 1$. Since $N_j$ is bounded between $1$ and $N_{\max}$, we have $I_s + (j-1) + 1 \leq i_j \leq I_s + N_{\max}(j-1) + 1$ for all $j \geq 1$. Let $S_1$ denote the sum of the harmonic series up to $I_s$ and $S_2$ denote the sum of squares of the harmonic series up to $I_s$. We now verify that this infinite sequence satisfies the two criteria of convergence of step-sizes in MSA:
\begin{align}
    \sum_k s_k &= S_1 +  \sum_{j=1}^{\infty} \frac{N_j}{i_j} \ge S_1 + \sum_{j=1}^{\infty} \frac{1}{I_s + N_{\max}(j-1)+1} = \infty \label{eq:wrong1}\,,\\
    \sum_k s_k^2 &= S_2 + \sum_{j=1}^{\infty} \frac{N_j}{i_j^2} \le S_2 + \sum_{j=1}^{\infty} \frac{N_{\max}}{(I_s +j)^2} < \infty \label{eq:wrong2}\,.
\end{align}

Since both conditions are satisfied, the standard convergence result for MSA with diminishing step-sizes \citep{powell1982convergence} together with compactness of the feasible region implies $\mb{h}^k \to \hat{\mb{h}}$. This requires the iterates to eventually enter every neighborhood of $\hat{\mb{h}}$, including $\mathbb{U}$, contradicting the assumption that they never enter $\mathbb{U}$. Hence the iterates must enter $\mathbb{U}$ in this case as well.
\end{proof}

While the convergence of the algorithm is guaranteed for any $\epsilon \in (0,1)$ and $I_s > 1$, its practical efficiency is sensitive to these hyperparameters. If $\epsilon$ is chosen to be large, the required gap reduction over iterations ($\frac{\bar{g}_0 - \bar{g}_{q-1}}{\bar{g}_0} \ge \epsilon$) becomes difficult to achieve. This forces the algorithm to continuously trigger the reset condition, meaning $s_k$ is set as $1/k$ most of the time. Under these conditions, the step-size $s_k$ becomes very small. Consequently, even after the iterates enter the neighborhood $\mathbb{U}$, the asymptotic linear convergence rate $(1 - s)$ approaches one, reducing the speed of linear convergence. Conversely, if $\epsilon$ is extremely small, the reset condition is satisfied very rarely, causing a long delay before entering $\mathbb{U}$ where linear convergence applies. An ideal choice of $\epsilon$ balances the two, allowing the iterates to enter $\mathbb{U}$ quickly, after which the algorithm ceases resetting and fixes a constant step-size $s$, enabling rapid linear convergence at the rate $(1 - s)$. A similar tradeoff applies for $I_s$ as well. If $I_s$ is set too small, $\hat{s}$ may not have been reduced enough by the end of the harmonic phase. The algorithm then spends several constant step-size steps without $\mb{h}^k$ converging, until the reset condition evaluates as true. At that point, $\hat{s}$ drops to a much smaller value, possibly well below what is needed for linear convergence. On the other hand, if $I_s$ is set too large, the first $\hat{s}$ becomes too small to begin with.

\section{A quadratically convergent method for logit-based SUE}\label{sec:newton}

The linear convergence rate of $1 - \hat{s}$ established in Section~\ref{sec:bound} shows that MSA with a constant step-size avoids the sublinear tail convergence of diminishing step-size sequences. However, linear convergence may still be slow if $\hat{s} \approx 0$. In this section, we develop a quadratically convergent method that uses step-sizes inspired by Newton's method.  The straightforward implementation of Newton's method in convex optimization involves inverting the Hessian of the objective function, which in this case is large and dense.  Approximating the Hessian by its diagonal entries, a common alternative in quasi-Newton schemes, is reasonable in deterministic user equilibrium, but \cite{bekhor2005investigating} show that in the SUE setting, such an approximation discards a significant amount of information.  We sidestep these difficulties by working with the Jacobian of the logit mapping, rather than the Hessian of the SUE objective, exploiting its spectral properties to guarantee quadratic convergence in a neighborhood of the SUE solution. We further show that using an inexact Newton method, this strategy can be implemented efficiently in practice.

We start from the fact that SUE is the solution to the system $\mathcal{L}(\mb{h}) = \mb{h}$ and reformulate it as a root-finding problem for the mapping $\mathcal{F}(\mb{h}) := \mathcal{L}(\mb{h}) - \mb{h}$. We then apply Newton's method to find a root of $\mathcal{F}(\mb{h})$, which requires solving a linear system involving the Jacobian $\mathcal{F}'(\mb{h}) = \mathcal{L}'(\mb{h}) - \mb{I}$. We refer to this linear system as the \emph{full Newton system}, and show that it is singular and admits infinitely many solutions. Among these infinite solutions, we show that exactly one preserves demand conservation, and we construct a reduced Newton system which has a unique solution that coincides with this demand-preserving solution of the full Newton system. We verify that the Jacobian of the reduced system is nonsingular and Lipschitz continuous near the solution, which together guarantee locally quadratic convergence of the resulting Newton iterates.

Furthermore, we show that close to the SUE solution, the step obtained from the reduced Newton system to the current path flow vector corresponds to an updated path flow vector that satisfies both demand conservation and non-negativity. Consequently, no projection onto the feasible region is required, eliminating the need for the manifold optimization typically required in second-order path-based methods for SUE. To make the approach computationally practical, we solve the reduced Newton system inexactly using a Krylov subspace method, and show that the resulting iterates also preserve demand and retain local quadratic convergence.

We begin in Section~\ref{subsec:4_update} by formulating the step-size updates and show that it preserves demand conservation. In Section~\ref{subsec:4_convergence}, we prove that the method converges to SUE and that the convergence is quadratic near equilibrium. Finally, in Section~\ref{subsec:4_computation}, we develop an inexact Newton implementation that solves the reduced Newton system with a matrix-free representation. We show that this implementation preserves demand at every iteration and retains local quadratic convergence under an adaptive tolerance rule. Throughout this section, readers may refer to Appendix~\ref{app:braess_newton} for a worked example on the Braess network that illustrates the construction of the full Newton system, its singularity, the reduced Newton system, and the demand-preserving property of the reduced Newton system.

\subsection{A Newton-like step-size}\label{subsec:4_update}
The SUE solution is the fixed point of $\mathcal{F}(\mb{h}) = \mathcal{L}(\mb{h}) - \mb{h}$. Equivalently, to find the SUE solution, we need a root of $\mathcal{F}(\mb{h})$. The standard Newton step $\boldsymbol{\hat{\delta}}^k$ to find the root at iterate $\mb{h}^k$ solves
\begin{equation}
\mathcal{F}'(\mb{h}^k)\,\boldsymbol{\hat{\delta}}^k = -\mathcal{F}(\mb{h}^k). \label{eq:fullNewton}
\end{equation}
We refer to the system in \eqref{eq:fullNewton} as the full Newton system. Using the decomposition of the logit mapping from Section~\ref{sec:3_Jacobian}, $\boldsymbol{\hat{\delta}}^k$ satisfies
\begin{equation}
    \bigl(\mathcal{H}'_{\mb{h}}(\mb{p}^k) + \mathcal{K}(\mb{h}^k) - \mb{I}\bigr)\,\boldsymbol{\hat{\delta}}^k = -\mathcal{F}(\mb{h}^k)\,. \label{eq:4_fullnewton}
\end{equation}
However, the matrix $\mathcal{H}'_{\mb{h}}(\mb{p})$ contributes an eigenvalue of one to $\mathcal{L}'(\mb{h})$ for each OD pair. This is because, as shown in Section~\ref{subsec:3_Jacobian}, the within-OD blocks $[\mathcal{H}'_{\mb{h}}(\mb{p})]_{OD}$ of $\mathcal{H}'_{\mb{h}}(\mb{p})$ have every entry in row $i$ equal to $p_i$. Since the probabilities sum to one, the column sums of each within-OD block are all equal to one. So the all-ones vector $\mb{1}_{OD}$ is a left eigenvector with eigenvalue one. For the same eigenvector, Lemma~\ref{lem:onesS} shows $\mathcal{K}(\mb{h}^k)$ has an eigenvalue zero. Consequently, $\mathcal{L}'(\mb{h})$ has eigenvalue one, and $\mathcal{F}'(\mb{h}) = \mathcal{L}'(\mb{h}) - \mb{I}$ is singular. Hence, it would be impossible to invert $\mathcal{H}'_{\mb{h}}(\mb{p}^k) + \mathcal{K}(\mb{h}^k) - \mb{I}$ to obtain the step $\boldsymbol{\hat{\delta}}^k$ as the system of equations in~\eqref{eq:4_fullnewton} has either infinitely many solutions or no solutions. If there are infinite solutions, we are interested in solutions that satisfy feasibility.

We resolve this singularity by dropping $\mathcal{H}'_{\mb{h}}$ from the Newton system and solving
\begin{equation}
    (\mb{I} - \mathcal{K}(\mb{h}^k))\,\boldsymbol{\delta}^k = \mathcal{F}(\mb{h}^k)\label{eq:4_reduced}
\end{equation}
instead. We refer to the set of equations in~\eqref{eq:4_reduced} as the \textit{reduced Newton system}. We now establish two results that justify this reduction: in Theorem~\ref{thm:4_reduced_solves_full} we show that any solution of the reduced system also solves the original, full Newton system. We then show in Theorem~\ref{thm:4_demand_unique} that the reduced system has a unique solution which is the unique demand-preserving solution of the full Newton system.

\begin{thm}\label{thm:4_reduced_solves_full}
Every solution $\boldsymbol{\delta}^k$ of the reduced Newton system~\eqref{eq:4_reduced} also solves the full Newton system~\eqref{eq:4_fullnewton}.
\end{thm}
\begin{proof}
It suffices to show that $\mathcal{H}'_{\mb{h}}(\mb{p}^k)\,\boldsymbol{\delta}^k = \mb{0}$ for every solution to the reduced Newton system. First, we observe that $\sum_{i \in \Set{P}_{OD}} \mathcal{K}_{ij} = 0$ for every $j$ and every OD pair. This follows directly from $\mb{1}_{OD}^T [\mb{S}]_{OD} = \mb{0}^T$ (Equation~\eqref{eq:odDS} in Lemma~\ref{lem:onesS}) together with the block-diagonal structure of $\mb{S}$. Second, $\mathcal{F}(\mb{h})$ sums to zero within each OD pair, since both $\mathcal{L}(\mb{h})$ and $\mb{h}$ satisfy the demand constraints:
\begin{equation}
    \sum_{i \in \Set{P}_{OD}} \mathcal{F}(\mb{h})_i = \sum_{i \in \Set{P}_{OD}} \mathcal{L}(\mb{h})_i - \sum_{i \in \Set{P}_{OD}} h_i = d_{OD} - d_{OD} = 0\,.
\end{equation}
Summing both sides of $(\mb{I} - \mathcal{K})\,\boldsymbol{\delta}^k = \mathcal{F}$ over paths $i \in \Set{P}_{OD}$ yields
\begin{equation}
    \sum_{i \in \Set{P}_{OD}} \delta^k_i - \sum_j \left(\sum_{i \in \Set{P}_{OD}} \mathcal{K}_{ij}\right) \delta^k_j = \sum_{i \in \Set{P}_{OD}} \mathcal{F}_i\,.
\end{equation}
The second term on the left vanishes because $\sum_{i \in \Set{P}_{OD}} \mathcal{K}_{ij} = 0$ for every $j$, and the right-hand side vanishes because $\mathcal{F}(\mb{h})$ sums to zero within each OD pair.
Therefore $\sum_{i \in \Set{P}_{OD}} \delta^k_i = 0$ for every OD pair in any solution to the reduced Newton system.  

Furthermore, $\mathcal{H}'_{\mb{h}}(\mb{p}^k)\,\boldsymbol{\delta}^k = \mb{0}$ because each row $i$ of the within-OD block $\mathcal{H}'_{\mb{h}}(\mb{p}^k)$ consists of the constant $p_i$ in every column, so the $i$-th component of the product is $p_i \sum_{j \in \Set{P}_{OD}} \delta_j^k = 0$, and because the cross-OD blocks of $\mathcal{H}'_{\mb{h}}(\mb{p}^k)$ are zero.
Therefore $\boldsymbol{\delta}^k$ also satisfies the original system of equations in~\eqref{eq:4_fullnewton}.
\end{proof}

Theorem~\ref{thm:4_reduced_solves_full} shows that any solution of the reduced Newton system would also solve the full Newton system, which would imply that the (singular) full system has infinitely many solutions.  If that is the case, we must isolate a solution which is meaningful.  In addition to the fixed-point condition $\mb{h} = \mathcal{L}(\mb{h})$, the SUE solution must also satisfy the demand constraints. The next theorem shows that such a solution exists; indeed, that the reduced Newton system is in fact nonsingular, and its unique solution is demand-feasible.

\begin{thm}\label{thm:4_demand_unique}
The reduced Newton system~\eqref{eq:4_reduced} has a unique solution $\boldsymbol{\delta}^k$, and this solution is the only solution in the solution set of the full Newton system~\eqref{eq:4_fullnewton} that preserves demand, that is $\sum_{i \in \Set{P}_{OD}} \delta_i^k = 0$ for every OD pair.
\end{thm}
\begin{proof}
We first establish existence. We showed in Theorem~\ref{thm:spectrum} that all eigenvalues of $\mathcal{K}(\mb{h})$ are real and non-positive. Therefore, all eigenvalues of $\mb{I} - \mathcal{K}(\mb{h})$ are at least one, and the matrix is invertible. This guarantees a unique $\boldsymbol{\delta}^k$ from the reduced system, and as we now show, this $\boldsymbol{\delta}^k$ is the only solution of the full Newton system that preserves demand.

By Theorem~\ref{thm:4_reduced_solves_full}, $\boldsymbol{\delta}^k$ is at least one demand-preserving solution since $\sum_{i \in \Set{P}_{OD}} \delta^k_i = 0$ for every OD pair. It remains to show that $\boldsymbol{\delta}^k$ is the \emph{only} demand-preserving solution of the full system. Let $\boldsymbol{\delta}^* \neq \boldsymbol{\delta}^k$ be any solution of the full Newton system satisfying demand, and define the difference $\mb{v} := \boldsymbol{\delta}^* - \boldsymbol{\delta}^k$. Since both $\boldsymbol{\delta}^*$ and $\boldsymbol{\delta}^k$ satisfy the full Newton system, $\mb{v}$ lies in the null space of $\mathcal{F}'(\mb{h}^k)$, that is,
\begin{equation}
    \bigl(\mathcal{H}'_{\mb{h}}(\mb{p}^k) + \mathcal{K}(\mb{h}^k) - \mb{I}\bigr)\,\mb{v} = \mb{0}. \label{eq:proof_null}
\end{equation}

Looking at $\mathcal{H}'_{\mb{h}}(\mb{p}^k)\,\mb{v}$, we note that the cross-OD blocks of $\mathcal{H}'_{\mb{h}}(\mb{p}^k)$ are zero. So, the action of $\mathcal{H}'_{\mb{h}}(\mb{p}^k)$ on $\mb{v}$ decomposes into within-OD blocks, each acting on the corresponding $[\mb{v}]_{OD}$. Each row $i$ of the within-OD block of $\mathcal{H}'_{\mb{h}}$ has the form $p_i \mb{1}_{OD}^T$, so within each OD pair,
\begin{equation}
    [\mathcal{H}'_{\mb{h}}(\mb{p}^k)\,\mb{v}]_{OD} = [\mb{p}]_{OD}\,(\mb{1}_{OD}^T [\mb{v}]_{OD}) = [\mb{p}]_{OD} \cdot 0 = \mb{0}.
\end{equation}
The last step holds because both $\boldsymbol{\delta}^*$ and $\boldsymbol{\delta}^k$ preserve demand, so $\mb{v} = \boldsymbol{\delta}^* - \boldsymbol{\delta}^k$ has zero OD sums, that is, $\mb{1}_{OD}^T [\mb{v}]_{OD} = 0$ for every OD pair. Since this holds for every OD pair, $\mathcal{H}'_{\mb{h}}(\mb{p}^k)\,\mb{v} = \mb{0}$ globally. Substituting into equation~\eqref{eq:proof_null} eliminates the $\mathcal{H}'_{\mb{h}}$ term and leaves
\begin{equation*}
    (\mb{I} - \mathcal{K}(\mb{h}^k))\,\mb{v} = \mb{0}.
\end{equation*}
Since we assumed $\boldsymbol{\delta}^* \neq \boldsymbol{\delta}^k$, or equivalently, $\mb{v} \neq \mb{0}$, the matrix $\mb{I} - \mathcal{K}(\mb{h}^k)$ would have to be singular. However, we showed earlier that $\mb{I} - \mathcal{K}(\mb{h}^k)$ is invertible, yielding a contradiction. Hence $\mb{v} = \mb{0}$, which means $\boldsymbol{\delta}^* = \boldsymbol{\delta}^k$.
\end{proof}

\subsection{Convergence analysis}\label{subsec:4_convergence}

The Newton-like update takes the form
\begin{equation}\label{eq:4_newton_update}
    \mb{h}^{k+1} = \mb{h}^k + \boldsymbol{\delta}^k\,,
\end{equation}
where $\boldsymbol{\delta}^k$ solves the reduced Newton system~\eqref{eq:4_reduced} at $\mb{h}^k$. Although we have shown the Newton step $\boldsymbol{\delta}^k$ preserves demand feasibility in Theorem~\ref{thm:4_demand_unique}, it does not guarantee non-negativity. However, since logit probabilities are strictly positive, there always exists a neighborhood around the SUE where the Newton step is non-negative. To prove this rigorously, we assume that the Newton step is used inside a globally convergent algorithm such as \cite{du2021faster} or even MSA. The logit model assigns strictly positive probabilities to all paths, so at equilibrium, $\hat{h}_i = d_{OD}\, p_i > 0$ for every path $i$.

Let $\zeta = \min_{i \in \Set{P}} \hat{h}_i > 0$ be the minimum equilibrium path flow, and consider a neighborhood of $\hat{\mb{h}}$ of radius $\zeta / 2$. Within this neighborhood, every path flow satisfies $h_i^k > \zeta / 2 > 0$. Since $\mathcal{F}(\hat{\mb{h}}) = \mb{0}$ and $(\mb{I} - \mathcal{K})^{-1}$ is bounded (from the spectral analysis in Theorem~\ref{thm:spectrum}), we have $\boldsymbol{\delta}^k = (\mb{I} - \mathcal{K})^{-1} \mathcal{F}(\mb{h}^k) \to \mb{0}$ as $\mb{h}^k \to \hat{\mb{h}}$. In particular, for iterates sufficiently close to $\hat{\mb{h}}$, we have $\|\boldsymbol{\delta}^k\|_\infty < \zeta / 2$, so the next iterate $\mb{h}^k + \boldsymbol{\delta}^k$ remains strictly positive. Since path flows are strictly positive and demand is preserved, no projection is needed and the Newton step is feasible.

While the Newton step is feasible within this neighborhood, reaching the neighborhood from an arbitrary feasible starting solution requires a globally convergent algorithm. We therefore use the Newton step only in conjunction with such an algorithm. The natural questions then are: 1) when to switch to a Newton step, and 2) whether we can keep using the Newton step once we have switched. We show we can accept the Newton step and continue using it if (a) the path flows after taking the Newton step are feasible, and (b) the sufficient decrease condition
\begin{equation}\label{eq:4_armijo}
    \left\|\mathcal{F}\!\left(\mb{h}^k + \boldsymbol{\delta}^k\right)\right\| \leq (1 - \nu_1)\left\|\mathcal{F}(\mb{h}^k)\right\|
\end{equation}
is satisfied, with any $\nu_1 \in (0, 1)$. \cite{kelley2003solving} suggests $\nu_1 = 10^{-4}$, which we adopt. Otherwise, we switch back to the globally convergent algorithm. We now show that this entire procedure is globally convergent and achieves local quadratic convergence.

By standard unconstrained Newton's method convergence theory \citep{kelley2003solving}, there exists a neighborhood around the SUE in which the Newton step leads to quadratic convergence, provided the Jacobian is nonsingular and Lipschitz continuous. Nonsingularity on the reduced system follows from the spectral analysis in Theorem~\ref{thm:spectrum}, and Lipschitz continuity of the Jacobian follows from the fact that it is a composition of smooth functions on the bounded feasible set of path flows. We also showed that there exists a neighborhood around the SUE where feasibility is guaranteed when the Newton step is taken. Quadratic convergence holds in the smaller of these two neighborhoods, which we refer to as the ``convergence basin.''

Away from equilibrium, we rely on the globally convergent algorithm to drive the iterates into the convergence basin. However, there is no simple way to decide if the iterates are already in this basin. Hence, we take the Newton step whenever it is feasible and the sufficient decrease condition in Equation~\eqref{eq:4_armijo} holds. Accepting the Newton step once does not guarantee that subsequent iterates remain in the basin. We may therefore need to revert to the globally convergent algorithm even after a Newton step has been accepted. The iterates are nonetheless eventually guaranteed to enter the basin through the globally convergent algorithm, at which point quadratic convergence is assured. Taking a Newton step outside the basin does no harm, since by construction we accept it only when feasibility is maintained and the sufficient decrease condition holds.

\subsection{Inexact Newton iteration}\label{subsec:4_computation}

Although the Newton-like step described in Section~\ref{subsec:4_update} achieves local quadratic convergence, each iteration requires solving a linear system of size $n \times n$, which is computationally expensive. In particular, an exact implementation that inverts $\mb{I} - \mathcal{K}(\mb{h})$ at every iteration incurs a cost on the order of $n^3$. Moreover, explicitly forming $\mathcal{K}(\mb{h}) \in \mathbb{R}^{n \times n}$ is impractical for large-scale networks where the number of paths can exceed $10^6$. At such scales, even storing the matrix in computer memory becomes prohibitive, let alone solving the associated linear system. To address this issue, we adopt an inexact Newton approach. At each iteration, we approximate the action of the inverse of $\mb{I} - \mathcal{K}(\mb{h})$ within a Krylov subspace, without ever explicitly forming $\mathcal{K}(\mb{h})$. This approximation is progressively refined as the iterations proceed. We show that the resulting update preserves demand and retains local quadratic convergence.

From Section~\ref{subsec:3_evals}, we observe that $\mathcal{K}(\mb{h})$ can be factorized as $-\mb{S}\mb{J}$, where $\mb{S} = -\mathcal{H}'_{\mb{p}}(\mb{h})\,\mathcal{P}'(\mb{c})$ is block-diagonal, and $\mb{J} = \mathcal{C}'(\mb{h})$ factors through the link-path incidence matrix $\mb{D}$ and the diagonal matrix of marginal link costs $\mathcal{T}^{\prime}(\mb{a})$ as $\mb{D}^T \mathcal{T}^{\prime}(\mb{a})\,\mb{D}$. As a result, the matrix $\mb{I} - \mathcal{K}(\mb{h})$ never needs to be formed explicitly. Instead, its product with any vector $\mb{v}$ can be computed via
\begin{equation}\label{eq:4_matvec}
    (\mb{I} - \mathcal{K}(\mb{h}))\,\mb{v} = \mb{v} + \mb{S}\,(\mb{D}^T\,\mathcal{T}^{\prime}(\mb{a})\,\mb{D}\,\mb{v})\,.
\end{equation}
Evaluating \eqref{eq:4_matvec} from right to left involves only sparse matrix-vector multiplications (with $\mb{D}$ and $\mb{D}^T$), block-diagonal multiplications (with $\mb{S}$), and element-wise operations (with $\mathcal{T}^{\prime}(\mb{a})$), so the computational cost scales with the number of nonzeros rather than with $n^2$.

With this matrix-free representation, we solve the reduced Newton system using the generalized minimal residual method (GMRES) \citep{saad1986gmres}. GMRES is an iterative Krylov subspace method for non-symmetric linear systems that requires only matrix-vector products, making it well suited to our setting. By the Cayley-Hamilton theorem,  $(\mb{I} - \mathcal{K}(\mb{h}^k))^{-1}$ can be written as a polynomial in $\mb{I} - \mathcal{K}(\mb{h}^k)$. Consequently, $(\mb{I} - \mathcal{K}(\mb{h}^k))^{-1}\,\mathcal{F}(\mb{h}^k)$ can be expressed as a linear combination of
\begin{align}
\mathcal{F}(\mb{h}^k),\ 
(\mb{I} - \mathcal{K}(\mb{h}^k))\,\mathcal{F}(\mb{h}^k),\ 
(\mb{I} - \mathcal{K}(\mb{h}^k))^2\,\mathcal{F}(\mb{h}^k),\ \ldots,\ 
(\mb{I} - \mathcal{K}(\mb{h}^k))^{n-1}\,\mathcal{F}(\mb{h}^k). \label{eq:krylov}
\end{align}
The subspace spanned by the vectors in \eqref{eq:krylov} forms the Krylov subspace generated by $\mathcal{F}(\mb{h}^k)$ and $\mb{I} - \mathcal{K}(\mb{h}^k)$. At each GMRES iteration, starting from one dimension, the Krylov subspace is expanded by one dimension, and the method selects the best linear combination in this subspace that minimizes the residual norm of the linear system. If the subspace reaches its maximal dimension, the exact solution to the reduced Newton system is recovered.

In our case, however, we terminate GMRES once a prescribed relative tolerance $\eta_k$ is met. This yields an inexact Newton step $\boldsymbol{\delta}^k$ satisfying
\begin{equation}\label{eq:forcing}
    \frac{\left\|(\mb{I} - \mathcal{K}(\mb{h}^k))\,\boldsymbol{\delta}^k - \mathcal{F}(\mb{h}^k)\right\|}{\left\|\mathcal{F}(\mb{h}^k)\right\|} \leq \eta_k.
\end{equation}
This raises two natural questions: whether the step still preserves demand, and whether local quadratic convergence is retained. We answer both in the affirmative in Proposition~\ref{lem:feasibility} and Proposition~\ref{lem:adaptive_tolerance}, respectively.

\begin{prp}\label{lem:feasibility}
At any iteration $k$ of a globally convergent traffic assignment solver, GMRES initialized at the zero vector produces an iterate $\boldsymbol{\delta}^k$ that preserves demand, that is, $\sum_{j \in \Set{P}_{OD}} \delta_j^k = 0$ for every OD pair.
\end{prp}
\begin{proof}
At iteration $k$ of traffic assignment, GMRES initialized at $\mb{0}$ produces the iterate $\boldsymbol{\delta}^k$ that minimizes $\|\mathcal{F}(\mb{h}^k) - (\mb{I} - \mathcal{K}(\mb{h}^k))\,\boldsymbol{\delta}^k\|$ over the Krylov subspace
\begin{equation}\label{eq:krylov1}
    \mathbb{K}_b := \mathrm{span}\bigl\{\mathcal{F}(\mb{h}^k),\ (\mb{I} - \mathcal{K}(\mb{h}^k))\,\mathcal{F}(\mb{h}^k),\ \ldots,\ (\mb{I} - \mathcal{K}(\mb{h}^k))^{b-1}\,\mathcal{F}(\mb{h}^k)\bigr\},
\end{equation}
where $b$ is the smallest dimension at which GMRES attains the relative tolerance $\eta_k$. Since $\boldsymbol{\delta}^k \in \mathbb{K}_b$, it suffices to show that every basis vector in~\eqref{eq:krylov1} has zero OD sums.

We proceed by induction. For the base case, $\mathcal{F}(\mb{h}^k)$ has zero OD sums because it is the difference of two feasible path flows. For the inductive step, suppose $\mb{v}$ has zero OD sums, that is, $\mb{1}_{OD}^T [\mb{v}]_{OD} = 0$ for every OD pair. From Equation~\eqref{eq:odDS} in the proof of Lemma~\ref{lem:onesS}, $\mb{1}_{OD}^T [\mb{S}]_{OD} = \mb{0}^T$ for each block, so
\begin{equation*}
    \mb{1}_{OD}^T [(\mb{I} - \mathcal{K}(\mb{h}^k))\,\mb{v}]_{OD} = \mb{1}_{OD}^T [\mb{v}]_{OD} + \mb{1}_{OD}^T [\mb{S}]_{OD} [\mb{J}\,\mb{v}]_{OD} = 0 + \mb{0}^T [\mb{J}\,\mb{v}]_{OD} = 0.
\end{equation*}
Hence $(\mb{I} - \mathcal{K}(\mb{h}^k))\,\mb{v}$ also has zero OD sums, and by induction every spanning vector in~\eqref{eq:krylov1} does too. The result follows.
\end{proof}

In practice, GMRES does not work directly with the spanning sequence defining $\mathbb{K}_b$. Instead, it uses the Arnoldi process to build an orthonormal basis for the same subspace, which yields a much better estimate of the inverse in fewer dimensions. Since the orthonormal basis spans the same subspace, the conclusion $\boldsymbol{\delta}^k \in \mathbb{K}_b$ is independent of the spanning vectors used.

\begin{prp}\label{lem:adaptive_tolerance}
Setting the GMRES relative tolerance as
\begin{equation}\label{eq:adaptive_tol}
    \eta_k = \min\!\left(\eta_{\text{tol}},\ \nu_2 \left\|\mathcal{F}(\mb{h}^k)\right\|\right),
\end{equation}
with $\eta_{\text{tol}} < 1$ and $\nu_2 > 0$, ensures that the inexact Newton method converges locally quadratically.
\end{prp}

\begin{proof}
By \citet{dembo1982inexact}, once the iterates enter a neighborhood of the solution, the inexact Newton method converges at least quadratically whenever $\mathcal{F}'$ is Lipschitz continuous and the residual of the linear system scales with $\|\mathcal{F}(\mb{h}^k)\|^2$. Lipschitz continuity of $\mathcal{F}'$ was already established in Section~\ref{subsec:4_convergence}. For the residual scaling, we observe that the condition~\eqref{eq:forcing} bounds the residual by $\eta_k \|\mathcal{F}(\mb{h}^k)\|$, so it suffices to show that $\eta_k$ itself scales with $\|\mathcal{F}(\mb{h}^k)\|$ in the asymptotic regime. The choice in~\eqref{eq:adaptive_tol} satisfies this: as $\|\mathcal{F}(\mb{h}^k)\| \to 0$, the second argument of the minimum dominates and $\eta_k = \nu_2 \left\|\mathcal{F}(\mb{h}^k)\right\|$. Substituting back, the residual is bounded by $\nu_2 \|\mathcal{F}(\mb{h}^k)\|^2$, which gives the required quadratic scaling.
\end{proof}

We summarize the full Newton-based procedure in Algorithm~\ref{alg:newton_step} and refer to this Newton step procedure as \texttt{Newton}. The procedure depends on three hyperparameters: the line-search parameter $\nu_1$, the initial GMRES tolerance $\eta_{\text{tol}}$, and the multiplier $\nu_2$ that controls when the residual-dependent tolerance regime activates. In our implementation, we set $\nu_1 = 10^{-4}$ following the standard choice in \cite{kelley2003solving}, $\eta_{\text{tol}} = 10^{-2}$ as a conventional starting tolerance, and $\nu_2 = 10^3$. The algorithm is largely insensitive to these choices. Any small positive value of $\nu_1$ works, and any positive value of $\nu_2$ yields quadratic convergence once the iterates enter the asymptotic regime; larger $\nu_2$ simply delays the tightening. The only practical guidance is to avoid choosing $\eta_{\text{tol}}$ or $\nu_2$ too small, since tightening the GMRES tolerance early is wasteful as the Newton step only delivers quadratic convergence close to equilibrium, where feasibility is also guaranteed.
\begin{algorithm}
\caption{Newton step with GMRES (\texttt{Newton})}
\label{alg:newton_step}
\begin{algorithmic}[1]
\State \textbf{Input:} Current iterate $\mb{h}^k$
\State \textbf{Hyperparameters:} $\eta_{\text{tol}} = 10^{-2}$, $\nu_1 = 10^{-4}$, $\nu_2 = 10^3$
\State $\mathcal{F}(\mb{h}^k) \gets \mathcal{L}(\mb{h}^k) - \mb{h}^k$
\State $\eta_k \gets \min(\eta_{\text{tol}},\ \nu_2 \|\mathcal{F}(\mb{h}^k)\|)$
\State Solve $(\mb{I} - \mathcal{K}(\mb{h}^k))\,\boldsymbol{\delta}^k = \mathcal{F}(\mb{h}^k)$ iteratively using GMRES with tolerance $\eta_k$
\State $\mb{h}_{\text{trial}} \gets \mb{h}^k + \boldsymbol{\delta}^k$
\If{$\mb{h}_{\text{trial}}$ is infeasible \textbf{or} $\left\|\mathcal{F}(\mb{h}_{\text{trial}})\right\| > (1 - \nu_1)\left\|\mathcal{F}(\mb{h}^k)\right\|$}
\State \textbf{Return:} ``step rejected'' indicating the need to revert to a globally convergent algorithm.
\EndIf
\State $\mb{h}^{k+1} \gets \mb{h}_{\text{trial}}$
\State \textbf{Return:} Next iterate $\mb{h}^{k+1}$
\end{algorithmic}
\end{algorithm}

In an efficient implementation of Algorithm~\ref{alg:newton_step}, the only matrix needed to be stored in computer memory is the link-path incidence matrix $\mb{D} \in \mathbb{R}^{|\Set{L}| \times |\Set{P}|}$, which is used in computation of $\mathcal{L}$. Because each path traverses only a small subset of the network's links, $\mb{D}$ is  sparse and can be stored in a compressed sparse row format, which records only the nonzero entries and their indices. All other quantities, including path flows, probabilities, costs, and marginal link costs, can be stored as one-dimensional arrays, with a separate index array marking the boundaries between OD pairs.  This keeps the memory footprint proportional to the number of nonzero entries in $\mb{D}$, making the method feasible for networks with millions of paths on standard hardware.


\section{Computational results}\label{sec:results}

We present several computational results to complement our theoretical derivations. These experiments serve two primary purposes. First, in Subsection~\ref{subsec:res_MSA}, we test the convergence rate of MSA with a small constant step-size $s$ near equilibrium and verify that the rate is exactly $1-s$. Second, in Subsection~\ref{subsec:res_newton}, we benchmark our proposed methods, \texttt{MSA-ACS} and \texttt{Newton}, against MSA with harmonic step-sizes and the recent method of \cite{du2021faster}.

Because these subsections investigate different quantities, the experiments are conducted on networks of varying sizes. Verifying the theoretical results in Subsection~\ref{subsec:res_MSA} requires detailed knowledge of $\mathcal{K}$, including its eigenvalues. Since $\mathcal{K}$ has dimensions $n \times n$, where $n$ is the number of paths, these tests are restricted to smaller networks. In contrast, the benchmarking experiments in Subsection~\ref{subsec:res_newton} leverage the matrix-free approach of Subsection~\ref{subsec:4_computation}. This avoids forming $\mathcal{K}$ explicitly and allows us to scale to larger networks.

All experiments use network data and trip tables from \cite{transportationNetworks}. We classify the tested networks into three groups: small-scale, medium-scale, and large-scale. The small-scale networks include Sioux Falls, Berlin-Mitte-Center (BMC), Eastern-Massachusetts (EMA), and Anaheim. The medium-scale networks include Winnipeg Asymmetric, Berlin Center, and Chicago Sketch. The large-scale networks include Chicago Regional and Philadelphia. In all experiments, the initial solution was taken as the path flow obtained from logit network loading under free-flow conditions and RGAP, as defined in Equation~\eqref{eq:rgaap}, was used as the convergence criterion. All algorithms were implemented in Python 3.12.5 and run on a Dell Precision 3680 workstation with an Intel Core i9-14900K processor and 128~GB of RAM. The source code of all our algorithms is available at \url{https://github.com/spartalab/logit-SUE}.

The number of paths chosen for each OD pair is shown in Table~\ref{tab:4_res}.
For small and medium-scale networks, the path set was generated using Yen's $k$-shortest-paths algorithm. For larger networks, Yen's algorithm proved computationally slow, so we instead used a penalty-based approach. For each origin, we first compute a single-source shortest path tree to obtain the shortest path to every destination. We then repeat the following procedure up to $k-1$ times. For every destination, we identify the most recently added path, select each of its links independently with probability $0.5$, and increase their cost by $50\%$. We then recompute the shortest path tree on the penalized graph and add any newly discovered unique paths. The process terminates early if no new unique paths are found for any destination. This procedure is repeated for every origin.

We use two metrics to evaluate the generated path sets. The first is the coefficient of variation of path costs per OD pair, averaged across all OD pairs, which measures whether the paths in the set are competitive with one another. The second is the mean edge Jaccard index across all OD pairs, which measures whether the paths are sufficiently distinct or differ only by a few links. Table~\ref{tab:4_res} summarizes the key attributes of all networks, along with the subsections in which they appear.

\begin{table}
    \centering
    \small
    \begin{tabular}{lrrrrrrccccc}
        \toprule
        \textbf{Network} & \textbf{OD} & \textbf{Links} & \textbf{Nodes} & $\mb{k}$ & \textbf{Paths} & \textbf{Proc.} & $\overline{\text{CV}}$ & $\overline{J}$ & \textbf{S1} & \textbf{S2} \\
        \midrule
        \multicolumn{11}{l}{\textbf{Small-scale}} \\
        Sioux Falls               &       528 &    76 &    24 & 20 &        10\,560 & Yen     & 0.210 & 0.164 & \checkmark & \checkmark \\
        BMC             &     1\,260 &   871 &   398 & 20 &        25\,188 & Yen     & 0.116 & 0.428 & \checkmark & \checkmark \\
        EMA              &     1\,113 &   258 &    74 & 20 &        21\,824 & Yen     & 0.142 & 0.292 & \checkmark & \checkmark \\
        Anaheim          &     1\,406 &   914 &   416 & 20 &        28\,120 & Yen     & 0.064 & 0.455 & \checkmark & \checkmark \\
        \midrule
        \multicolumn{11}{l}{\textbf{Medium-scale}} \\
        Winnipeg Asymmetric         &     4\,345 & 2\,535 &   1\,057 & 20 &        86\,900 & Yen     & 0.067 & 0.385 &            & \checkmark \\
        Chicago Sketch   &    93\,135 & 2\,950 &   933 & 20 &  1\,862\,700 & Yen     & 0.048 & 0.434 &            & \checkmark \\
        Berlin Center    &    49\,688 & 28\,376 & 12\,981 & 20 &     993\,760 & Yen     & 0.054 & 0.573 &            & \checkmark \\
        \midrule
        \multicolumn{11}{l}{\textbf{Large-scale}} \\
        Philadelphia     & 1\,149\,795 & 40\,003 & 13\,389 & 5  &  5\,609\,985 & Penalty & 0.120 & 0.130 &            & \checkmark \\
        Chicago Regional & 2\,296\,227 & 39\,018 & 12\,982 & 5  & 11\,190\,758 & Penalty & 0.205 & 0.155 &            & \checkmark \\
        \bottomrule
    \end{tabular}
    \caption{Networks used in the study and properties of the generated path sets. 
    OD: number of OD pairs with positive demand; 
    Links/Nodes: network size; 
    ${k}$: target number of paths per OD pair; 
    Paths: total number of paths in the generated set (summed over OD pairs); 
    Proc.: path-set generation procedure (Yen's $k$-shortest-paths algorithm or link-penalty method); 
    $\overline{\text{CV}}$: mean coefficient of variation of free-flow path costs across OD pairs; 
    $\overline{J}$: mean pairwise edge Jaccard index across OD pairs; 
    S1/S2: network is included in the experiments of Section~\ref{subsec:res_MSA} / Section~\ref{subsec:res_newton}. 
    Networks are grouped by size into small, medium, and large scale instances.}
    \label{tab:4_res}
\end{table}

\subsection{Analysis of MSA with constant step-size}\label{subsec:res_MSA}

We tested the convergence rate of MSA in logit-based SUE on four networks: Sioux Falls, BMC, EMA, and Anaheim. We first compare MSA with the \texttt{ACS} step-size rule against the harmonic step-size sequence $\{1/k\}$. We then compute the empirical convergence rate of MSA with \texttt{ACS} step-sizes and analyze the relationship between the maximum step-size admitting linear convergence and the eigenvalues of $\mathcal{K}$.

Table~\ref{tab:4_runtime} compares the number of MSA iterations required to reach different gap levels under the two step-size rules on the tested networks with $\theta = 0.5$. Here, Rule~1 refers to the harmonic step-size, and Rule~2 refers to the \texttt{ACS} step-size rule in Algorithm~\ref{alg:4_stepsize} with $I_s = 10$ and $\epsilon = 0.01$. The results show that the \texttt{ACS} step-size rule leads to significantly faster convergence. Moreover, under Rule~2, the number of iterations required to reduce the gap by one order of magnitude remains roughly constant, which is characteristic of linear convergence. Figure~\ref{fig:convergence_small} shows the evolution of RGAP in log scale over iterations for three networks under both rules for $\theta = 1.0$.
\begin{table}
    \centering
    \begin{tabular}{r  c c c c c c c c}
        \toprule
        \textbf{Gap} & \multicolumn{2}{c}{\textbf{Sioux Falls}} & \multicolumn{2}{c}{\textbf{BMC}} & \multicolumn{2}{c}{\textbf{EMA}} & \multicolumn{2}{c}{\textbf{Anaheim}} \\
        \cmidrule(lr){2-3} \cmidrule(lr){4-5} \cmidrule(lr){6-7} \cmidrule(lr){8-9}
        RGAP & Rule 1 & Rule 2 & Rule 1 & Rule 2 & Rule 1 & Rule 2 & Rule 1 & Rule 2 \\
        \midrule
        $10^{-1}$ & 23 & 19 & 3 & 3 & 1 & 1 & 1 & 1 \\
        $10^{-2}$ & 531 & 50 & 4 & 4 & 4 & 4 & 4 & 4 \\
        $10^{-3}$ & $>$1000 & 77 & 14 & 14 & 9 & 9 & 11 & 11 \\
        $10^{-4}$ & $>$1000 & 102 & 131 & 36 & 47 & 26 & 70 & 30 \\
        $10^{-5}$ & $>$1000 & 126 & $>$1000 & 59 & 327 & 44 & 630 & 51 \\
        $10^{-6}$ & $>$1000 & 149 & $>$1000 & 82 & $>$1000 & 65 & $>$1000 & 72 \\
        $10^{-7}$ & $>$1000 & 172 & $>$1000 & 105 & $>$1000 & 86 & $>$1000 & 94 \\
        $10^{-8}$ & $>$1000 & 195 & $>$1000 & 127 & $>$1000 & 107 & $>$1000 & 116 \\
        $10^{-9}$ & $>$1000 & 218 & $>$1000 & 150 & $>$1000 & 129 & $>$1000 & 138 \\
        $10^{-10}$ & $>$1000 & 241 & $>$1000 & 172 & $>$1000 & 151 & $>$1000 & 160 \\
        \bottomrule
    \end{tabular}
    \caption{Number of iterations of MSA required to reach desired RGAP level using Rule 1 and Rule 2 for step-size selection across different networks ($\theta = 0.5$). Rule 1 sets step-size as the reciprocal of iteration, while Rule 2 sets step-size as per Algorithm~\ref{alg:4_stepsize}.}
    \label{tab:4_runtime}
\end{table}

\begin{figure}
    \centering
    \begin{subfigure}[b]{0.32\textwidth}
        \centering
        \includegraphics[width=\textwidth]{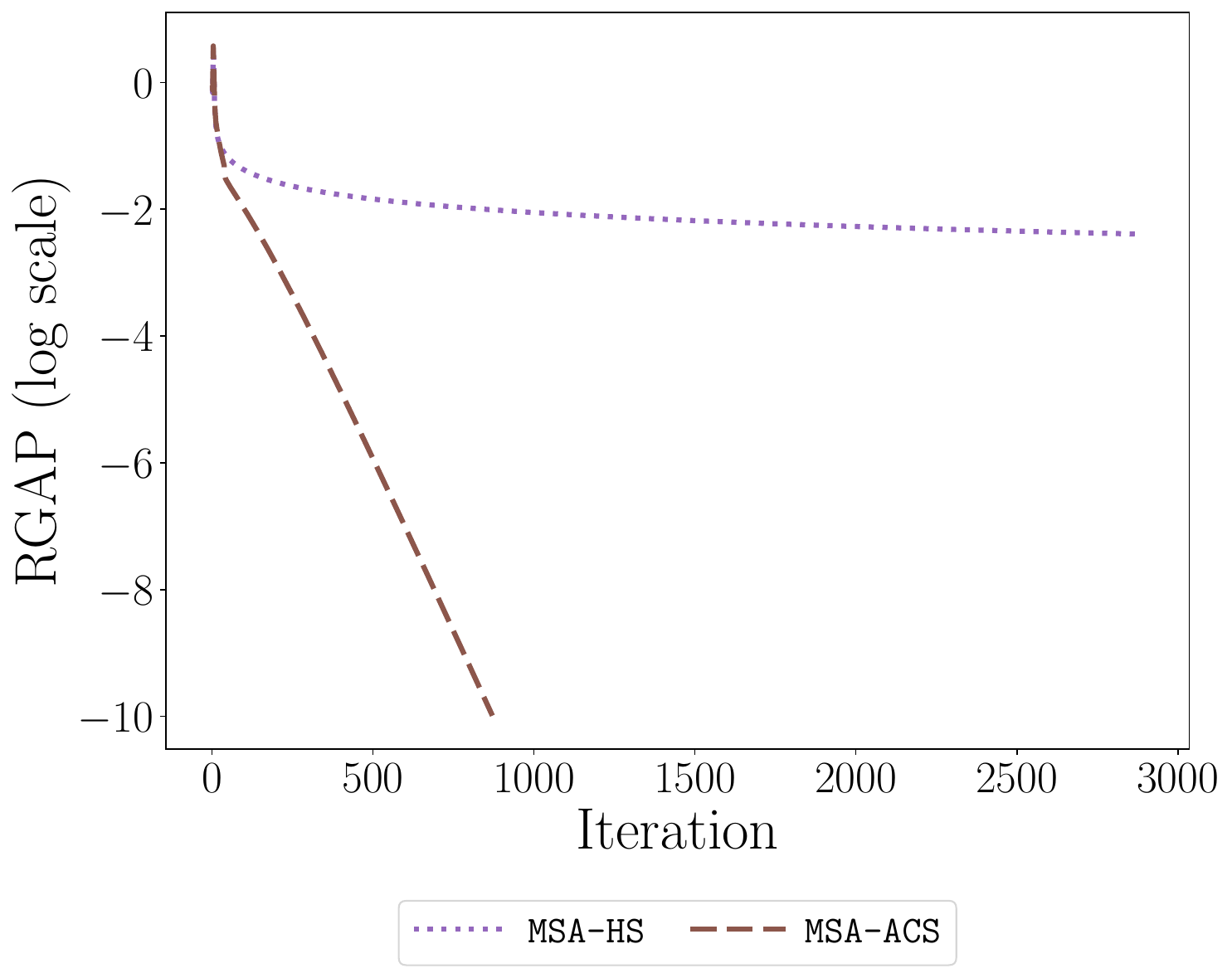}
        \caption{Sioux Falls}
        \label{fig:conv_sioux_falls}
    \end{subfigure}
    \hfill
    \begin{subfigure}[b]{0.32\textwidth}
        \centering
        \includegraphics[width=\textwidth]{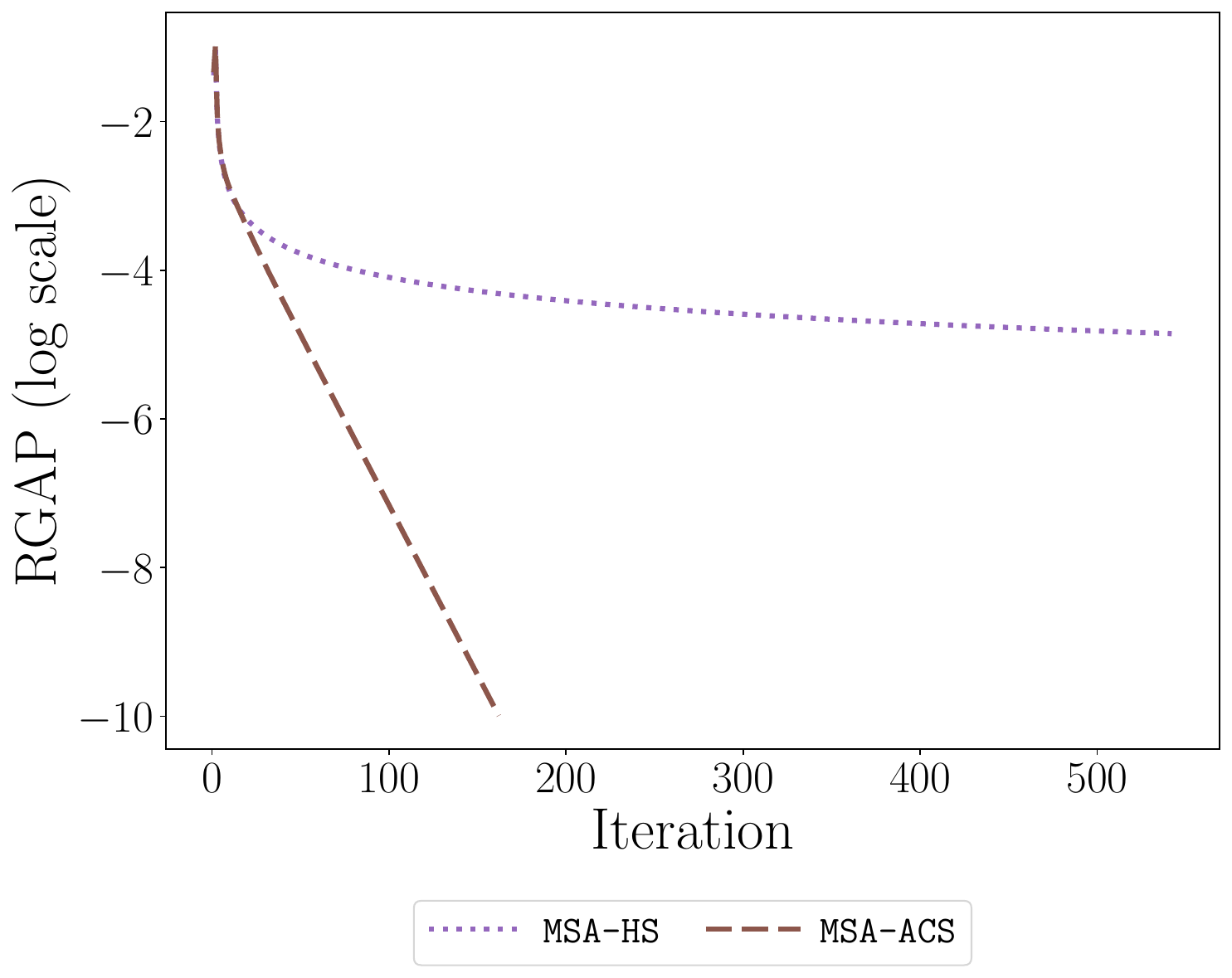}
        \caption{Anaheim}
        \label{fig:conv_Berlin-Mitte-Center}
    \end{subfigure}
    \hfill
    \begin{subfigure}[b]{0.32\textwidth}
        \centering
\includegraphics[width=\textwidth]{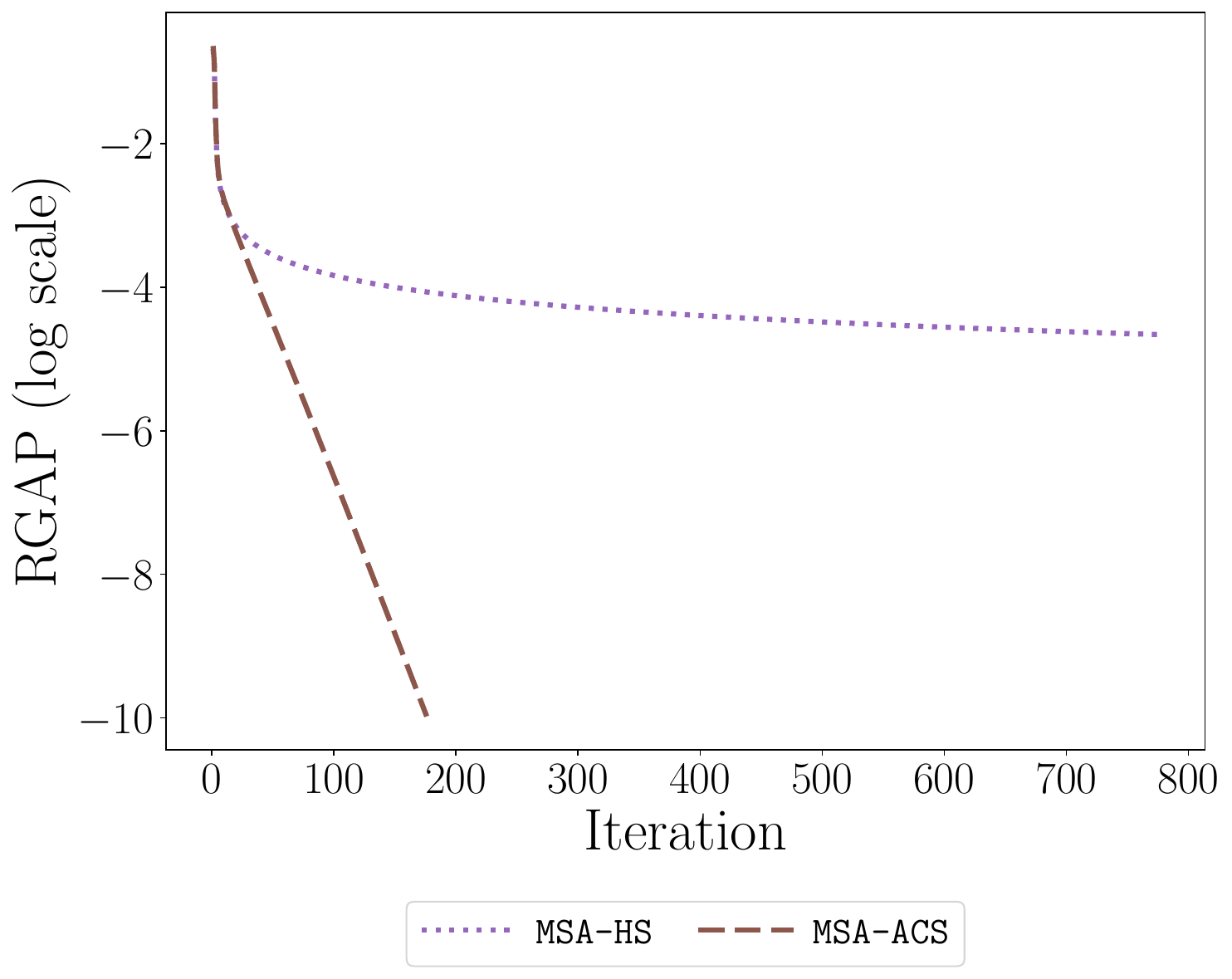}
        \caption{BMC}
        \label{fig:conv_ema}
    \end{subfigure}
    \caption{Convergence of RGAP (log scale) versus iteration for MSA with harmonic step-sizes and \texttt{MSA-ACS} on three small-scale networks ($\theta = 1.0$, $I_s = 10$).}
    \label{fig:convergence_small}
\end{figure}

We next examine the convergence rate of MSA with the \texttt{ACS} step-size. By Theorem~\ref{thm:msa_convergence}, the convergence rate near equilibrium should be approximately $1 - \hat{s}$, where $\hat{s}$ is the final constant step-size selected by Algorithm~\ref{alg:4_stepsize}. For sufficiently large $I_s$, such that the reset condition in Algorithm~\ref{alg:4_stepsize} is never triggered, we expect $\hat{s} = 1/I_s$. We test the convergence rate of MSA for $I_s \in \{5, 10, 20, 30\}$ and $\theta \in \{0.5, 1.0, 1.5\}$ by computing the ratio
\begin{align}
    \text{Rate}^{k} = \frac{\|\mb{h}^{k} - \hat{\mb{h}}\|}{\|\mb{h}^{k-1} - \hat{\mb{h}}\|} \label{eq:conv}
\end{align}
at each iteration $k$. The equilibrium solution $\hat{\mb{h}}$ in Equation~\eqref{eq:conv} is defined as the path flow iterate satisfying $\text{RGAP} \leq 10^{-10}$. As iterations progress, $\text{Rate}^{k}$ stabilizes, indicating linear convergence of $\mb{h}^k$. Table~\ref{tab:msa_boundgap_convergence} summarizes the observed convergence rates for all tested networks. For each combination of $I_s$ and $\theta$, the table reports three quantities: (1) $1 - 1/I_s$, the expected rate assuming the reset condition in Algorithm~\ref{alg:4_stepsize} is never triggered, (2) $1 - \hat{s}$, the expected rate based on the final constant step-size selected by Algorithm~\ref{alg:4_stepsize}, and (3) the observed rate, computed as the average of $\text{Rate}^k$ over the final $25$ iterations before RGAP reaches $10^{-9}$. Figure~\ref{fig:convergence_rate_sf} illustrates this procedure for the Anaheim network across three values of $\theta$.

\begin{figure}
    \centering
    \begin{subfigure}{0.32\textwidth}
        \centering\includegraphics[width=\linewidth]{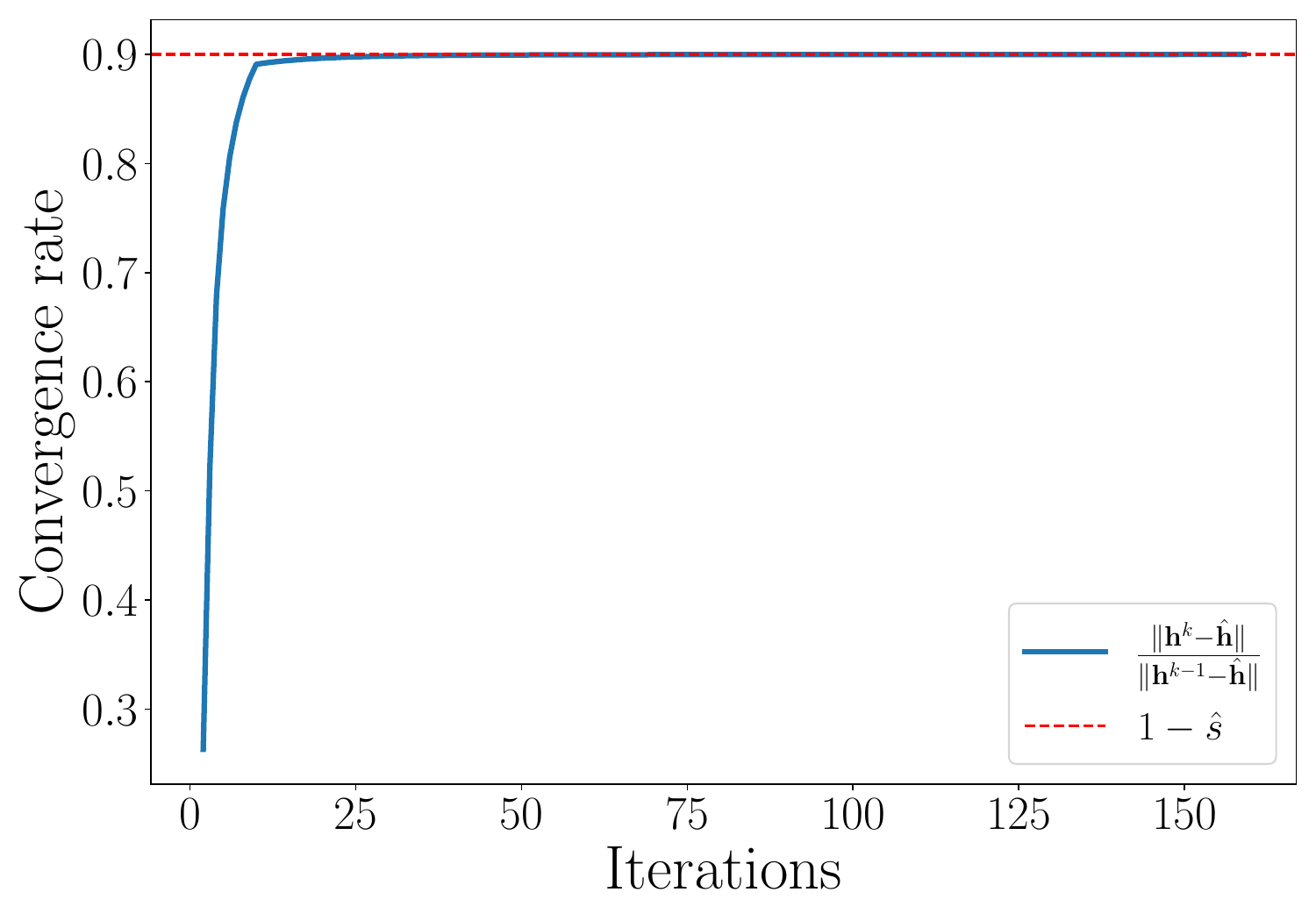}
        \caption{$\theta = 0.5$}
    \end{subfigure}\hfill
    \begin{subfigure}{0.32\textwidth}
        \centering\includegraphics[width=\linewidth]{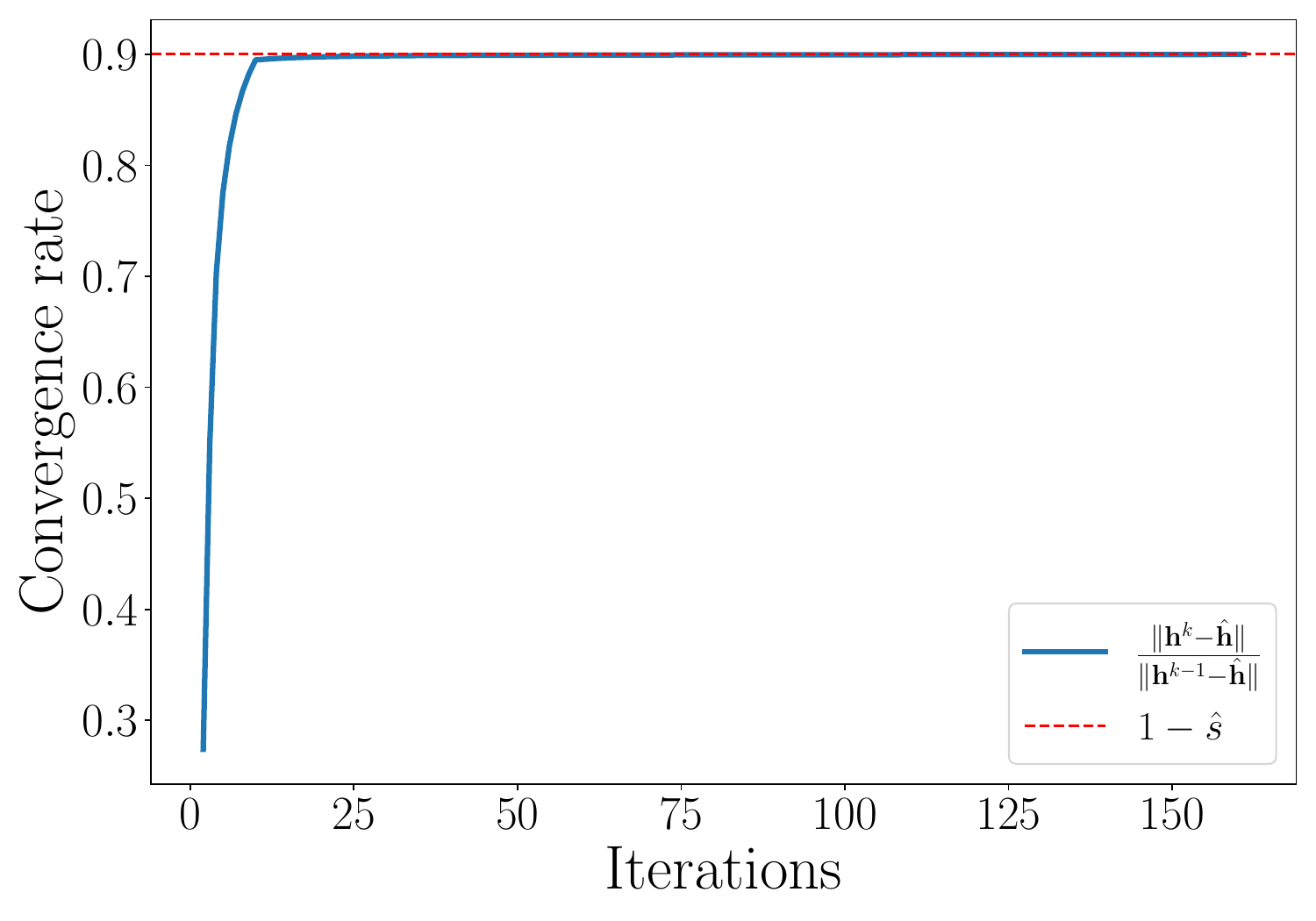}
        \caption{$\theta = 1.0$}
    \end{subfigure}\hfill
    \begin{subfigure}{0.32\textwidth}
        \centering\includegraphics[width=\linewidth]{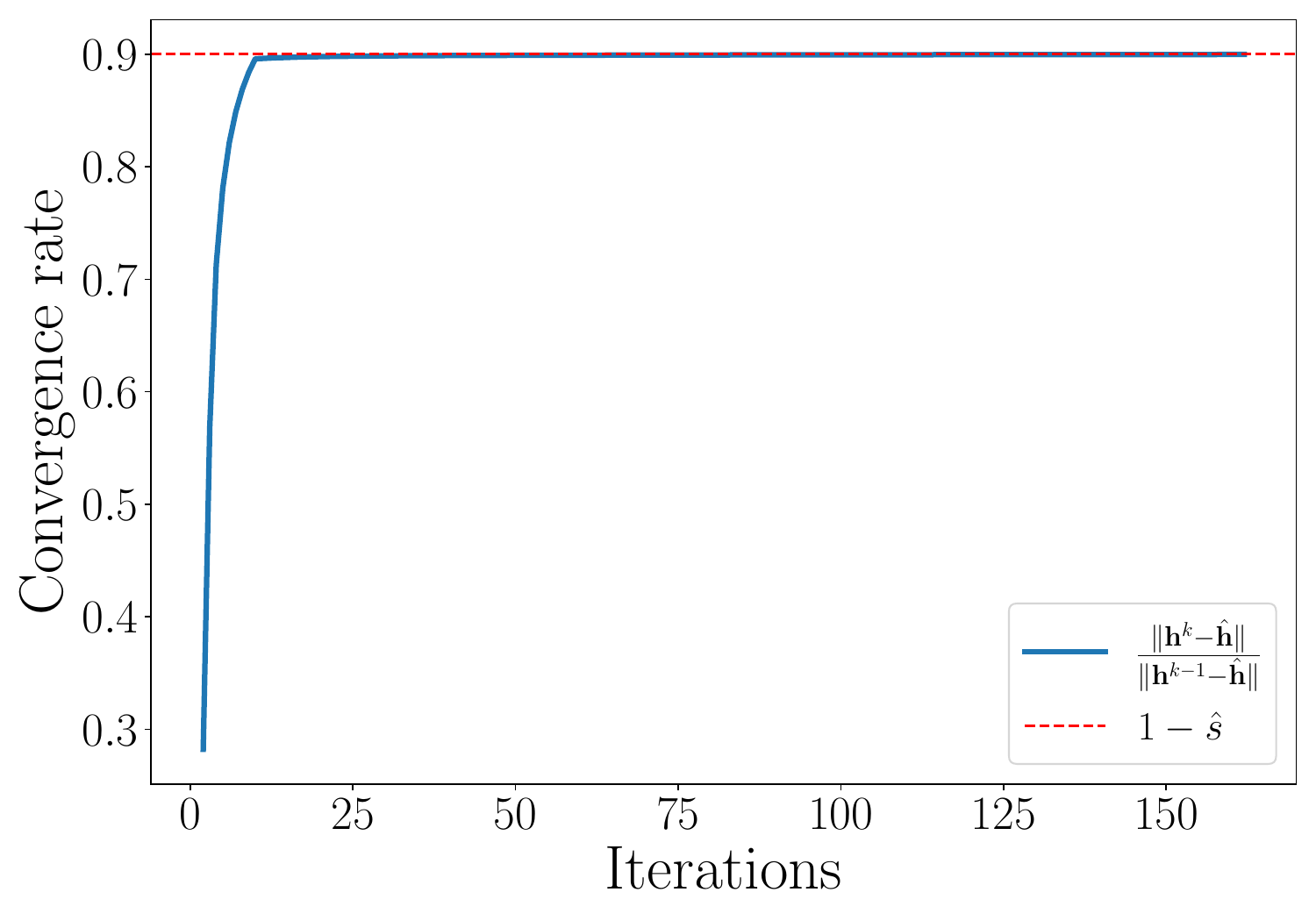}
        \caption{$\theta = 1.5$}
    \end{subfigure}
    \caption{Convergence rate of MSA for the Anaheim network across varying $\theta$ ($I_s = 10)$.}
    \label{fig:convergence_rate_sf}
\end{figure}

\begin{table}
\centering
\footnotesize
\begin{tabular}{@{}cl ccc ccc ccc ccc@{}}
\toprule
\multirow{2}{*}{$\theta$} & \multirow{2}{*}{Network} & \multicolumn{3}{c}{$I_s = 5$} & \multicolumn{3}{c}{$I_s = 10$} & \multicolumn{3}{c}{$I_s = 20$} & \multicolumn{3}{c}{$I_s = 30$} \\ \cmidrule(lr){3-5} \cmidrule(lr){6-8} \cmidrule(lr){9-11} \cmidrule(l){12-14}
 &  & $1-1/I_s$ & $1-\hat{s}$ & Obs. & $1-1/I_s$ & $1-\hat{s}$ & Obs. & $1-1/I_s$ & $1-\hat{s}$ & Obs. & $1-1/I_s$ & $1-\hat{s}$ & Obs. \\
\midrule
\multirow{4}{*}{0.5} & Sioux Falls & \textbf{0.80} & \textbf{0.95} & \textbf{0.95} & 0.90 & 0.90 & 0.90 & 0.95 & 0.95 & 0.95 & 0.97 & 0.97 & 0.97 \\
 & BMC & 0.80 & 0.80 & 0.80 & 0.90 & 0.90 & 0.90 & 0.95 & 0.95 & 0.95 & 0.97 & 0.97 & 0.97 \\
 & EMA & 0.80 & 0.80 & 0.80 & 0.90 & 0.90 & 0.90 & 0.95 & 0.95 & 0.95 & 0.97 & 0.97 & 0.97 \\
 & Anaheim & 0.80 & 0.80 & 0.80 & 0.90 & 0.90 & 0.90 & 0.95 & 0.95 & 0.95 & 0.97 & 0.97 & 0.97 \\
\midrule
\multirow{4}{*}{1.0} & Sioux Falls & \textbf{0.80} & \textbf{0.95} & \textbf{0.95} & \textbf{0.90} & \textbf{0.97} & \textbf{0.97} & 0.95 & 0.95 & 0.95 & 0.97 & 0.97 & 0.97 \\
 & BMC & 0.80 & 0.80 & 0.80 & 0.90 & 0.90 & 0.90 & 0.95 & 0.95 & 0.95 & 0.97 & 0.97 & 0.97 \\
 & EMA & 0.80 & 0.80 & 0.80 & 0.90 & 0.90 & 0.90 & 0.95 & 0.95 & 0.95 & 0.97 & 0.97 & 0.97 \\
 & Anaheim & 0.80 & 0.80 & 0.80 & 0.90 & 0.90 & 0.90 & 0.95 & 0.95 & 0.95 & 0.97 & 0.97 & 0.97 \\
\midrule
\multirow{4}{*}{1.5} & Sioux Falls & \textbf{0.80} & \textbf{0.95} & \textbf{0.95} & \textbf{0.90} & \textbf{0.97} & \textbf{0.97} & 0.95 & 0.95 & 0.95 & 0.97 & 0.97 & 0.97 \\
 & BMC & 0.80 & 0.80 & 0.80 & 0.90 & 0.90 & 0.90 & 0.95 & 0.95 & 0.95 & 0.97 & 0.97 & 0.97 \\
 & EMA & 0.80 & 0.80 & 0.80 & 0.90 & 0.90 & 0.90 & 0.95 & 0.95 & 0.95 & 0.97 & 0.97 & 0.97 \\
 & Anaheim & 0.80 & 0.80 & 0.80 & 0.90 & 0.90 & 0.90 & 0.95 & 0.95 & 0.95 & 0.97 & 0.97 & 0.97 \\
\bottomrule
\end{tabular}
\caption{Observed and expected convergence rates of the MSA algorithm across different networks and $\theta$ values for various initial step-size parameters ($I_s$). Three metrics are reported for each configuration: (1) $1 - 1/I_s$: the theoretically expected convergence rate assuming the initial constant step-size of $1/I_s$ remains unchanged; (2) $1-\hat{s}$: the theoretically expected rate based on the final stable constant step-size $\hat{s}$; (3) Obs.: the observed rate averaged over the final 25 iterations until RGAP $\leq 10^{-9}$.\label{tab:msa_boundgap_convergence}}
\end{table}

Several observations follow from Table~\ref{tab:msa_boundgap_convergence}. First, the observed convergence rate is consistent across networks and matches $1-\hat{s}$, where $\hat{s}$ is the final constant step-size used by the algorithm. Second, the step-size was rarely reduced from its initial value of $1/I_s$, with the only exceptions occurring in a few cases for the Sioux Falls network (shown in bold). For $I_s = 20$ in particular, all networks converged at exactly rate $1-1/I_s$. These results provide empirical support for the theoretical findings of Section~\ref{sec:bound}.

The Sioux Falls network offers additional insight into convergence behavior of MSA with the \texttt{ACS} step-size. When $I_s = 5$ (initial step-size $0.2$), the reset condition of Algorithm~\ref{alg:4_stepsize} was triggered for all $\theta$ values, meaning none maintained $\hat{s} = 1/I_s$. For $I_s = 10$, only $\theta = 0.5$ maintained $\hat{s} = 1/I_s$, while $\theta = 1.0$ and $\theta = 1.5$ required reduction. Only at $I_s = 20$ did all three $\theta$ values maintain $\hat{s} = 1/I_s$. The need to reduce $\hat{s}$ for larger $\theta$ suggests that the maximum admissible constant step-size decreases as $\theta$ increases. This is consistent with Theorem~\ref{thm:msa_convergence}, where the upper bound on the admissible step-size is inversely proportional to $\theta$. A further explanation comes from the behavior of the logit model itself. As $\theta$ increases, path choice probabilities concentrate on fewer paths, making the target path flows vary less smoothly with respect to current flows and weakening the linear approximation of the logit mapping near equilibrium. In the limiting case $\theta \to \infty$, target path flows approach all-or-nothing assignments, which are discontinuous in the current path flows.

To understand why this behavior appears only for the Sioux Falls network, we examine the admissible range of constant step-sizes that guarantee linear convergence using Inequality~\eqref{eq:step_size_b} from Theorem~\ref{thm:msa_convergence}. Table~\ref{tab:eigenvalues} reports the smallest and largest eigenvalues of $\mathcal{K}(\hat{\mb{h}})$ for all tested networks at $\theta = 0.5$. As expected, the largest eigenvalue is zero in all cases. Using the smallest eigenvalue, we compute $s_g = 2/(2-\lambda_{\min})$, the largest step-size guaranteeing linear convergence near equilibrium. For Sioux Falls at $\theta = 0.5$, we compute $s_g = 0.14$. This explains the boldfaced entries in Table~\ref{tab:msa_boundgap_convergence}. Setting $I_s = 5$ gives $1/I_s = 0.2 > s_g$, which fails to achieve rate $1-1/I_s$, whereas $I_s = 10$ gives $1/I_s = 0.1 < s_g$, recovering exact linear convergence. For all other networks, $s_g$ was large enough that no reduction was ever required.

Since $s_g$ depends on the eigenvalues of $\mathcal{K}(\hat{\mb{h}})$ at equilibrium, it is not known before solving traffic assignment. Inequality~\eqref{eq:step_size_c} provides a conservative step-size that is guaranteed to lie below $s_g$, and this conservative value decreases with $\theta$, $\|\mathcal{C}'(\mb{\hat{h}})\|$, and $\max d_{OD}$. Since $\mathcal{C}'(\mb{h})$ factors as $\mb{D}^T \mathcal{T}'(\mb{a})\, \mb{D}$, submultiplicativity of the spectral norm gives $\left\|\mathcal{C}'(\mb{h})\right\| \le \|\mb{D}\|^2 \|\mathcal{T}'(\mb{a})\|$, so the bound on $s_g$ decreases as $\|\mb{D}\|$ and $\|\mathcal{T}'(\mb{a})\|$ grow. A (loose) upper bound on $\|\mathcal{T}'(\mb{a})\|$ at equilibrium can be obtained by evaluating $\mathcal{T}'$ at $\mb{a}_{\max}$, the link flow vector in which every link flow equals the total demand in the network, provided each $\tau_l$ is convex, which BPR functions are. The factors $\|\mb{D}\|$, $\|\mathcal{T}'(\mb{a}_{\max})\|$, and $\max d_{OD}$ are reported in Table~\ref{tab:eigenvalues} for all tested networks at $\theta = 0.5$. Sioux Falls, for instance, has the largest product $\max d_{OD} (\|\mb{D}\|^2 \|\mathcal{T}'(\mb{a}_{\max})\|)$ forcing the upper bound on $s_g$ to be small. One could initialize MSA with this conservative value, but doing so sacrifices fast initial progress of MSA. The bound in \eqref{eq:step_size_c} is also loose, and hence the actual $|\lambda_{\min}|$ is substantially smaller than the upper bound, making the bound on $s_g$ far more pessimistic than necessary. This is precisely why an adaptive constant step-size like Algorithm~\ref{alg:4_stepsize} becomes necessary.


\begin{table}
\centering
\begin{tabular}{l ccccccc}
\toprule
Network & $\max d_{OD}$ & $\|\mb{D}\|$ & $\|\mathcal{T}'(\mb{a}_{\max})\|$ & $\tilde{s}_g$ & $\lambda_{\max}$  & $\lambda_{\min}$  & $s_g$ \\
\midrule
Sioux Falls & $4400.0$  & $82.4$  & $432.5$ & $3.1 \times 10^{-10}$ & $0.0$ & $-12.63$ & $0.14$ \\
BMC         & $97.7$    & $195.6$ & $887.7$ & $1.2 \times 10^{-9}$  & $0.0$ & $-2.80$  & $0.42$ \\
EMA         & $957.7$   & $111.6$ & $118.1$ & $2.8 \times 10^{-9}$  & $0.0$ & $-1.27$  & $0.61$ \\
Anaheim     & $2106.7$  & $191.0$ & $32.8$  & $1.6 \times 10^{-9}$  & $0.0$ & $-1.26$  & $0.61$ \\
\bottomrule
\end{tabular}
\caption{Factors influencing $s_g$ and eigenvalues of $\mathcal{K}(\mb{\hat{h}})$ at $\theta = 0.5$. The vector $\mb{a}_{\max}$ denotes the link flow vector in which every link flow equals $\sum_{OD} d_{OD}$,  $\|\mathcal{T}'(\mb{a}_{\max})\|$ is the spectral norm of the marginal link cost matrix evaluated at this vector. The column $\tilde{s}_g$ reports the conservative step-size obtained from Inequality~\eqref{eq:step_size_c}, computed using $\|\mathcal{C}'(\mb{\hat{h}})\| \,\,\leq \|\mb{D}\|^2 \|\mathcal{T}'(\mb{a}_{\max})\|$. ${\lambda_{\max}}$ and ${\lambda_{\min}}$ report the maximum and minimum eigenvalues of $\mathcal{K}(\mb{\hat{h}})$ at equilibrium, respectively, and $s_g$ represents the maximum step-size for linear convergence near equilibrium computed from $\lambda_{\min}$ via Inequality~\eqref{eq:step_size_b}. The gap between $s_g$ and $\tilde{s}_g$ illustrates how loose the bound in~\eqref{eq:step_size_c} is in practice.\label{tab:eigenvalues}}
\end{table}

\subsection{Analysis of computational performance}\label{subsec:res_newton}

We present computational performance for the traffic assignment algorithms derived in this paper: \texttt{Newton} and \texttt{MSA-ACS}. Since \texttt{MSA-ACS} was shown to be globally convergent, the step-size $s_k$ at iteration $k$ is obtained directly using Algorithm~\ref{alg:4_stepsize}, and this step-size is then applied within the MSA algorithm (Algorithm~\ref{alg:2_msa}). The \texttt{Newton} method, on the other hand, is only locally convergent. To ensure global convergence, we pair it with the globally convergent BB step-size rule of \cite{du2021faster}, adaptively shifting to the \texttt{Newton} method whenever possible. For any feasible path flow $\mb{h}^k$ close to equilibrium at iteration $k$, the next iterate $\mb{h}^{k+1}$ is then obtained by applying Algorithm~\ref{alg:newton_step}. We compare our two proposed methods against the standard MSA harmonic step-size sequence (\texttt{MSA-HS}) and the standalone BB step-size rule of \cite{du2021faster}.

The reasons for choosing the BB step-sizes from \cite{du2021faster} as a benchmark and as the globally convergent algorithm for the \texttt{Newton} method are twofold. First, these step-sizes incorporate partial second-order information while maintaining per-iteration times similar to first-order methods, and the authors have shown that they outperform popular rules such as Armijo and self-regulated averaging in terms of computational runtimes. Second, the BB step-sizes can be used directly within Algorithm~\ref{alg:2_msa} as an alternative step-size rule, which allows for a fair comparison of computation times since the underlying algorithmic framework remains identical across all tests. The two step-sizes proposed in \cite{du2021faster}, $s_k^{BB1}$ and $s_k^{BB2}$, are shown in Equations~\eqref{eq:BB1} and~\eqref{eq:BB2}. When these are used within MSA, we denote the resulting methods as \texttt{BB1} and \texttt{BB2}, respectively.
\begin{align}
    s_k^{BB1} &= \frac{(\mb{h}^k - \mb{h}^{k-1})^{T}\bigl[(\mb{h}^k - \mb{h}^{k-1}) - (\mathcal{L}(\mb{h}^k) - \mathcal{L}(\mb{h}^{k-1}))\bigr]}{\bigl\|(\mb{h}^k - \mb{h}^{k-1}) - (\mathcal{L}(\mb{h}^k) - \mathcal{L}(\mb{h}^{k-1}))\bigr\|^{2}} \label{eq:BB1}\\
    s_k^{BB2} &= \frac{\|\mb{h}^k - \mb{h}^{k-1}\|^{2}}{(\mb{h}^k - \mb{h}^{k-1})^{T}\bigl[(\mb{h}^k - \mb{h}^{k-1}) - (\mathcal{L}(\mb{h}^k) - \mathcal{L}(\mb{h}^{k-1}))\bigr]} \label{eq:BB2}
\end{align}

Although \cite{du2021faster} showed BB step-sizes to outperform other line search approaches, BB steps reduce the objective function non-monotonically, which is known to induce oscillations near the equilibrium \citep{dai2002r}. This issue is further exacerbated when the Jacobian of the logit mapping is ill-conditioned \citep{an2025regularized}. As shown in Theorem~\ref{thm:spectrum}, this ill-conditioning increases as demand and $\theta$ increase in the network. The extent of these oscillation issues near equilibrium was likely not observed in \cite{du2021faster} because their test networks (Winnipeg Asymmetric and Chicago Sketch) were medium-sized, with demand ranging from only 0.6 to 1.4 times the base demand. In reality, the trip tables for many networks in \cite{transportationNetworks} are dated and contain much lower demand values than current conditions. For instance, \cite{transportationNetworks} notes that the original data for the Chicago Sketch network provides low levels of congestion that are not realistic for the Chicago region, and recommends doubling the original trip table for algorithm testing. Furthermore, near equilibrium, the denominators in Equations~\eqref{eq:BB1} and \eqref{eq:BB2} can evaluate to exactly zero in floating point precision, causing numerical failures. To address these limitations, we pair \texttt{BB1} and \texttt{BB2} with a fallback to \texttt{ACS} step-sizes (Algorithm~\ref{alg:4_stepsize}) whenever \texttt{BB1} or \texttt{BB2} numerically fail. We refer to these rules as \texttt{BB1-ACS} and \texttt{BB2-ACS}.

The \texttt{Newton} step-size proposed in this study uses full second-order information of the logit mapping, but is only locally convergent. Hence, it must be paired with a globally convergent solver. We pair it with \texttt{BB1-ACS}, and refer to this combined rule as \texttt{BB-Newton}. We choose \texttt{BB1} over \texttt{BB2} since \texttt{BB1} step-sizes are shorter and known to produce fewer oscillations~\citep{dai2002r}. Following the discussion in Section~\ref{subsec:4_convergence} on when to switch from a globally convergent algorithm to \texttt{Newton}, we apply Algorithm~\ref{alg:bb_newton} to obtain global convergence together with local quadratic convergence. In Algorithm~\ref{alg:bb_newton}, we initialize a decreasing threshold sequence $\{\tau_j\}_{j=1}^{J}$ that determines when to check if we should shift from \texttt{BB1-ACS} to the \texttt{Newton} step. Starting from \texttt{BB1-ACS}, each time the RGAP crosses a new threshold, a \texttt{Newton} step is attempted. If it fails, we revert to \texttt{BB1-ACS}. Once a \texttt{Newton} step is accepted, we continue taking \texttt{Newton} steps unless one fails. In our experiments, a \texttt{Newton} step once accepted never failed, although in pathological instances it is possible for a \texttt{Newton} step to fall back to \texttt{BB1-ACS}. Nonetheless, the convergence guarantee established in Section~\ref{subsec:4_convergence} ensures that eventually only \texttt{Newton} steps are taken in a small enough neighborhood of SUE.

The threshold sequence $\{\tau_j\}_{j=1}^{J}$ in Algorithm~\ref{alg:bb_newton} has little effect on performance as long as the thresholds are spread out. This is because only one \texttt{Newton} step is attempted between consecutive thresholds until one is accepted. As long as the sequence does not check for a \texttt{Newton} step too frequently and allows gap reduction between successive checks, the overall behavior remains the same. We chose the threshold sequence as $\{10^{-3}, 10^{-4}, 10^{-5}, \ldots, 10^{-10}\}$.

\begin{algorithm}
\caption{\texttt{BB-Newton}}
\label{alg:bb_newton}
\begin{algorithmic}[1]
\State \textbf{Input:} Initial feasible path flow $\mb{h}^0$
\State \textbf{Hyperparameters:} Threshold sequence
\State Initialize $\texttt{newtonMode} \gets \textsc{False}$
\For{$k = 0,1,2,\dots$}
    \If{$\texttt{newtonMode}$ is True \textbf{or} RGAP has crossed a new threshold}
        \State Attempt $\mb{h}^{k+1} \gets \texttt{Newton}(\mb{h}^k)$ via Algorithm~\ref{alg:newton_step}
        \State Set \texttt{newtonMode} as True if Newton step accepted, else False.
    \EndIf
    \If{$\texttt{newtonMode}$ is False} 
        \State Compute $s_k$ using \texttt{BB1-ACS}.
        \State $\mb{h}^{k+1} \gets \mb{h}^k + s_k\bigl(\mathcal{L}(\mb{h}^k) - \mb{h}^k\bigr)$
    \EndIf
    \State If RGAP is sufficiently small, \textbf{terminate}.
\EndFor
\State \textbf{Return:} $\mb{h}^k$
\end{algorithmic}
\end{algorithm}

\begin{table}
\centering
\small
\begin{tabular}{p{1.5cm} p{6.5cm} p{2.2cm} l}
\toprule
\textbf{Method} & \textbf{Description} & \textbf{Asymptotic Convergence} & \textbf{Proposed by} \\
\midrule
\texttt{MSA-HS} &
Harmonic step-size sequence: $s_k = 1/k$ in iteration $k$ in Algorithm~\ref{alg:2_msa} &
Sublinear &
\cite{sheffi1982algorithm} \\
\texttt{MSA-ACS} &
Adaptive constant step scheme: uses step-size from Algorithm~\ref{alg:4_stepsize} in Algorithm~\ref{alg:2_msa} &
Linear &
This paper \\
\texttt{BB1} &
Barzilai-Borwein type 1: Use $s_k^{BB1}$ clipped to $[0, 1]$ for numerical stability in Algorithm~\ref{alg:2_msa}. &
Linear &
\cite{du2021faster} \\
\texttt{BB2} &
Barzilai-Borwein type 2: Use $s_k^{BB2}$ clipped to $[0, 1]$ for numerical stability in Algorithm~\ref{alg:2_msa}. &
Linear &
\cite{du2021faster} \\
\texttt{BB1-ACS} &
Uses BB1 when well-defined; reverts to the \texttt{ACS} (Algorithm~\ref{alg:4_stepsize}) when BB1 is numerically undefined. &
Linear &
This paper + \cite{du2021faster} \\
\texttt{BB2-ACS} &
Uses BB2 when well-defined; reverts to the \texttt{ACS} (Algorithm~\ref{alg:4_stepsize}) when BB2 is numerically undefined. &
Linear &
This paper + \cite{du2021faster} \\
\texttt{BB-Newton} & Uses Algorithm~\ref{alg:bb_newton} to adaptively shift between \texttt{BB1-ACS} and \texttt{Newton} &
Quadratic &
This paper + \cite{du2021faster}\\
\bottomrule
\end{tabular}
\caption{Description of step-size determination rules used in computational experiments.}
\label{tab:steprules}
\end{table}

In summary, we test a total of seven methods: \texttt{MSA-HS}, \texttt{MSA-ACS}, \texttt{BB-Newton}, \texttt{BB1}, \texttt{BB2}, \texttt{BB1-ACS}, and \texttt{BB2-ACS}. We summarize each of these rules in Table~\ref{tab:steprules}. We tested the computational performance of all seven step-size rules on the networks listed in Table~\ref{tab:4_res} under two scenarios: the base demand, and a scenario with demand values multiplied by a factor of two. For each case, the traffic assignment was run until reaching an RGAP of $10^{-10}$, and a $\theta$ value of $1.0$ was used for all experiments based on parameters from recent studies \citep{wang2025enhancing}.

Different runtime limits were imposed based on network scale: a two-second limit for small-scale networks, one minute for medium-scale networks, and 15 minutes for large-scale networks. The only exception was the Chicago Regional network with doubled demand, where no step-size rule reached the $10^{-10}$ threshold in RGAP within 15 minutes. In this case, the computational budget was extended to 30 minutes. The time budget was measured starting from the first iteration, so the time required to compute path sets, the link-path incidence matrix, and the initial logit loading on free-flow travel times was excluded from the timing.

The convergence behavior for three of the tested networks is illustrated in Figures~\ref{fig:conv_winnipeg_1x} through \ref{fig:conv_chicago_regional_2x}. The plots for \texttt{BB1} and \texttt{BB2} are omitted, as their convergence behavior is identical to that of \texttt{BB1-ACS} and \texttt{BB2-ACS} whenever they remain numerically stable. When they do stall, it is due to a division by zero in machine precision, and their curves simply match the \texttt{ACS} variants up to the point of failure. Several observations can be drawn from these results. First, \texttt{MSA-HS} exhibits sublinear convergence, while \texttt{MSA-ACS} shows linear convergence even far from equilibrium. However, even linear convergence at a slow rate requires a large number of iterations to reach the target gap. Second, the benefit of utilizing full second-order information in \texttt{BB-Newton} is evident, as it converges in significantly fewer iterations than the other step-size rules. Both phases of \texttt{BB-Newton} are shown in the plots. Once the shift to \texttt{Newton} step-sizes occurs, the algorithm reaches $10^{-10}$ in just $4$--$5$ iterations, indicating fast superlinear convergence. Finally, the \texttt{BB1-ACS} and \texttt{BB2-ACS} step-sizes struggle particularly when demand is high. This behavior is expected, since increasing the demand makes the problem more ill-conditioned, which in turn makes the use of full second-order information more valuable.

\begin{figure}
    \centering
    \begin{subfigure}[b]{0.48\textwidth}
        \centering
        \includegraphics[width=\textwidth]{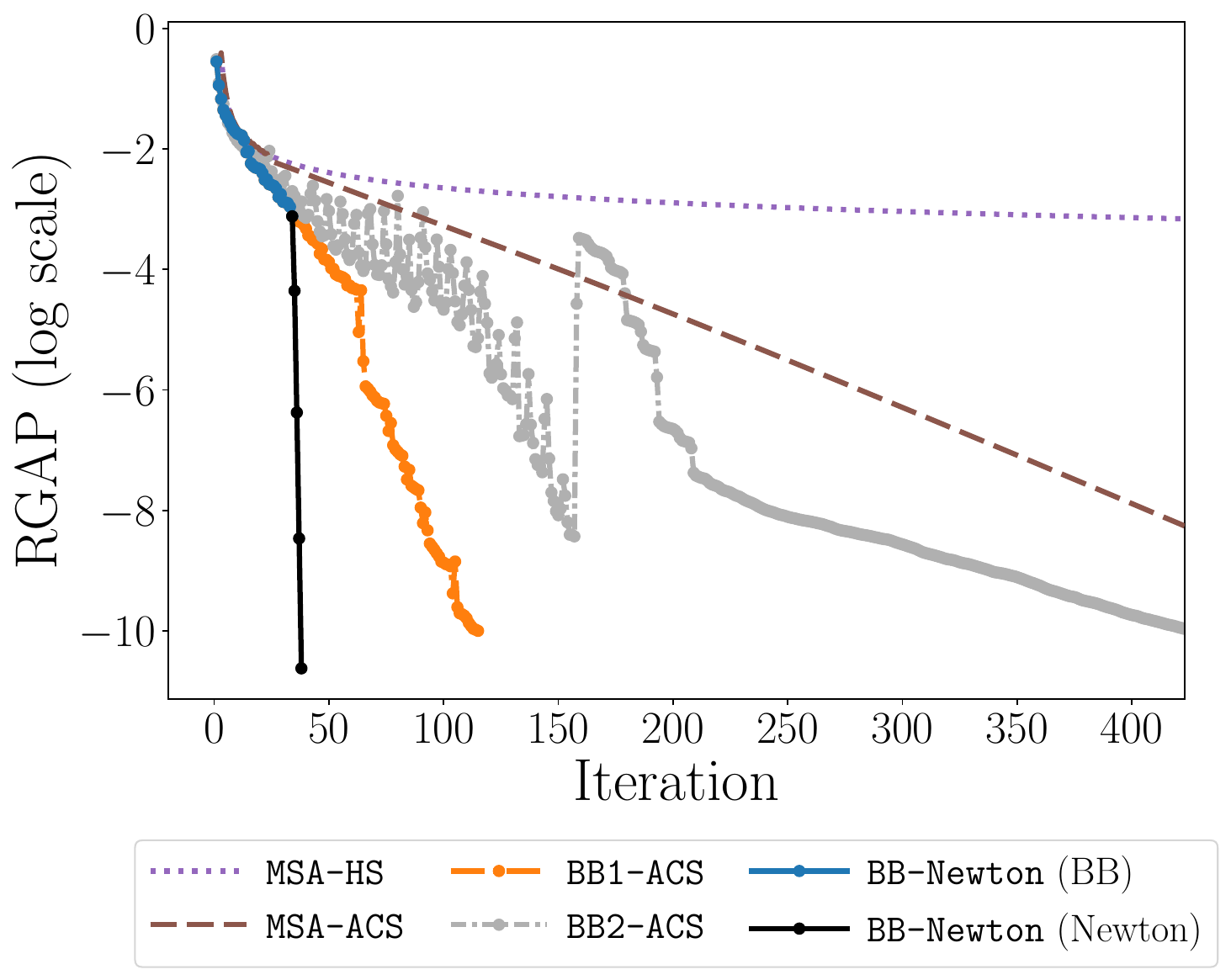}
        \caption{Winnipeg Asymmetric, 1x demand}
        \label{fig:conv_winnipeg_1x}
    \end{subfigure}
    \hfill
    \begin{subfigure}[b]{0.48\textwidth}
        \centering
        \includegraphics[width=\textwidth]{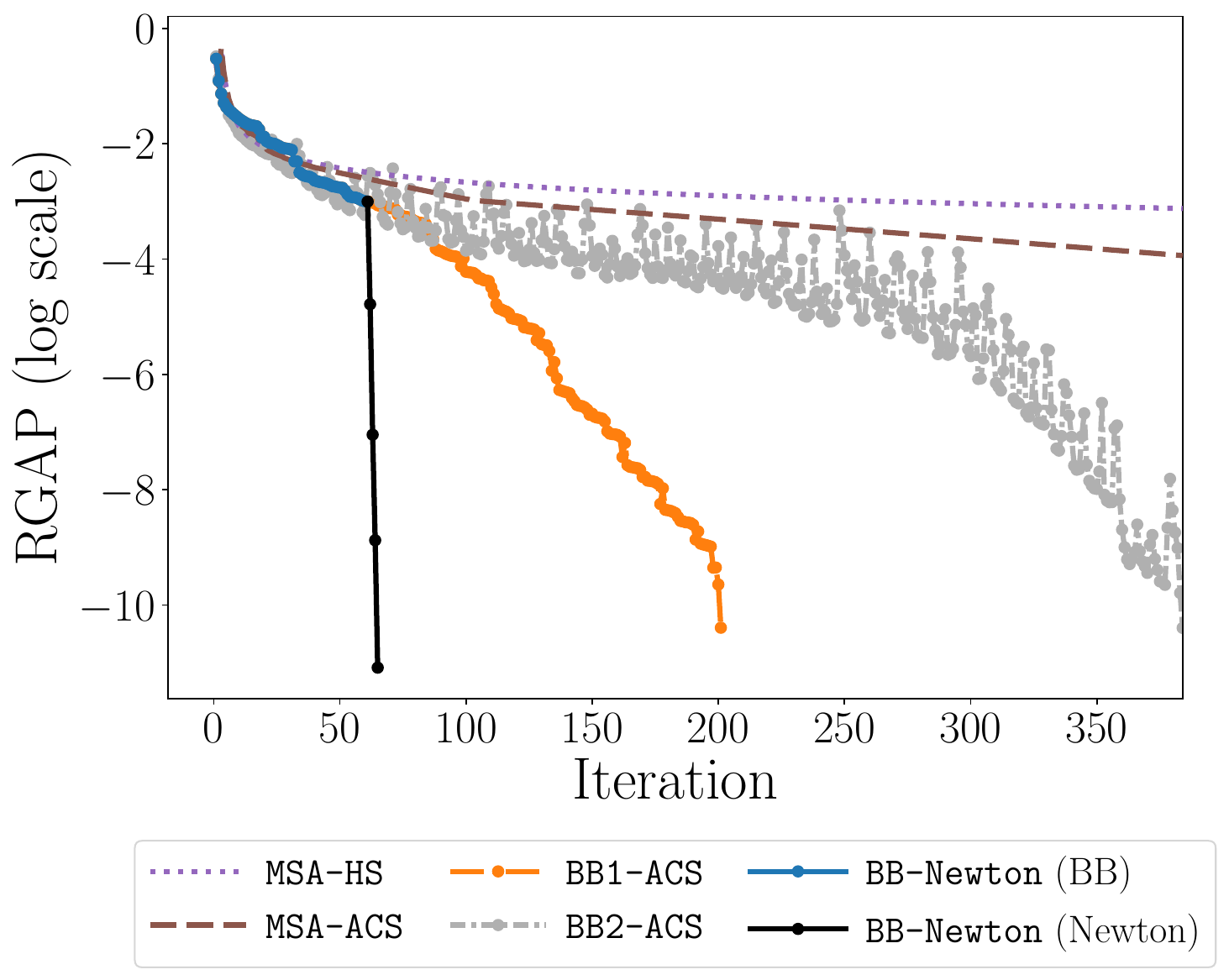}
        \caption{Winnipeg Asymmetric, 2x demand}
        \label{fig:conv_winnipeg_2x}
    \end{subfigure}

    \medskip

    \begin{subfigure}[b]{0.48\textwidth}
        \centering
        \includegraphics[width=\textwidth]{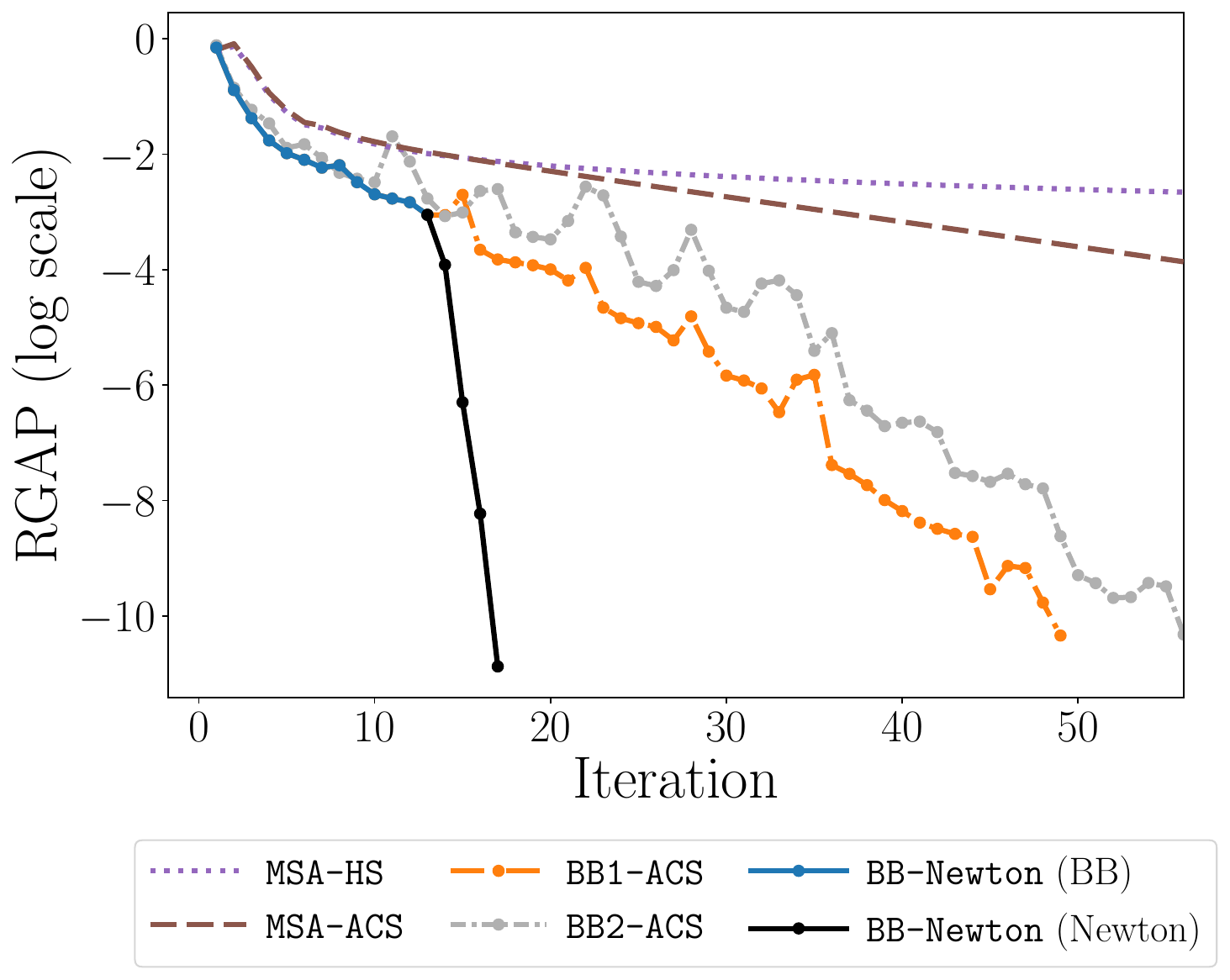}
        \caption{Chicago Sketch, 1x demand}
        \label{fig:conv_chicago_sketch_1x}
    \end{subfigure}
    \hfill
    \begin{subfigure}[b]{0.48\textwidth}
        \centering
        \includegraphics[width=\textwidth]{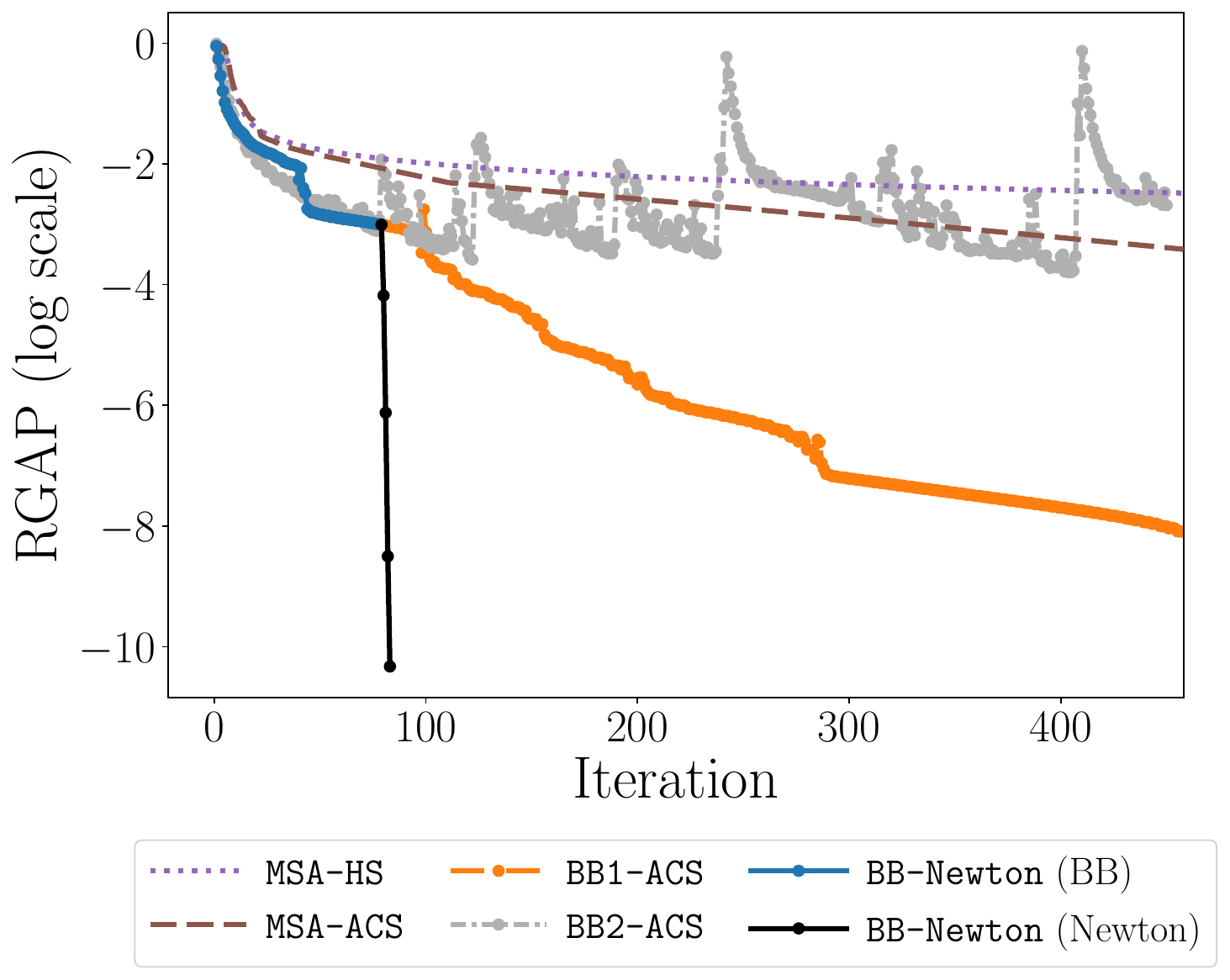}
        \caption{Chicago Sketch, 2x demand}
        \label{fig:conv_chicago_sketch_2x}
    \end{subfigure}

    \medskip

    \begin{subfigure}[b]{0.48\textwidth}
        \centering
        \includegraphics[width=\textwidth]{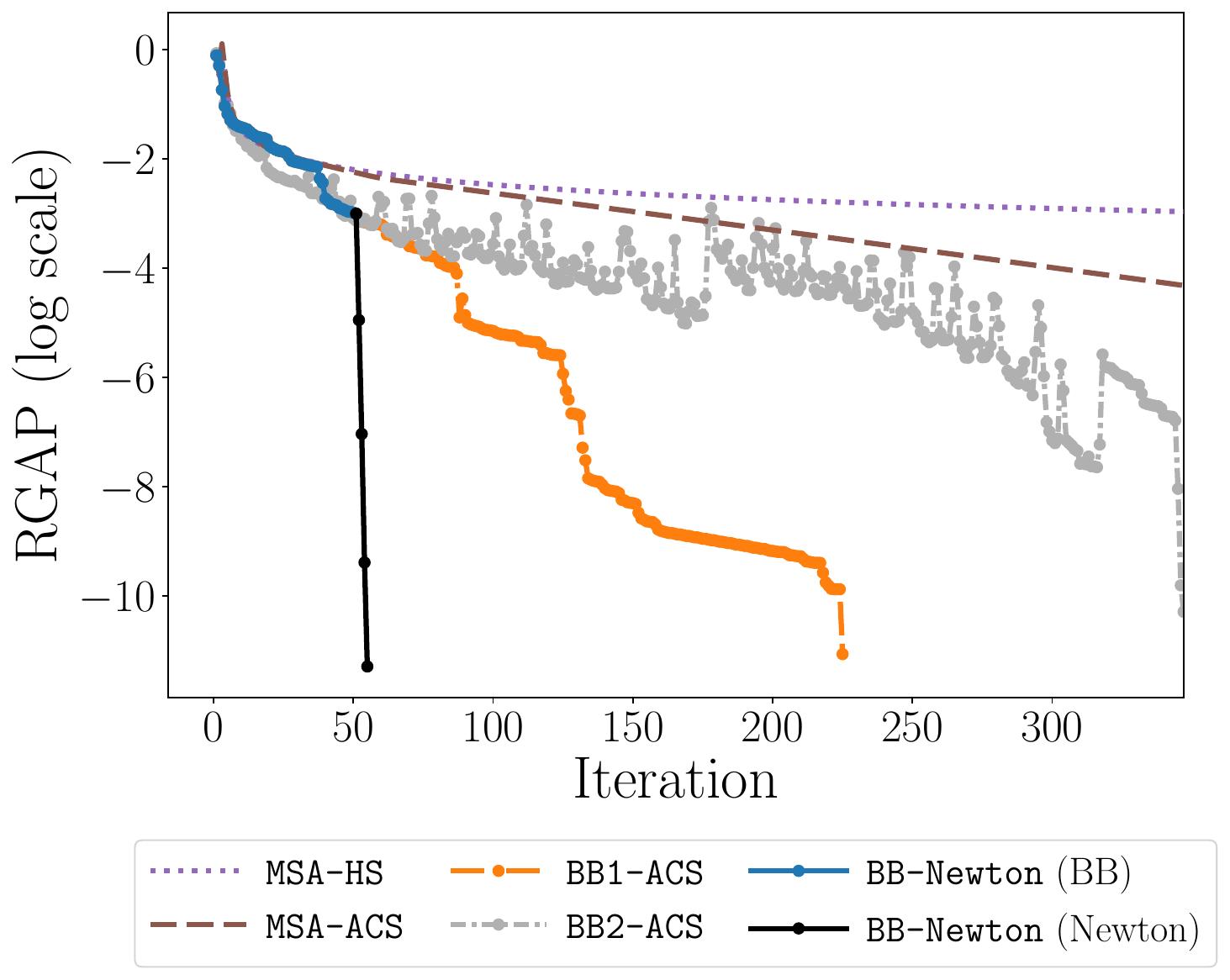}
        \caption{Chicago Regional, 1x demand}
        \label{fig:conv_chicago_regional_1x}
    \end{subfigure}
    \hfill
    \begin{subfigure}[b]{0.48\textwidth}
        \centering
        \includegraphics[width=\textwidth]{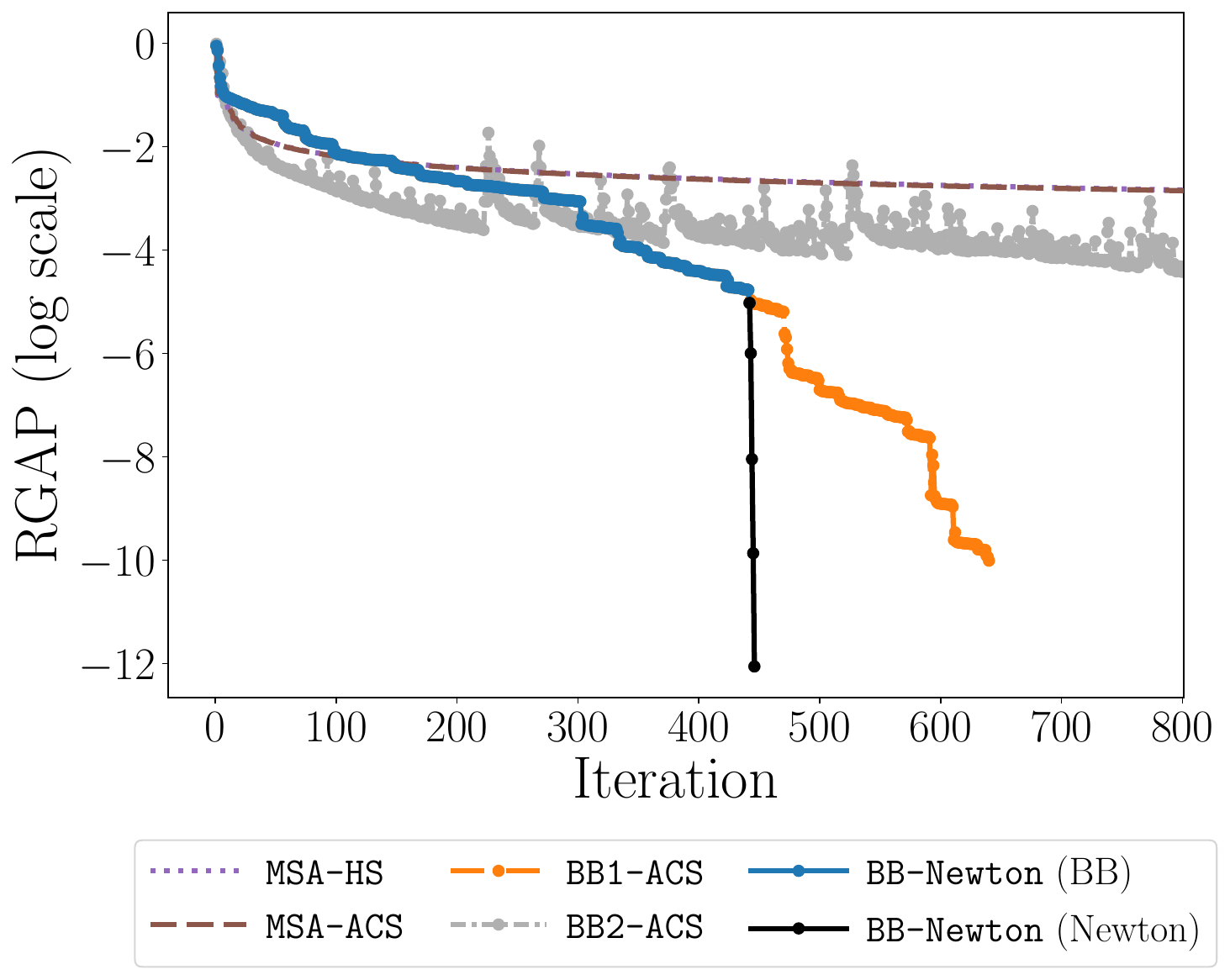}
        \caption{Chicago Regional, 2x demand}
        \label{fig:conv_chicago_regional_2x}
    \end{subfigure}
    \caption{Convergence of RGAP (log scale) versus iteration for \texttt{BB-Newton, MSA-HS, MSA-ACS, BB1-ACS}, and \texttt{BB2-ACS} on Winnipeg Asymmetric, Chicago Sketch, and Chicago Regional networks at base (1$\times$) and doubled (2$\times$) demand levels.}
    \label{fig:convergence_grid}
\end{figure}

Although the \texttt{BB-Newton} step provides fast convergence, the quadratic convergence only takes effect close to equilibrium. To verify the order of convergence empirically, we define the estimated order of convergence \citep{GRAUSANCHEZ2010472} $\hat{O}$ as:
\begin{equation}
    \hat{O} = \frac{1}{|\mathbb{I}|} \sum_{k \in \mathbb{I}} \frac{\ln(\text{RGAP}_k / \text{RGAP}_{k-1})}{\ln(\text{RGAP}_{k-1} / \text{RGAP}_{k-2})},
    \label{eq:empirical_order}
\end{equation}
where $\text{RGAP}_k$ is the relative gap at iteration $k$ and $\mathbb{I}$ is the set of iterations for which a Newton step was taken.

An empirical order of $\hat{O} = 2$ indicates quadratic convergence, $1 < \hat{O} < 2$ indicates superlinear convergence, and $\hat{O} = 1$ denotes linear convergence. Table~\ref{tab:convergence_summary_newton} summarizes the empirical order of convergence across all tested networks. For the majority of cases, the observed order exceeds one, signifying superlinear convergence. The only exception is the Berlin Center network with demand doubled, which exhibits an order of $0.8$. We note that the \texttt{Newton} step-size rule guarantees quadratic convergence only when the iterates are sufficiently close to the SUE. Although the condition used to switch to a Newton step is necessary for this behavior, it is not sufficient to ensure that the iterates have entered the region where local quadratic convergence takes effect. This explains the observed superlinear convergence.

\begin{table}
\centering
\footnotesize
\begin{tabular}{@{}lcccccccc@{}}
\toprule
\multirow{2}{*}{Network} & \multicolumn{4}{c}{Demand $1\times$} & \multicolumn{4}{c}{Demand $2\times$} \\
\cmidrule(lr){2-5} \cmidrule(l){6-9}
 & Iters & Newton iters & Newton start gap & $\hat{O}$ & Iters & Newton iters & Newton start gap & $\hat{O}$ \\
\midrule
Sioux Falls & 38 & 5 & $10^{-3}$ & 1.26 & 182 & 5 & $10^{-3}$ & 1.43 \\
BMC & 16 & 5 & $10^{-4}$ & 1.51 & 80 & 5 & $10^{-5}$ & 6.08 \\
EMA & 8 & 4 & $10^{-3}$ & 1.11 & 18 & 5 & $10^{-3}$ & 1.23 \\
Anaheim & 8 & 4 & $10^{-4}$ & 1.21 & 19 & 5 & $10^{-3}$ & 1.24 \\ \midrule
Chicago Sketch & 17 & 5 & $10^{-3}$ & 1.64 & 83 & 5 & $10^{-3}$ & 1.22 \\
Berlin Center & 30 & 5 & $10^{-3}$ & 1.18 & 313 & 4 & $10^{-4}$ & 0.78 \\
Winnipeg Asymmetric & 38 & 5 & $10^{-3}$ & 1.23 & 65 & 5 & $10^{-3}$ & 1.10 \\ \midrule
Philadelphia & 32 & 5 & $10^{-3}$ & 1.26 & 347 & 4 & $10^{-7}$ & 1.98 \\
Chicago Regional & 55 & 4 & $10^{-3}$ & 1.10 & 446 & 5 & $10^{-5}$ & 1.40 \\
\bottomrule
\end{tabular}
\caption{Order of convergence for \texttt{Newton} method across demand multipliers. For each demand level we report: iterations to reach RGAP $10^{-10}$, number of Newton-accepted iterations, the RGAP at which the first Newton step was accepted, and the empirical order of convergence $\hat{O}$.}
\label{tab:convergence_summary_newton}
\end{table}

Finally, we present the computational times required to reach an RGAP of $10^{-10}$ for each of the seven step-size rules. Figure~\ref{fig:runtime_grid} plots computational runtime against the gap level reached for three networks, under both base demand and demand scaled by a factor of two, across five step-size rules. As in Figure~\ref{fig:convergence_grid}, \texttt{BB1} and \texttt{BB2} are omitted in favor of \texttt{BB1-ACS} and \texttt{BB2-ACS}. The results for all step-sizes are summarized for small, medium, and large scale networks in Tables~\ref{tab:convergence_small}, \ref{tab:convergence_medium}, and \ref{tab:convergence_large}, respectively. For each tested scenario, the tables report the computational time to reach an RGAP of $10^{-10}$ and the RGAP achieved at the final iteration. If a step-size rule did not converge within the budget, this is denoted by ``--'' alongside the RGAP in the final iteration. Instances where the run stopped due to numerical issues are marked with an asterisk. The computational budget allocated for each set of experiments is specified in the table captions.

\begin{figure}
    \centering
    \begin{subfigure}[b]{0.45\textwidth}
        \centering
        \includegraphics[width=\textwidth]{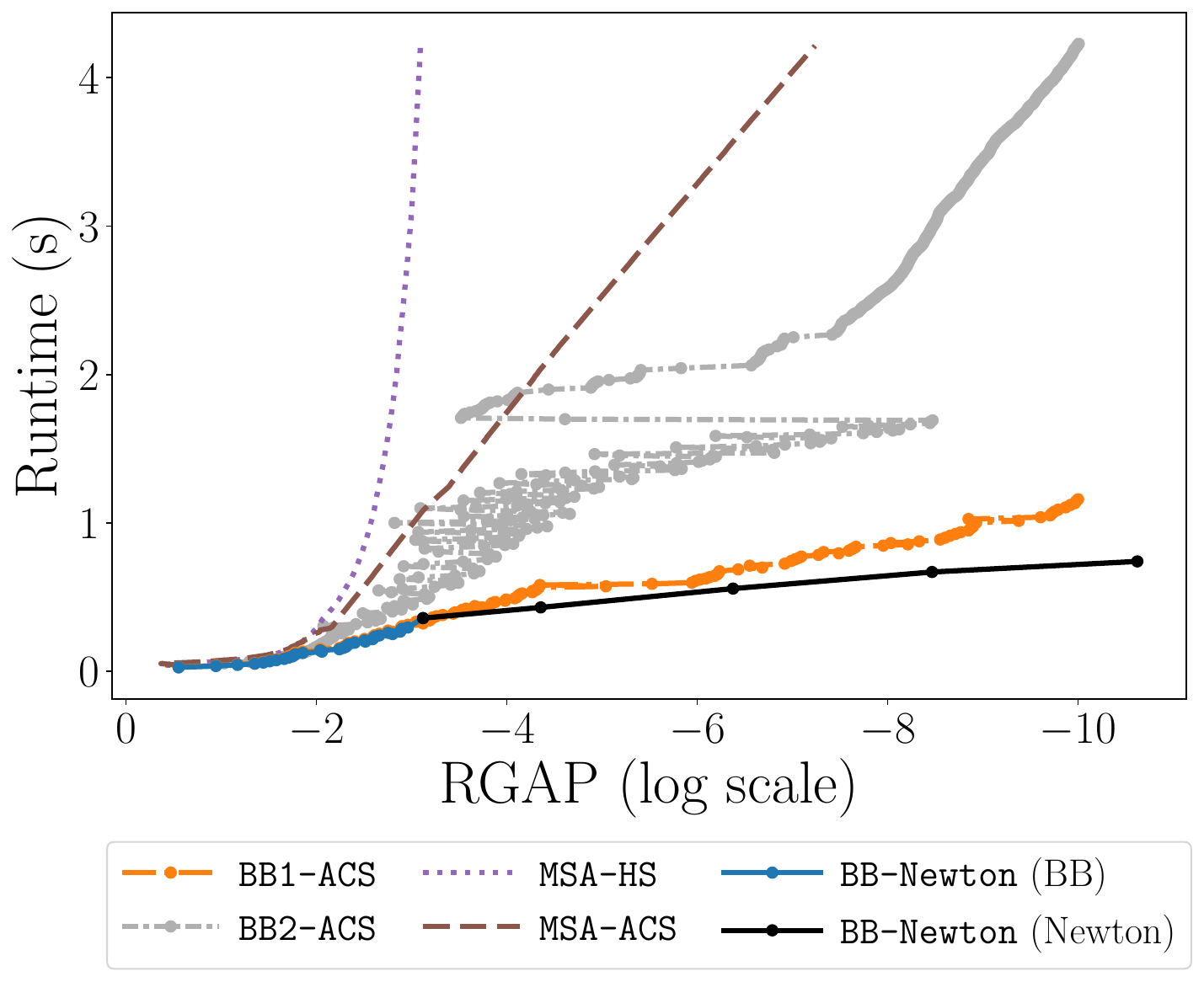}
        \caption{Winnipeg Asymmetric, 1x demand}
        \label{fig:runtime_grid_1}
    \end{subfigure}
    \hfill
    \begin{subfigure}[b]{0.45\textwidth}
        \centering
        \includegraphics[width=\textwidth]{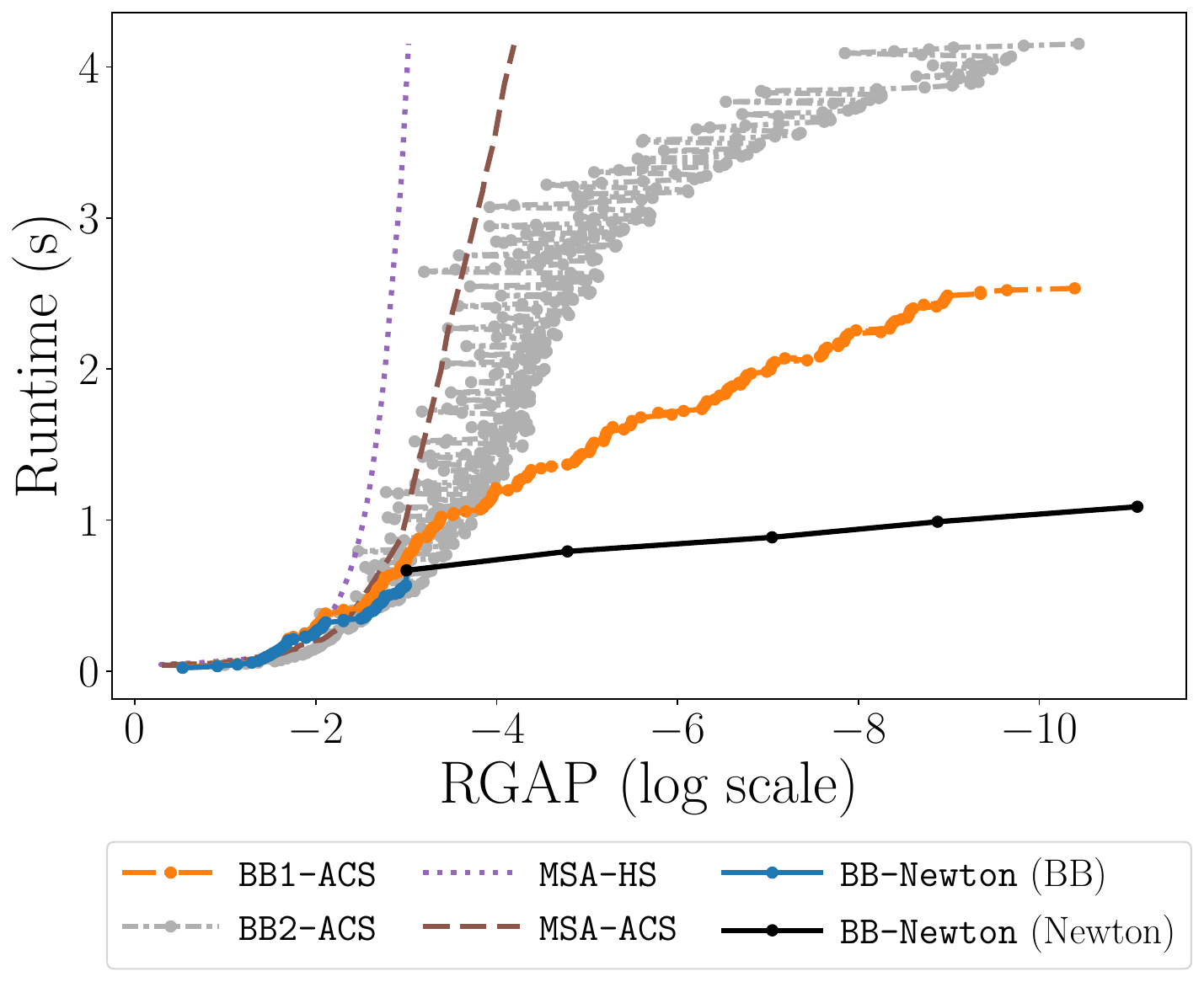}
        \caption{Winnipeg Asymmetric, 2x demand}
        \label{fig:runtime_grid_4}
    \end{subfigure}

    \medskip

    \begin{subfigure}[b]{0.45\textwidth}
        \centering
        \includegraphics[width=\textwidth]{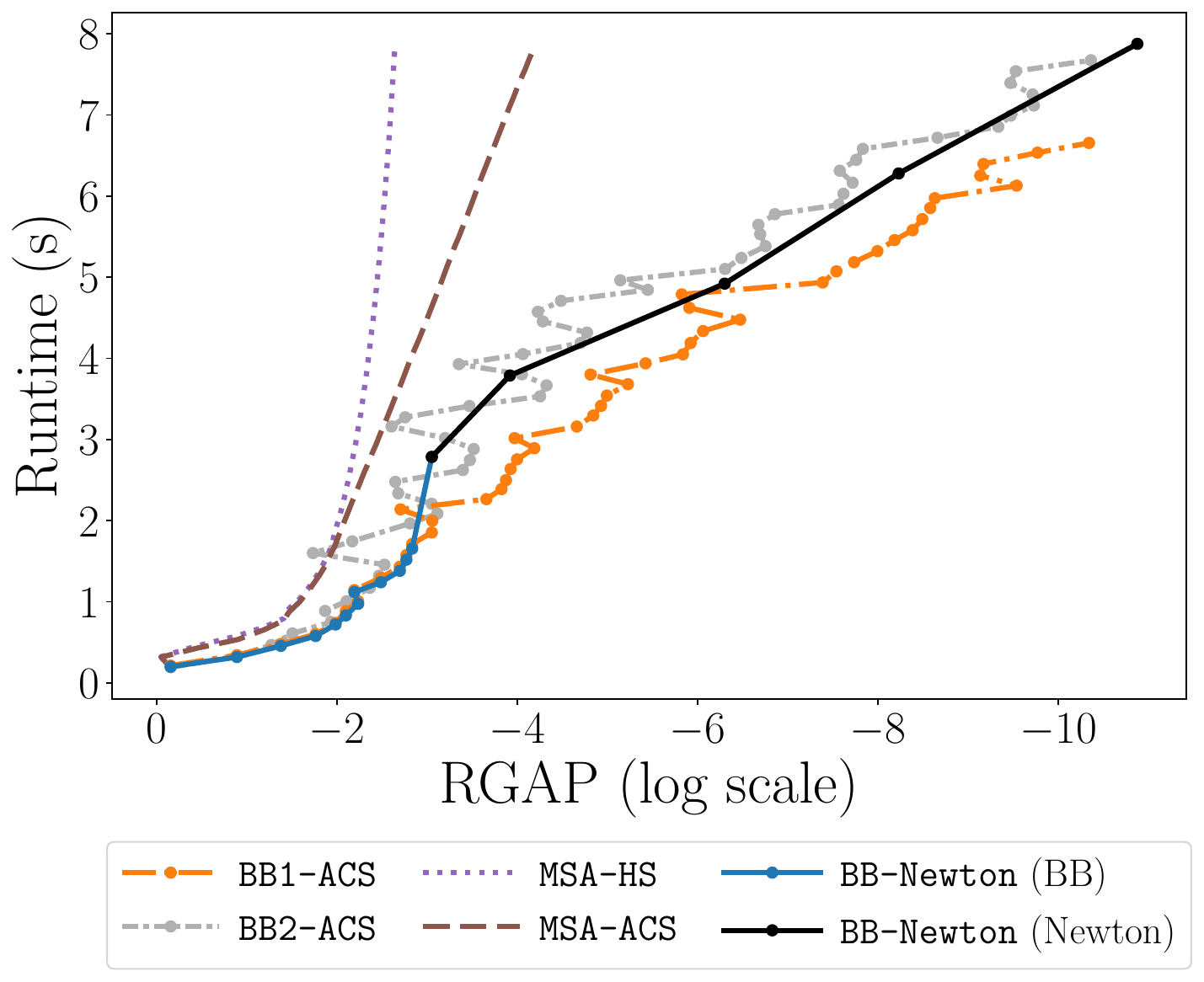}
        \caption{Chicago Sketch, 1x demand}
        \label{fig:runtime_grid_2}
    \end{subfigure}
    \hfill
    \begin{subfigure}[b]{0.45\textwidth}
        \centering
        \includegraphics[width=\textwidth]{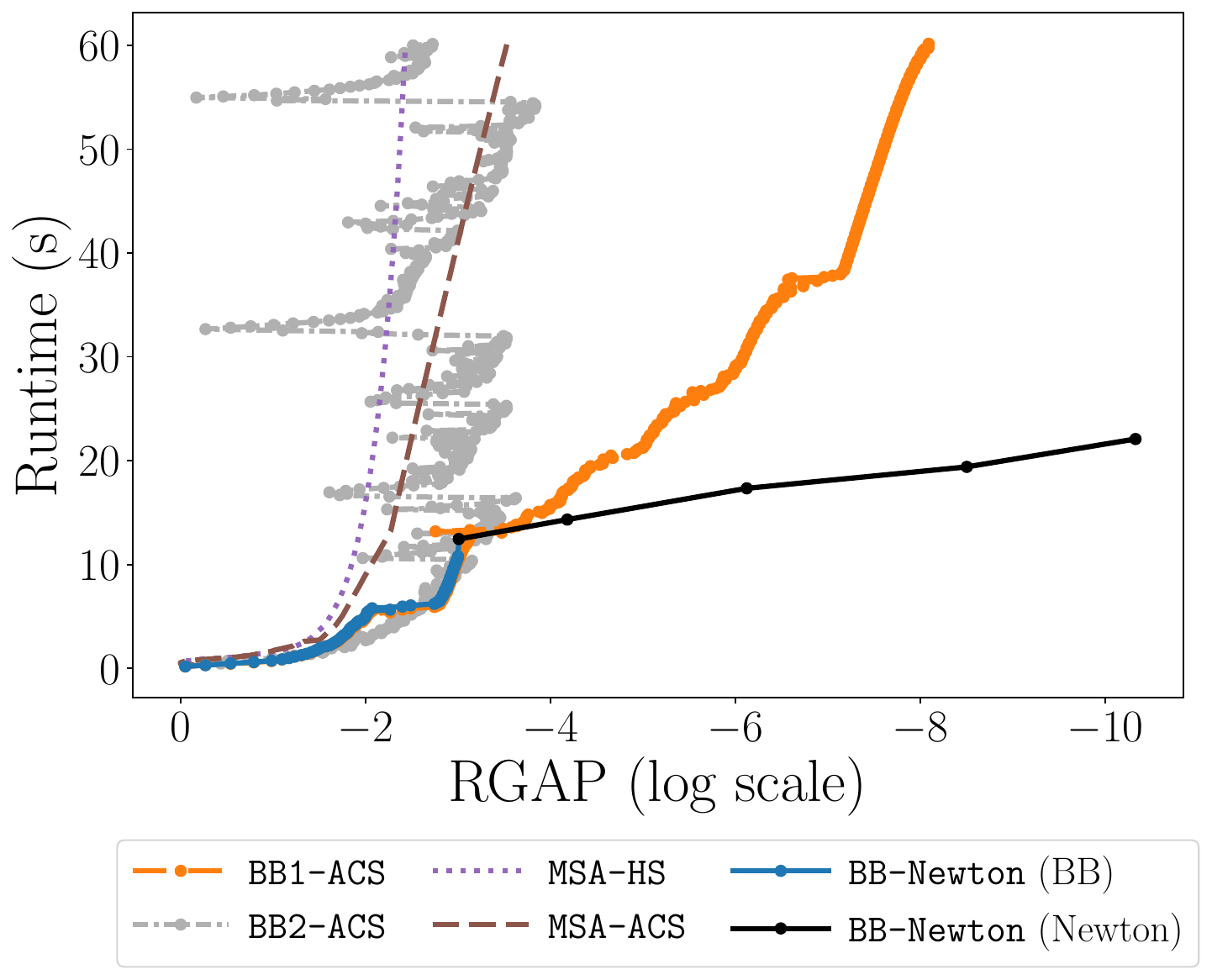}
        \caption{Chicago Sketch, 2x demand}
        \label{fig:runtime_grid_5}
    \end{subfigure}

    \medskip

    \begin{subfigure}[b]{0.45\textwidth}
        \centering
        \includegraphics[width=\textwidth]{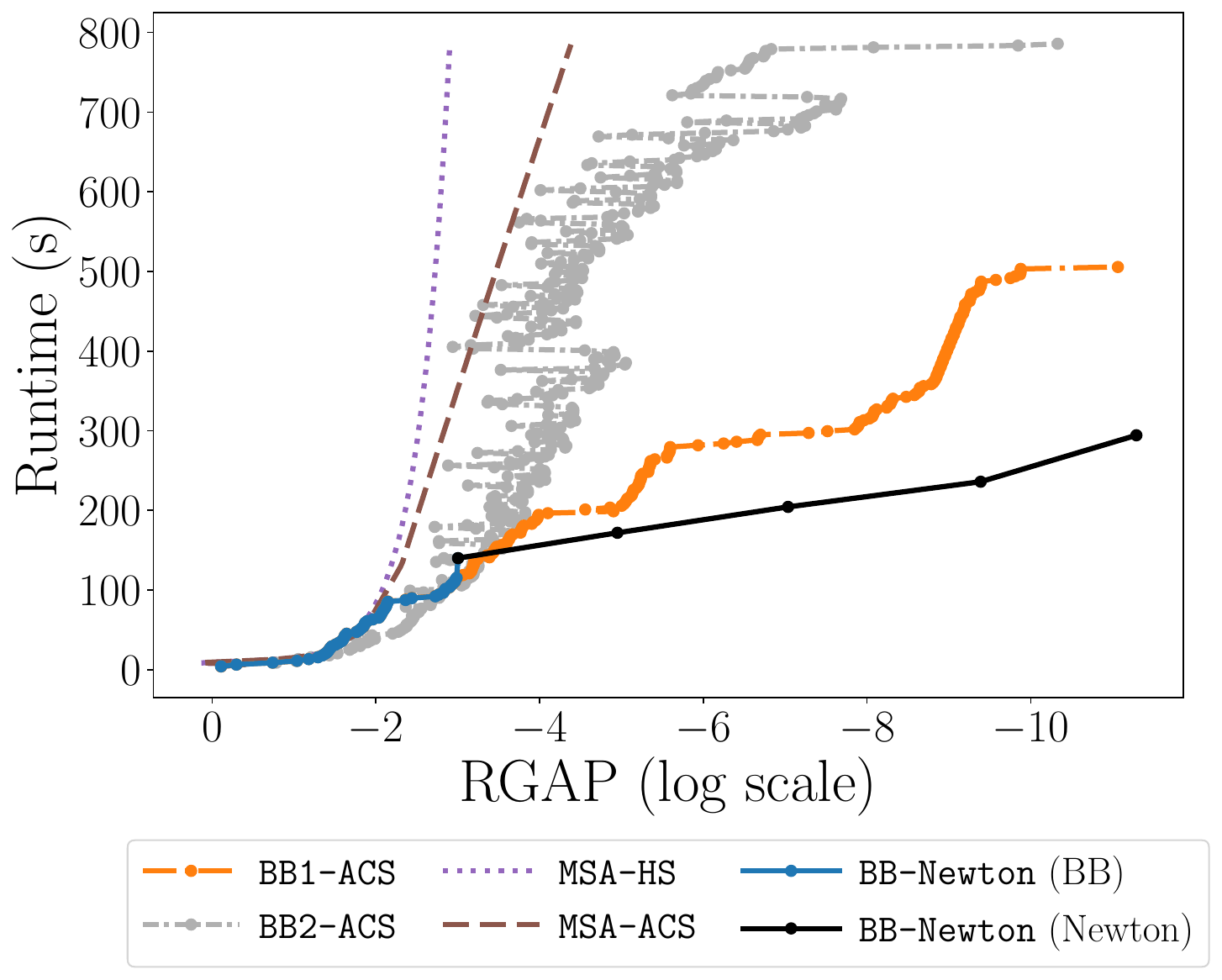}
        \caption{Chicago Regional, 1x demand}
        \label{fig:runtime_grid_3}
    \end{subfigure}
    \hfill
    \begin{subfigure}[b]{0.45\textwidth}
        \centering
        \includegraphics[width=\textwidth]{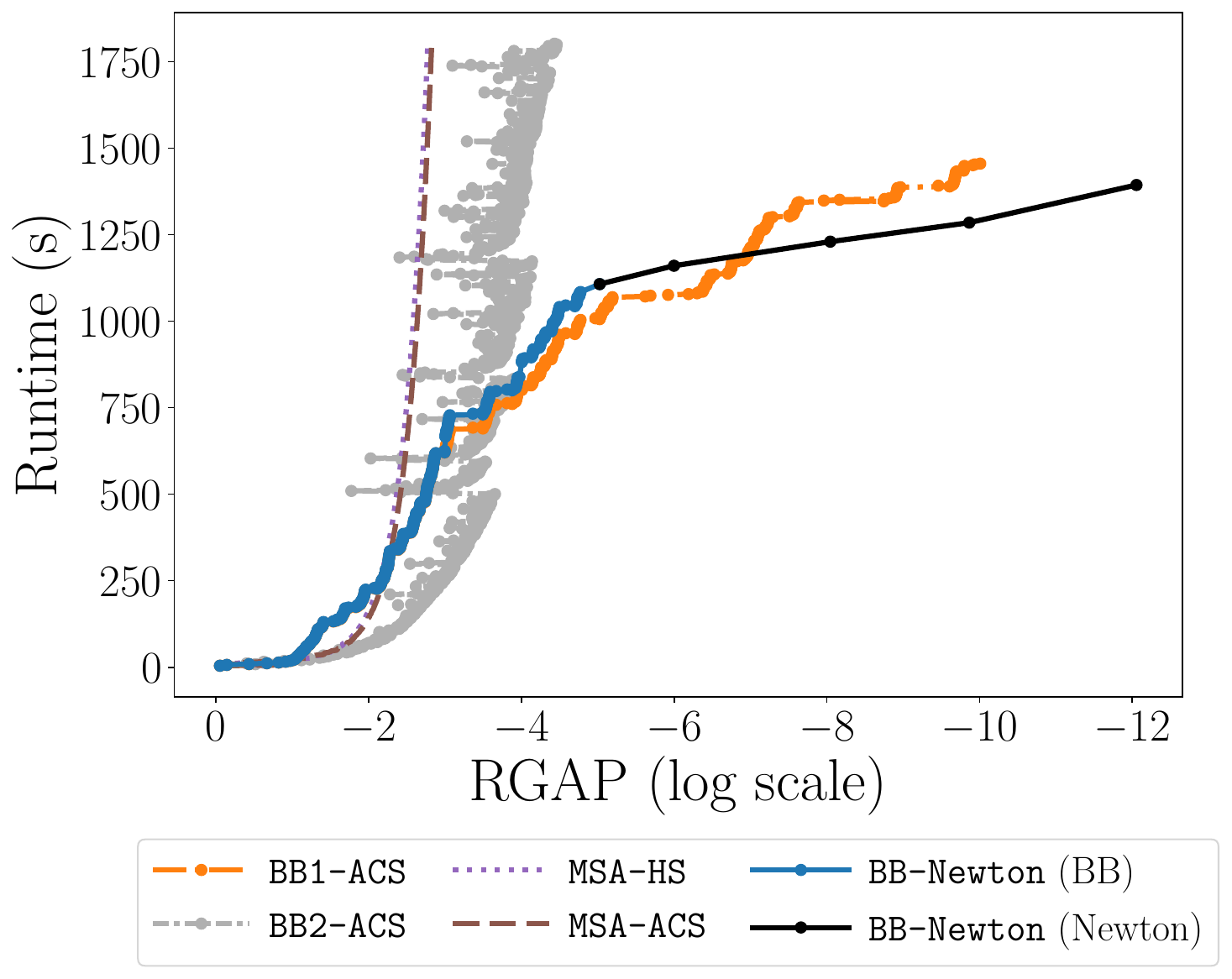}
        \caption{Chicago Regional, 2x demand}
        \label{fig:runtime_grid_6}
    \end{subfigure}
    \caption{Runtimes (in seconds) to reach a RGAP level (log scale) \texttt{BB-Newton, MSA-HS, MSA-ACS, BB1-ACS}, and \texttt{BB2-ACS} on Winnipeg Asymmetric, Chicago Sketch, and Chicago Regional networks at base (1$\times$) and doubled (2$\times$) demand levels. The $y$-axis is truncated at the last iteration of the non-MSA methods.}
    \label{fig:runtime_grid}
\end{figure}

For small-scale networks (Table \ref{tab:convergence_small}), \texttt{BB-Newton}, \texttt{BB1}, and \texttt{BB2} require similar run times of approximately 0.1 to 0.2 seconds under base demand. When the demand is doubled, the run times remain comparable, but \texttt{BB1} and \texttt{BB2} occasionally suffer from numerical issues and stall, as indicated by the asterisk marks for the Sioux Falls and BMC networks. The \texttt{BB1-ACS} and \texttt{BB2-ACS} rules resolve these numerical issues and provide similar computational performance. Even in these small networks, \texttt{MSA-HS} never reaches an RGAP of $10^{-10}$ within two seconds. In contrast, \texttt{MSA-ACS} performs remarkably well across all networks, providing run times similar to those of \texttt{BB-Newton} or the BB step-sizes.

For medium-scale networks (Table \ref{tab:convergence_medium}), a similar trend emerges. \texttt{BB-Newton} consistently converges within the one-minute budget and is either the fastest or competitive with other methods. Both \texttt{MSA-HS} and \texttt{MSA-ACS} struggle to reach the $10^{-10}$ RGAP threshold. The standard BB-based rules suffer from occasional numerical instability, especially when demand is high. The BB step-sizes with ACS fallbacks, \texttt{BB1-ACS} and \texttt{BB2-ACS}, converge reliably and match the performance of \texttt{BB-Newton} under base demand. When demand is doubled, however, both \texttt{BB1-ACS} and \texttt{BB2-ACS} fail to converge in most cases, leaving \texttt{BB-Newton} as the only rule that successfully converges within the one-minute budget.

Finally, we examine the large-scale networks in Table \ref{tab:convergence_large}. For the Philadelphia network, numerical issues occurs with \texttt{BB2} even under base demand. Using \texttt{BB2-ACS} resolves these issues and provides performance similar to \texttt{BB-Newton}. Both \texttt{BB1} and \texttt{BB1-ACS} converge in times slightly worse than \texttt{BB-Newton}. Under doubled demand, \texttt{BB-Newton}, \texttt{BB1}, and \texttt{BB1-ACS} converge within the 15-minute budget, while all other rules fail to reach the target gap. In the Chicago Regional network, \texttt{BB-Newton} performs significantly better than all other rules under base demand. When demand is doubled, no step-size rule converges within the 15-minute budget. With a 30-minute budget, only \texttt{BB-Newton} and \texttt{BB1-ACS} reach an RGAP of $10^{-10}$. Across the large networks, \texttt{BB-Newton} provides the fastest convergence to an RGAP of $10^{-10}$ in three out of four scenarios.

\begin{table}
\centering
\begin{tabular}{llrrrrrrr}
\toprule
\textbf{Network} & \textbf{Demand} & \texttt{BB-Newton} & \texttt{BB1} & \texttt{BB2} & \texttt{MSA-HS} & \texttt{MSA-ACS} & \texttt{BB1-ACS} & \texttt{BB2-ACS} \\\midrule
\multirow{4}{*}{Sioux Falls} & \multirow{2}{*}{1.0x} & 0.1 & 0.1 & 0.2 & -- & 0.6 & 0.4 & 0.2 \\
 &  & ($10^{-11}$) & ($10^{-11}$) & ($10^{-11}$) & ($10^{-3}$) & ($10^{-11}$) & ($10^{-11}$) & ($10^{-11}$) \\
\cmidrule(lr){2-9}
 & \multirow{2}{*}{2.0x} & 0.7 & \multirow{2}{*}{$\ast$} & \multirow{2}{*}{$\ast$} & -- & -- & 1.4 & 1.9 \\
 &  & ($10^{-11}$) &  &  & ($10^{-3}$) & ($10^{-6}$) & ($10^{-11}$) & ($10^{-11}$) \\
\midrule
\multirow{4}{*}{BMC} & \multirow{2}{*}{1.0x} & 0.1 & 0.2 & 0.1 & -- & 0.5 & 0.1 & 0.1 \\
 &  & ($10^{-14}$) & ($10^{-11}$) & ($10^{-11}$) & ($10^{-5}$) & ($10^{-11}$) & ($10^{-11}$) & ($10^{-11}$) \\
\cmidrule(lr){2-9}
 & \multirow{2}{*}{2.0x} & 0.4 & 1.0 & \multirow{2}{*}{$\ast$} & -- & -- & 0.9 & 0.9 \\
 &  & ($10^{-13}$) & ($10^{-11}$) &  & ($10^{-4}$) & ($10^{-7}$) & ($10^{-11}$) & ($10^{-11}$) \\
\midrule
\multirow{4}{*}{EMA} & \multirow{2}{*}{1.0x} & 0.0 & 0.0 & 0.1 & -- & 0.2 & 0.0 & 0.0 \\
 &  & ($10^{-11}$) & ($10^{-11}$) & ($10^{-11}$) & ($10^{-6}$) & ($10^{-11}$) & ($10^{-11}$) & ($10^{-11}$) \\
\cmidrule(lr){2-9}
 & \multirow{2}{*}{2.0x} & 0.1 & 0.1 & 0.1 & -- & 0.5 & 0.1 & 0.1 \\
 &  & ($10^{-13}$) & ($10^{-11}$) & ($10^{-11}$) & ($10^{-5}$) & ($10^{-11}$) & ($10^{-11}$) & ($10^{-11}$) \\
\midrule
\multirow{4}{*}{Anaheim} & \multirow{2}{*}{1.0x} & 0.1 & 0.1 & 0.1 & -- & 0.5 & 0.1 & 0.1 \\
 &  & ($10^{-11}$) & ($10^{-11}$) & ($10^{-11}$) & ($10^{-5}$) & ($10^{-11}$) & ($10^{-11}$) & ($10^{-11}$) \\
\cmidrule(lr){2-9}
 & \multirow{2}{*}{2.0x} & 0.2 & 0.2 & 0.2 & -- & 0.6 & 0.2 & 0.2 \\
 &  & ($10^{-11}$) & ($10^{-11}$) & ($10^{-11}$) & ($10^{-5}$) & ($10^{-11}$) & ($10^{-11}$) & ($10^{-11}$) \\
\bottomrule
\end{tabular}
\caption{Small networks (time limit 2 seconds): CPU time (in seconds) required to reach a relative gap of $10^{-10}$ under different step-size rules. For each scenario, the upper row reports the time and the lower row, in parentheses, reports the final gap reached at termination. A dash (--) in the time row indicates that the $10^{-10}$ gap was not reached within the time budget. An asterisk ($\ast$) indicates that a numerical error in step-size calculation caused the run to stop. Demand values of 1x and 2x correspond to the base and doubled OD demand, respectively.}
\label{tab:convergence_small}
\end{table}

\begin{table}
\centering
\begin{tabular}{llrrrrrrr}
\toprule
\textbf{Network} & \textbf{Demand}  & \texttt{BB-Newton} & \texttt{BB1} & \texttt{BB2} & \texttt{MSA-HS} & \texttt{MSA-ACS} & \texttt{BB1-ACS} & \texttt{BB2-ACS} \\\midrule
\multirow{4}{*}{Chicago Sketch} & \multirow{2}{*}{1.0x} & 7.9 & 6.6 & 7.6 & -- & 23.7 & 6.7 & 7.7 \\
 &  & ($10^{-11}$) & ($10^{-11}$) & ($10^{-11}$) & ($10^{-4}$) & ($10^{-11}$) & ($10^{-11}$) & ($10^{-11}$) \\
\cmidrule(lr){2-9}
 & \multirow{2}{*}{2.0x} & 22.1 & \multirow{2}{*}{$\ast$} & \multirow{2}{*}{$\ast$} & -- & -- & -- & -- \\
 &  & ($10^{-11}$) &  &  & ($10^{-3}$) & ($10^{-4}$) & ($10^{-9}$) & ($10^{-3}$) \\
\midrule
\multirow{4}{*}{Berlin Center} & \multirow{2}{*}{1.0x} & 9.4 & 10.1 & 9.2 & -- & -- & 10.3 & 9.2 \\
 &  & ($10^{-14}$) & ($10^{-11}$) & ($10^{-11}$) & ($10^{-6}$) & ($10^{-10}$) & ($10^{-11}$) & ($10^{-11}$) \\
\cmidrule(lr){2-9}
 & \multirow{2}{*}{2.0x} & 57.7 & -- & -- & -- & -- & -- & -- \\
 &  & ($10^{-12}$) & ($10^{-5}$) & ($10^{-7}$) & ($10^{-5}$) & ($10^{-6}$) & ($10^{-7}$) & ($10^{-7}$) \\
\midrule
\multirow{4}{*}{\makecell{Winnipeg \\ Asymmetric}} & \multirow{2}{*}{1.0x} & 0.7 & 1.1 & 3.9 & -- & 6.2 & 1.2 & 4.2 \\
 &  & ($10^{-11}$) & ($10^{-11}$) & ($10^{-11}$) & ($10^{-5}$) & ($10^{-11}$) & ($10^{-11}$) & ($10^{-11}$) \\
\cmidrule(lr){2-9}
 & \multirow{2}{*}{2.0x} & 1.1 & 2.4 & 3.8 & -- & 18.0 & 2.5 & 4.2 \\
 &  & ($10^{-12}$) & ($10^{-11}$) & ($10^{-11}$) & ($10^{-4}$) & ($10^{-11}$) & ($10^{-11}$) & ($10^{-11}$) \\
 \bottomrule
\end{tabular}
\caption{Medium networks (time limit 1 minute): CPU time (in seconds) required to reach a relative gap of $10^{-10}$ under different step-size rules. For each scenario, the upper row reports the time and the lower row, in parentheses, reports the final gap reached at termination. A dash (--) in the time row indicates that the $10^{-10}$ gap was not reached within the time budget. An asterisk ($\ast$) indicates that a numerical error in step-size calculation caused the run to stop. Demand values of 1x and 2x correspond to the base and doubled OD demand, respectively.}
\label{tab:convergence_medium}
\end{table}

\begin{table}
\centering
\begin{tabular}{llrrrrrrr}
\toprule
\textbf{Network} & \textbf{Demand} & \texttt{BB-Newton} & \texttt{BB1} & \texttt{BB2} & \texttt{MSA-HS} & \texttt{MSA-ACS} & \texttt{BB1-ACS} & \texttt{BB2-ACS} \\\midrule
\multirow{4}{*}{Philadelphia} & \multirow{2}{*}{1.0x} & 95.6 & 109.2 & \multirow{2}{*}{$\ast$} & -- & 775.7 & 105.5 & 126.2 \\
 &  & ($10^{-12}$) & ($10^{-11}$) &  & ($10^{-4}$) & ($10^{-11}$) & ($10^{-11}$) & ($10^{-11}$) \\
\cmidrule(lr){2-9}
 & \multirow{2}{*}{2.0x} & 500.1 & 481.7 & -- & -- & -- & 488.3 & -- \\
 &  & ($10^{-12}$) & ($10^{-11}$) & ($10^{-4}$) & ($10^{-3}$) & ($10^{-3}$) & ($10^{-11}$) & ($10^{-4}$) \\
\midrule
\multirow{4}{*}{Chicago Regional} & \multirow{2}{*}{1.0x} & 294.3 & \multirow{2}{*}{$\ast$} & 782.5 & -- & -- & 505.5 & 785.8 \\
 &  & ($10^{-12}$) &  & ($10^{-11}$) & ($10^{-3}$) & ($10^{-5}$) & ($10^{-12}$) & ($10^{-11}$) \\
\cmidrule(lr){2-9}
 & \multirow{2}{*}{2.0x} & 1393.5 & \multirow{2}{*}{$\ast$} & -- & -- & -- & 1455.2 & -- \\
 &  & ($10^{-13}$) &  & ($10^{-5}$) & ($10^{-3}$) & ($10^{-3}$) & ($10^{-11}$) & ($10^{-5}$) \\
\bottomrule
\end{tabular}
\caption{Large networks (time limit 15 minute): CPU time (in seconds) required to reach a relative gap of $10^{-10}$ under different step-size rules. For each scenario, the upper row reports the time and the lower row, in parentheses, reports the final gap reached at termination. A dash (--) in the time row indicates that the $10^{-10}$ gap was not reached within the time budget. An asterisk ($\ast$) indicates that a numerical error in step-size calculation caused the run to stop. Demand values of 1x and 2x correspond to the base and doubled OD demand, respectively. (No step-size rule reached $10^{-10}$ in 15 minutes in Chicago Regional network with 2x demand and hence a maximum of 30 minutes was provided.)}
\label{tab:convergence_large}
\end{table}

In summary, the following trends are observed regarding the computational performance of the seven tested methods:
\begin{samepage}
\begin{itemize}
    \item \texttt{BB-Newton} is the only step-size rule that consistently converged in every tested scenario within the respective computational budget. In most cases, it was either competitive or the fastest-performing rule.
    \item In small networks, all step-sizes except \texttt{MSA-HS} perform equally well. Notably, \texttt{MSA-ACS} performs exceptionally well, on par with the other rules. This is particularly interesting given that \texttt{MSA-ACS} utilizes only first-order information.
    \item In many cases, \texttt{BB1} and \texttt{BB2} suffer from numerical issues, especially when demand is doubled. However, when paired with the ACS fallback, the BB-based step-sizes often demonstrate run times similar to \texttt{BB-Newton} under base demand.
   \item In most large networks or under high demand, \texttt{BB-Newton} provides significantly faster convergence than other rules. 
\end{itemize}
\end{samepage}

\section{Conclusion}\label{sec:conc}
In this paper, we performed a spectral analysis of the Jacobian of the target function in logit-based SUE.  The main findings from this analysis were that MSA converges linearly with a rate $1 - s$ for small step-sizes $s$, in a neighborhood of equilibrium, and a bound on the most negative eigenvalue of $\mathcal{K}(\mb{h})$, which controls the convergence rate of MSA.  Motivated by these findings, we proposed a practical step-size selection rule (Algorithm~\ref{alg:4_stepsize}) that achieves asymptotic linear convergence. The rule can be used as a drop-in replacement in any MSA-based solver and performs significantly better than the standard MSA with step-sizes given by the harmonic series $\{1/k : k \geq 1\}$.

We further extended the spectral analysis of the Jacobian of the logit mapping to develop a quadratically convergent step-size for logit-based SUE. In particular, we show that using the Jacobian of the logit mapping to define step-sizes keeps the computations tractable, as we only deal with matrix vector products instead of building and inverting a full matrix. We show using Krylov subspace methods, such as GMRES with an appropriate tolerance sequence, can provide fast, superlinear convergence. We further showed that our step-sizes always map feasible path flow vectors to feasible vectors near equilibrium, obviating the need for manifold projection.

We benchmarked our results against the BB step-sizes from \cite{du2021faster}, and showed that when we start with BB step-sizes and then switch to the Newton step-size (\texttt{BB-Newton}), performance is similar in computational runtime when demand is low and networks are of moderate size; and faster  when demand is high or networks are of large scale. We test seven different methods across networks of different sizes under different demand levels. In our experiments, \texttt{BB-Newton} was the only step-size rule that converged in all tested experiments in the given time budget.

Several directions for future research follow naturally from this work. First, we restricted ourselves to acyclic path sets in our analysis. However, most of our analysis carries through for any finite path set. Some bounds would need to be adjusted, since the link-path incidence matrix is no longer binary when cyclic paths are allowed, but equivalent versions of our theorems should hold. Furthermore, the analysis of the constant step-size in MSA could be extended to a link-based formulation. Some initial experiments with link-based MSA show promising results, but this was not tested in greater detail and poses an excellent direction for future research. The constant step-size rule in MSA was applied here purely based on gap reduction. Other ``reset conditions'', such as ones that adapt to local curvature, could improve MSA performance further.

Second, broadening the analysis to non-logit SUE would be valuable. Much of our analysis depended on the Jacobian of the logit mapping, and determining the extent to which the conclusions apply to other distributions of the error term would be very valuable.
Third, one can investigate whether the Newton method can be globalized. We paired Newton with a globally convergent BB-based rule, but a direct globalization of the Newton step itself could improve the rate of convergence. This could include Newton's method variants for constrained optimization, where instead of taking the full Newton step, a partial step that ensures sufficient decrease in the objective function is used. However, such a step does not preserve demand feasibility, making this direction non-trivial. Finally, the Jacobian properties developed in this paper can be used beyond algorithm design. One direct application is deriving an error bound for any feasible path flow relative to the SUE solution. \cite{bagchi2025error} provides initial ideas in this direction with an OD-level approximation. Using the full Jacobian described here, their approach can be extended to obtain an exact error bound. 

\section*{Acknowledgments}
This work was partially funded through University Transportation Center National Center for Understanding Future Travel Behavior and Demand.

\section*{Declaration of generative AI and AI-assisted technologies in the writing process}
During the preparation of this work, the authors used Claude-Opus 4.7 in order to improve writing in select parts of the paper. After using this tool/service, the authors reviewed and edited the content as needed and take full responsibility for the content of the published article.

\section*{Author Contributions}

The authors confirm their contribution to the paper as follows:\\
Debojjal Bagchi: Conceptualization, Methodology, Formal Analysis, Validation, Writing (original draft), Writing (review \& editing).\\
Stephen D. Boyles: Methodology, Validation, Writing (review \& editing), Supervision.\\
All authors reviewed the results and approved the final version of the manuscript.

\bibliography{refs} 

\appendix

\section{Worked example on the Braess network}\label{app:braess}

To provide intuition for the results in Section~\ref{sec:3_Jacobian} and Section~\ref{sec:newton}, we work through a numerical example on the Braess network shown in Figure~\ref{fig:braess_network}. The network has one OD pair $(\texttt{O}, \texttt{D})$ with demand $d_{\texttt{OD}} = 6$ and three paths: $\texttt{O} \to \texttt{A} \to \texttt{D}$, $\texttt{O} \to \texttt{B} \to \texttt{D}$, and $\texttt{O} \to \texttt{A} \to \texttt{B} \to \texttt{D}$, with path flows $h_1$, $h_2$, and $h_3$, respectively. We set the dispersion parameter $\theta = 1.0$. The path costs are
\begin{equation*}
    c_1(\mb{h}) = (h_1 + h_3) + 5,\quad c_2(\mb{h}) = 5 + (h_2 + h_3),\quad c_3(\mb{h}) = (h_1 + h_3) + (h_2 + h_3).
\end{equation*}
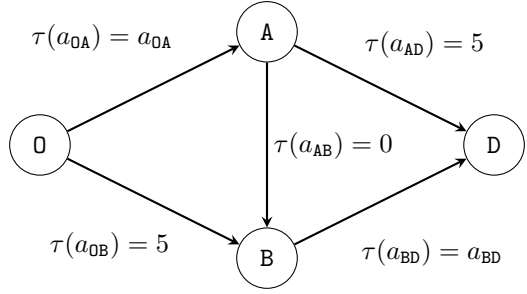
\begin{figure}[H]
    \centering
    \begin{tikzpicture}[
        node distance=2.5cm,
        every node/.style={circle, draw, minimum size=0.8cm, inner sep=0pt},
        every edge/.style={draw, thick, ->, >=stealth}
    ]
        \node (O) at (0, 0) {$\texttt{O}$};
        \node (A) at (3, 1.5) {$\texttt{A}$};
        \node (B) at (3, -1.5) {$\texttt{B}$};
        \node (D) at (6, 0) {$\texttt{D}$};
        \draw[->, >=stealth, thick] (O) -- (A) node[midway, above left, draw=none] {$\tau(a_{\texttt{OA}}) = a_{\texttt{OA}}$};
        \draw[->, >=stealth, thick] (O) -- (B) node[midway, below left, draw=none] {$\tau(a_{\texttt{OB}}) = 5$};
        \draw[->, >=stealth, thick] (A) -- (D) node[midway, above right, draw=none] {$\tau(a_{\texttt{AD}}) = 5$};
        \draw[->, >=stealth, thick] (B) -- (D) node[midway, below right, draw=none] {$\tau(a_{\texttt{BD}}) = a_{\texttt{BD}}$};
        \draw[->, >=stealth, thick] (A) -- (B) node[midway, right, draw=none] {$\,\tau(a_{\texttt{AB}}) = 0$};
    \end{tikzpicture}
    \caption{Braess network with one OD pair $(\texttt{O}, \texttt{D})$, three paths, and link cost functions shown on each link. Link $l$ connecting nodes $\texttt{A}$ and $\texttt{B}$ is denoted as $l_{\texttt{AB}}$.}
    \label{fig:braess_network}
\end{figure}

\subsection{Computations of the Jacobian}\label{app:braess_jacob}

The Jacobian of path costs with respect to path flows, as defined in Equation~\eqref{eq:3_codtderdef}, is
\begin{equation*}
    \mb{J} = \mathcal{C}'(\mb{h}) = \begin{pmatrix} 1 & 0 & 1 \\ 0 & 1 & 1 \\ 1 & 1 & 2 \end{pmatrix}.
\end{equation*}
We will use $\mb{J} := \mathcal{C}'(\mb{h})$ as a shorthand throughout this example, consistent with Section~\ref{subsec:3_evals}.

At iterate $\mb{h}^k = (2, 2, 2)^T$, the path costs are $\mb{c} = \mathcal{C}(\mb{h}^k) = (9, 9, 8)^T$. The logit path-choice probabilities are proportional to $(e^{-9}, e^{-9}, e^{-8})$, or equivalently $(1, 1, e)$, giving
\begin{equation*}
    \mb{p} = \mathcal{P}(\mb{c}) = \frac{1}{2 + e}(1,\ 1,\ e)^T \approx (0.21,\ 0.21,\ 0.58)^T.
\end{equation*}

The matrix $\mathcal{P}'(\mb{c})$ as shown in Equation~\eqref{eq:pprime} is
\begin{equation*}
    \mathcal{P}'(\mb{c}) \approx \begin{pmatrix} -\theta p_1(1-p_1) & \theta p_1 p_2 & \theta p_1 p_3 \\ \theta p_2 p_1 & -\theta p_2(1-p_2) & \theta p_2 p_3 \\ \theta p_3 p_1 & \theta p_3 p_2 & -\theta p_3(1-p_3) \end{pmatrix} \approx \begin{pmatrix} -0.17 & 0.04 & 0.12 \\ 0.04 & -0.17 & 0.12 \\ 0.12 & 0.12 & -0.24 \end{pmatrix}.
\end{equation*}

The matrix $\mathcal{H}'_{\mb{h}}(\mb{p}^k)$ has rows of constant value $p_i$:
\begin{equation*}
    \mathcal{H}'_{\mb{h}}(\mb{p}^k) \approx \begin{pmatrix} 0.21 & 0.21 & 0.21 \\ 0.21 & 0.21 & 0.21 \\ 0.58 & 0.58 & 0.58 \end{pmatrix}.
\end{equation*}

The matrix $\mathcal{H}'_{\mb{p}}(\mb{h}^k)$ is diagonal with entries equal to the OD demand $d_{\texttt{OD}} = 6$:
\begin{equation*}
    \mathcal{H}'_{\mb{p}}(\mb{h}^k)  = \begin{pmatrix} 6 & 0 & 0 \\ 0 & 6 & 0 \\ 0 & 0 & 6 \end{pmatrix}.
\end{equation*}

Following the decomposition in Equation~\eqref{eq:JacobianFull}, we compute each piece of $\mathcal{K}(\mb{h}^k) = \mathcal{H}'_{\mb{p}}(\mb{h}^k)\,\mathcal{P}'(\mb{c})\,\mathcal{C}'(\mb{h})$.
\begin{equation*}
    \mathcal{K}(\mb{h}^k) \approx \begin{pmatrix} 6 & 0 & 0 \\ 0 & 6 & 0 \\ 0 & 0 & 6 \end{pmatrix} \begin{pmatrix} -0.17 & 0.04 & 0.12 \\ 0.04 & -0.17 & 0.12 \\ 0.12 & 0.12 & -0.24 \end{pmatrix} \begin{pmatrix} 1 & 0 & 1 \\ 0 & 1 & 1 \\ 1 & 1 & 2 \end{pmatrix} \approx \begin{pmatrix} -0.27 & 1.00 & 0.73 \\ 1.00 & -0.27 & 0.73 \\ -0.73 & -0.73 & -1.47 \end{pmatrix}.
\end{equation*}

The matrix $\mb{S} := -\mathcal{H}'_{\mb{p}}(\mb{h}^k)\,\mathcal{P}'(\mb{c})$ defined in Section~\ref{subsec:3_evals} evaluates to
\begin{equation*}
    \mb{S} = d_{\texttt{OD}}\,\theta \bigl(\diag(\mb{p}) - \mb{p}\mb{p}^T\bigr) \approx \begin{pmatrix} 1.00 & -0.27 & -0.73 \\ -0.27 & 1.00 & -0.73 \\ -0.73 & -0.73 & 1.46 \end{pmatrix}.
\end{equation*}
We note that $\mb{1}^T \mb{S} \approx (0,\ 0,\ 0)$, illustrating Lemma~\ref{lem:onesS}. The reduced Jacobian is then
\begin{equation*}
    \mathcal{K}(\mb{h}^k) = -\mb{S}\,\mb{J} \approx -\begin{pmatrix} 1.00 & -0.27 & -0.73 \\ -0.27 & 1.00 & -0.73 \\ -0.73 & -0.73 & 1.46 \end{pmatrix} \begin{pmatrix} 1 & 0 & 1 \\ 0 & 1 & 1 \\ 1 & 1 & 2 \end{pmatrix} \approx \begin{pmatrix} -0.27 & 1.00 & 0.73 \\ 1.00 & -0.27 & 0.73 \\ -0.73 & -0.73 & -1.47 \end{pmatrix}.
\end{equation*}
We note that the decomposition of $\mathcal{K}(\mb{h})$ as $-\mb{S}\,\mb{J}$ allows it to be expressed as a product of two symmetric matrices $\mb{S}$ and $\mb{J}$, although $\mathcal{K}(\mb{h})$ need not necessarily be symmetric. Furthermore, the eigenvalues of $\mathcal{K}(\mb{h}^k)$ are approximately $\{0,\ -1.27,\ -0.73\}$, all real and non-positive, with zero as the maximum eigenvalue, illustrating Theorem~\ref{thm:spectrum}.

\subsection{The Newton step}\label{app:braess_newton}

The full Newton step requires
\begin{equation*}
    \mathcal{F}'(\mb{h}^k) = \mathcal{H}'_{\mb{h}}(\mb{p}^k) + \mathcal{K}(\mb{h}^k) - \mb{I} \approx \begin{pmatrix} -1.06 & 1.21 & 0.94 \\ 1.21 & -1.06 & 0.94 \\ -0.16 & -0.16 & -1.89 \end{pmatrix}.
\end{equation*}
The column sums are $(-0.01,\ -0.01,\ -0.01) \approx (0,\ 0,\ 0)$ (up to rounding), confirming $\mb{1}^T \mathcal{F}'(\mb{h}^k) = \mb{0}^T$ and hence the singularity of $\mathcal{F}'(\mb{h}^k)$.

Solving $\mathcal{F}'(\mb{h}^k)\,\mb{v} = \mb{0}$, subtracting row 1 from row 2 gives $2.27\,(v_1 - v_2) = 0$, so $v_1 = v_2$. Substituting into row 3 yields $-0.32\,v_1 - 1.89\,v_3 = 0$, so $v_3 \approx -0.17\,v_1$. The null space of $\mathcal{F}'(\mb{h}^k)$ is therefore the line
\begin{equation*}
    \bigl\{\, t\,(1,\ 1,\ -0.17)^T : t \in \mathbb{R} \,\bigr\}.
\end{equation*}

Inverting $\mb{I} - \mathcal{K}(\mb{h}^k)$ gives
\begin{equation*}
    (\mb{I} - \mathcal{K}(\mb{h}^k))^{-1} \approx \begin{pmatrix} 0.93 & 0.49 & 0.42 \\ 0.49 & 0.93 & 0.42 \\ -0.42 & -0.42 & 0.15 \end{pmatrix},
\end{equation*}
and the fixed point residual is
\begin{equation*}
\mathcal{F}(\mb{h}^k) = \mathcal{L}(\mb{h}^k) - \mb{h}^k = d_{\texttt{OD}}\,\mb{p} - \mb{h}^k \approx (-0.73,\ -0.73,\ 1.46)^T.
\end{equation*}
The reduced Newton step from Equation~\eqref{eq:4_reduced} is
\begin{equation*}
    \boldsymbol{\delta}^k = (\mb{I} - \mathcal{K}(\mb{h}^k))^{-1}\,\mathcal{F}(\mb{h}^k) \approx (-0.42,\ -0.42,\ 0.84)^T.
\end{equation*}
We observe that $\sum_i \delta_i^k \approx 0$, so $\boldsymbol{\delta}^k$ preserves demand, as expected from Theorem~\ref{thm:4_demand_unique}. Substituting $\boldsymbol{\delta}^k$ into the full Newton system we observe
\begin{equation*}
\mathcal{F}'(\mb{h}^k)\,\boldsymbol{\delta}^k \approx (0.73,\ 0.73,\ -1.46)^T = -\mathcal{F}(\mb{h}^k).
\end{equation*}
So, $\boldsymbol{\delta}^k$ also solves the full Newton system, illustrating Theorem~\ref{thm:4_reduced_solves_full}. The full solution set is the line
\begin{equation*}
    \boldsymbol{\hat{\delta}}^k(t) = \boldsymbol{\delta}^k + t\,(1,\ 1,\ -0.17)^T \approx (-0.42 + t,\ -0.42 + t,\ 0.84 - 0.17\,t)^T,
\end{equation*}
with OD sum
\begin{equation*}
    \sum_i \hat{\delta}_i^k(t) \approx 1.83\,t.
\end{equation*}
Demand preservation requires $\sum_i \hat{\delta}_i^k(t) = 0$, which forces $t = 0$. The only demand-preserving solution is $\boldsymbol{\delta}^k$, illustrating Theorem~\ref{thm:4_demand_unique}.

To provide further intuition for Theorem~\ref{thm:4_demand_unique}, we show that $\sum_i v_i \neq 0$ for nonzero null-space directions is not a coincidence. We recall that $\mb{v}$ lies in the null space of $\mathcal{F}'(\mb{h}^k) = \mathcal{H}'_{\mb{h}}(\mb{p}^k) + \mathcal{K}(\mb{h}^k) - \mb{I}$, that is,
\begin{equation*}
    \mathcal{H}'_{\mb{h}}(\mb{p}^k)\,\mb{v} + (\mathcal{K}(\mb{h}^k) - \mb{I})\,\mb{v} = \mb{0}.
\end{equation*}
We are interested in nonzero null-space directions $\mb{v} \neq \mb{0}$ (equivalently, $t \neq 0$), since these correspond to solutions of the full Newton system distinct from $\boldsymbol{\delta}^k$. Since $\mb{I} - \mathcal{K}(\mb{h}^k)$ is invertible, its null space contains only $\mb{0}$. So $\mb{v}$ lies in the null space of $\mathcal{F}'(\mb{h}^k)$ but \emph{not} in the null space of $\mb{I} - \mathcal{K}(\mb{h}^k)$. The contribution from $\mathcal{H}'_{\mb{h}}(\mb{p}^k)\,\mb{v}$ must therefore be nonzero to balance $(\mathcal{K}(\mb{h}^k) - \mb{I})\,\mb{v}$.

Computing $\mathcal{H}'_{\mb{h}}(\mb{p}^k)\,\mb{v}$ component-wise:
\begin{equation*}
    \mathcal{H}'_{\mb{h}}(\mb{p}^k)\,\mb{v} \approx \begin{pmatrix} 0.21 & 0.21 & 0.21 \\ 0.21 & 0.21 & 0.21 \\ 0.58 & 0.58 & 0.58 \end{pmatrix}\begin{pmatrix} v_1 \\ v_2 \\ v_3 \end{pmatrix} = (v_1 + v_2 + v_3)\begin{pmatrix} 0.21 \\ 0.21 \\ 0.58 \end{pmatrix},
\end{equation*}
which is nonzero if and only if $v_1 + v_2 + v_3 \neq 0$. This will always be the case because each row of $\mathcal{H}'_{\mb{h}}(\mb{p}^k)$ is same and strictly positive. Hence a nonzero null-space direction $\mb{v}$ of $\mathcal{F}'(\mb{h}^k)$ \emph{cannot} satisfy $v_1 + v_2 + v_3 = 0$. This is exactly what we observed: $v_1 + v_2 + v_3 \approx 1.83\,t \neq 0$ for $t \neq 0$. 

To complete the example, we take the Newton step and verify the size of the residual after one update. The new iterate is
\begin{equation*}
    \mb{h}^{k+1} = \mb{h}^k + \boldsymbol{\delta}^k = (2,\ 2,\ 2)^T + (-0.42,\ -0.42,\ 0.84)^T \approx (1.58,\ 1.58,\ 2.84)^T,
\end{equation*}
which sums to $6 = d_{\texttt{OD}}$, confirming demand preservation. Recomputing path costs and logit probabilities at $\mb{h}^{k+1}$ gives $\mb{c}(\mb{h}^{k+1}) \approx (9.42,\ 9.42,\ 8.84)^T$ and $\mb{p}^{k+1} \approx (0.264,\ 0.264,\ 0.472)^T$, so
\begin{equation*}
    \mathcal{F}(\mb{h}^{k+1}) = d_{\texttt{OD}}\,\mb{p}^{k+1} - \mb{h}^{k+1} \approx (0.005,\ 0.005,\ -0.010)^T.
\end{equation*}
This gives $\left\|\mathcal{F}(\mb{h}^{k+1})\right\| \approx 0.01$ from $\left\|\mathcal{F}(\mb{h}^k)\right\| \approx 1.79$ in one step, illustrating the rapid local convergence of the Newton step established in Section~\ref{subsec:4_convergence}.
\end{document}

%% file: include/header.tex
\usepackage{algorithm}
\usepackage{amsmath}
\usepackage{amsthm}
\usepackage{amssymb}
\usepackage{caption}
\usepackage{array}
\usepackage{bm}
\usepackage{enumerate}
\usepackage{fancyhdr}
\usepackage[pdftex]{graphicx}
\usepackage{multirow}
\usepackage{nag}
\usepackage{natbib}
\usepackage{verbatim}
\usepackage[colorlinks]{hyperref}
\usepackage{url}
\usepackage[dvipsnames]{xcolor}
\usepackage[normalem]{ulem}
\usepackage{breqn}
\usepackage{enumitem}
\usepackage{makecell}

\hypersetup{bookmarks=true, citecolor=blue, urlcolor=blue, linkcolor=blue}

\usepackage{titlecaps}
\Addlcwords{an and the}
\usepackage[utf8]{inputenc}
\usepackage{amsmath}
\usepackage{amsfonts}
\usepackage{amssymb}
\usepackage{algorithm}
\usepackage[noend]{algpseudocode}
\usepackage{multirow}
\usepackage{float}
\usepackage{array}
\usepackage{lipsum}
\usepackage[left=0.8in,right=0.8in,top=0.8in,bottom=0.8in]{geometry}
\usepackage{comment}
\usepackage{graphicx}
\usepackage{url}
\usepackage{color}
\usepackage{pgfplots}
\usepackage{mathtools}
\usepackage{caption}
\usepackage{subcaption}
\usepackage{pdflscape}
\usepackage{booktabs} 
\usepackage{enumerate}
\usepackage{authblk}
\usepackage{blkarray}
\usepackage{bm}
\usepackage{marvosym} 
\usepackage{soul}
\usepackage{textcomp, xspace}
\usepackage{lineno}
\usepackage{amssymb}  
\usepackage{dsfont} 
\usepgfplotslibrary{dateplot}
\usetikzlibrary{patterns}

\newtheorem{dfn}{Definition}

\newtheorem{lem}{Lemma}
\newtheorem{prp}{Proposition}
\newtheorem{thm}{Theorem}

\definecolor{greenxllite}{RGB}{196,215,155}
\definecolor{bluexllite}{RGB}{149,179,215}
\definecolor{redxllite}{RGB}{218,150,148}
\definecolor{greenxl}{RGB}{155, 187, 89}
\definecolor{bluexl}{RGB}{79,129,189}
\definecolor{redxl}{RGB}{192, 80, 77}
\definecolor{OliveGreen}{RGB}{70,136,52}

\providecommand{\Keywords}[1]{\\\textbf{\textit{Keywords:}} #1}

%% file: include/symbols.tex
\newcommand{\pdr}[2]{\frac{\partial #1}{\partial #2}} 

\newcommand{\Set}[1]{\mathbb{#1}}

\newcommand{\mb}[1]{\mathbf{ #1}}

\newcommand{\diag}{\operatorname{diag}}

\usepackage{xcolor}